\documentclass[10pt]{amsart}%
\usepackage[utf8x]{inputenc}
\usepackage{amsmath,amssymb}
\usepackage{fullpage}
\usepackage{amsthm}
\usepackage{multirow}
\usepackage{color}
\usepackage{enumerate}
\usepackage{tikz}
\usepackage{ytableau}
\usepackage{pifont}
\usepackage{amsmath}
\usepackage{graphicx}
  \usepackage[all]{xy}
\usepackage{amsfonts}
\usepackage{amssymb}
\usepackage{ytableau}%

\providecommand{\U}[1]{\protect\rule{.1in}{.1in}}
\providecommand{\U}[1]{\protect\rule{.1in}{.1in}}
\providecommand{\U}[1]{\protect\rule{.1in}{.1in}}
\providecommand{\U}[1]{\protect\rule{.1in}{.1in}}
\providecommand{\U}[1]{\protect\rule{.1in}{.1in}}
\newcommand{\ulambda}{{\boldsymbol{\lambda}}}

\newcommand{\umu}{{\boldsymbol{\mu}}}
\newcommand{\unu}{{\boldsymbol{\nu}}}

\newcommand{\uemptyset }{{\boldsymbol{\emptyset}}}
\newtheorem{Th}{Theorem}[section]
\numberwithin{equation}{section}

\newtheorem{Prop}[Th]{Proposition}

\theoremstyle{remark}
\newtheorem{Rem}[Th]{Remark}{\rmfamily}
\theoremstyle{definition}
\newtheorem{Def}[Th]{Definition}{\rmfamily}
\newtheorem{exa}[Th]{Example}{\rmfamily}

\def\Z{\mathbb{Z}}

\def\Z{{\mathbb Z}}

\newcommand\blfootnote[1]{%
  \begingroup
  \renewcommand\thefootnote{}\footnote{#1}%
  \addtocounter{footnote}{-1}%
  \endgroup
}

\begin{document}

\title{Blocks of Ariki-Koike algebras and level-rank duality}

\author{David Declercq}
\address{Universit\'{e} de Reims Champagne-Ardennes, UFR Sciences exactes et
naturelles. Laboratoire de Math\'{e}matiques UMR CNRS 9008. Moulin de la Housse BP
1039. 51100 Reims. France.}
\email{david.declercq@univ-reims.fr} 

\author{Nicolas Jacon}
\address{Universit\'{e} de Reims Champagne-Ardennes, UFR Sciences exactes et
naturelles. Laboratoire de Math\'{e}matiques UMR CNRS 9008. Moulin de la Housse BP
1039. 51100 Reims. France.}
\email{nicolas.jacon@univ-reims.fr}

\maketitle
\date{}
\blfootnote{\textup{2020} \textit{Mathematics Subject Classification}: \textup{20C08,05E10,20C20}} 
\begin{abstract}
We study the  blocks for Ariki-Koike algebras using a general notion of core for $l$-partitions. We interpret  the action of the affine symmetric group on the blocks in the context of level rank duality and study the orbits under this action. 
\end{abstract}


\section{Introduction}

One of the main (and still open) problem in the  modular representation of the symmetric group  $\mathfrak{S}_n$ (that is over a field $k$ of characteristic $p>0$) 
 is the determination of the decomposition matrix. This matrix describes the decomposition of certain 
  remarkable modules, the Specht modules, into simples via a process of modular reduction. The rows are indexed by the partitions of 
   $n$ (which naturally index the Specht modules)  where as the columns  are indexed by  certain  partitions, named the  $p$-regular partitions (which index the simple modules of $k\mathfrak{S}_n$.)  The decomposition  matrix is in fact a block diagonal matrix and we may try to study these ``blocks'' independently one from  the other.  This leads to fundamental problems which have both combinatorial and geometric flavors:
   \begin{itemize}
   \item The distributions of the  partitions (and the $p$-regular partitions)  into blocks may be described  using the notion of $p$-core and $p$-quotient of a partition. The complexity of the block is measured by  a combinatorial datum, the {weight} of the block.
   
   \item There is a strong relation between some of the blocks by the works and Scopes \cite{Sc}  and Chuang-Rouquier \cite{CR}.  In particular, an action of the affine symmetric group on the set of blocks induces derived equivalences between the categories  associated to these blocks. In some cases, these equivalences are Morita equivalences and the associated block decomposition matrices are just the same.     
   \end{itemize}
   
   This paper is concerned with  the Ariki-Koike algebra which may be seen as a generalization of the group algebra of the symmetric group. It is in fact a deformation of the complex reflection group of type $G(l,1,n)$.  This algebra has  been intensively studied during the past decades (see \cite{Ar, GJ}).  In particular, we still have a notion of  decomposition matrix in this case.  The analogues of the Specht modules are 
    this time indexed by $l$-tuple of partitions (``$l$-partitions''), and the set of $e$-regular partitions by the so called Uglov $l$-partitions. A generalized version  of the notion of weight has been provided  by Fayers in \cite{Fa} and the distribution into blocks has been studied by Lyle and Mathas in \cite{LM}. 
    More recently, a notion of Scopes equivalence has been studied by  Dell'Arciprete
\cite{AD}, Lyle \cite{Ly}, Webster \cite{We}  (in a more general context) and in \cite{LiT}.   In \cite{JL}, the second author together with C.Lecouvey have introduced a generalization of the notion of $p$-core partitions in this context and show that some of the results known in the case of the modular representation theory of the symmetric group are still verified in the context of Ariki-Koike algebras.

The aim of this paper is to  go deeper in the study of the notion of blocks using this generalized notion of core.  We describe precisely the action of the affine symmetric group and we develop the combinatorics of blocks  using this notion.  In particular, generalizing works by Gerber and Lecouvey \cite{LG}, we interpret the induced action on the Uglov $l$-partitions in the context of level-rank duality studied by Uglov \cite{U} and Yvonne \cite{Y}, and show that this action may be obtained using the notion of crystal isomorphism as studied in \cite{JL}.   Hence, the equivalence between two blocks can be interpreted as a crystal isomorphism between two Fock spaces associated with two different multicharges. 
    We also show that the notion of Scopes equivalence can be easily settled  in this context using the notion of core.   Contrary to the case of the symmetric group, two blocks with the same weight may not be in the same orbit. We then study  and classify the different orbits under the action.  We will see that the generalized notion of core nicely simplifies the different classical properties of blocks.
    
    The paper will be organized as follows. First, we describe the combinatorial notions which are needed in this paper: this includes  the notion of partition, symbol, abacus, core and quotient (in the case of the symmetric group).  The third part 
     first generalizes this notion in the context of Ariki-Koike algebras. We in particular introduce the concept of level rank duality. This section contains several new results on the generalized cores (see Proposition \ref{rank} and \ref{core}). In the fourth and fifth part, we study the blocks of Ariki-Koike algebras and the action of the affine Weyl group together with  the analogue of Scopes equivalence (see Proposition \ref{scopes}) and we interpret this action using the level-rank duality in Theorem \ref{duality}.  The last part is finally devoted to the classification of the orbits under this action.

\section{Partitions, symbols and abaci}
In this part, we introduce the combinatorial notions needed to deal with the modular representation theory of Ariki-Koike algebras and to introduce the notion of blocks. 

\subsection{Partitions} By definition, a {\it partition} of $n$  is a sequence of non increasing positive integers $(\lambda_1,\ldots,\lambda_r)$ of total sum $n$.  
We denote by:
\begin{itemize}
\item  $\Pi^l$ the set of all $l$-partitions, that is, the set of all $l$-tuples $(\lambda^{0},\ldots,\lambda^{l-1})$ of partitions.
\item $\Pi=\Pi^1$ the set of all partitions.
\end{itemize}
For any subset $E$ of $\Pi^l$ and any $n\in\Z_{\geq0}$,
we denote $E(n)$ the set of $l$-partitions in $E$  with total size $n$.

\subsection{Abaci and symbols}    Let $m\in \mathbb{Z}$. 
By definition,  a {\it $\beta$-set} or {\it symbol}  of {\it charge}  $m\in \mathbb{Z}$  is an infinite sequence of integers 
 $X=(\beta_i)_{i<m}$ such that:
 \begin{enumerate}
 \item For all $i< m$, we have $\beta_{i-1}<\beta_i$ (that is $X$ is a strictly increasing sequence)
 \item There exists $N< m$ such that for all $j\leq N$, we have $\beta_j=j$.
 \end{enumerate}
 We denote by $\operatorname{Symb} (m)$ the set of symbols of charge $m$.
   If $m\in \mathbb{Z}$ then the symbol  $X=(\beta_i)_{i<m}$ such that $\beta_i=i$ for all $i<m$ is called the {\it trivial symbol} of charge $m$.
 
If $X\in \operatorname{Symb} (m)$ and if $r\in \mathbb{Z}$, we write $X[r]$ for the symbol 
 $(\gamma)_{j< r+m}\in \operatorname{Symb} (m+r) $  such that for all $j<r+m$, we have $\gamma_j=\beta_{j-r}+r$.  We write 
 $$X\simeq X[r].$$
In this way, we obtain  an equivalence relation on the set of symbols and each equivalence class has at  most  one element 
  with a fixed charge.

A symbol may be conveniently represented using its abacus configuration. We associate to a symbol $X$
  an horizontal  runner 
  full  of (an infinite number of ) beads numbered in $\mathbb{Z}$. A bead numbered by $a\in \mathbb{Z}$  is colored in black if and only if $a\in X$. The others are written in white. 
\begin{exa}
The abacus associated to $X=(\ldots,-3,-2,-1,3,5,6)$ is
        \begin{center}
\begin{tikzpicture}[scale=0.5, bb/.style={draw,circle,fill,minimum size=2.5mm,inner sep=0pt,outer sep=0pt}, wb/.style={draw,circle,fill=white,minimum size=2.5mm,inner sep=0pt,outer sep=0pt}]
	
	\node [] at (11,-1) {18};
	\node [] at (10,-1) {17};
	\node [] at (9,-1) {16};
	\node [] at (8,-1) {15};
	\node [] at (7,-1) {14};
	\node [] at (6,-1) {13};
	\node [] at (5,-1) {12};
	\node [] at (4,-1) {11};
	\node [] at (3,-1) {10};
	\node [] at (2,-1) {9};
	\node [] at (1,-1) {8};
	\node [] at (0,-1) {7};
	\node [] at (-1,-1) {6};
	\node [] at (-2,-1) {5};
	\node [] at (-3,-1) {4};
	\node [] at (-4,-1) {3};
	\node [] at (-5,-1) {2};
	\node [] at (-6,-1) {1};
	\node [] at (-7,-1) {0};
	\node [] at (-8,-1) {-1};
	\node [] at (-9,-1) {-2};

	\node [wb] at (11,0) {};
	\node [wb] at (10,0) {};
	\node [wb] at (9,0) {};
	\node [wb] at (8,0) {};
	\node [wb] at (7,0) {};
	\node [wb] at (6,0) {};
	\node [wb] at (5,0) {};
	\node [wb] at (4,0) {};
	\node [wb] at (3,0) {};
	\node [wb] at (2,0) {};
	\node [wb] at (1,0) {};
	\node [wb] at (0,0) {};
	\node [bb] at (-1,0) {};
	\node [bb] at (-2,0) {};
	\node [wb] at (-3,0) {};
	\node [bb] at (-4,0) {};
	\node [wb] at (-5,0) {};
	\node [wb] at (-6,0) {};
	\node [wb] at (-7,0) {};
	\node [bb] at (-8,0) {};
	\node [bb] at (-9,0) {};
		\draw (-10,0) node[]{$\ldots$};
	\end{tikzpicture}
\end{center}

\end{exa}
The charge of a symbol  $X$ may be conveniently read in the associated abacus as follows. For each black bead, if there exists a white bead at its left then replace the leftmost one with a black bead and the black bead itself with a white bead. We obtain the abacus of a trivial symbol. The charge of  $X$ is then the charge of the trivial symbol, that is  the number associated with the leftmost white bead. 
\begin{exa}
Keeping the above example, if we perform the above manipulation, we obtain 
        \begin{center}
\begin{tikzpicture}[scale=0.5, bb/.style={draw,circle,fill,minimum size=2.5mm,inner sep=0pt,outer sep=0pt}, wb/.style={draw,circle,fill=white,minimum size=2.5mm,inner sep=0pt,outer sep=0pt}]
	
	\node [] at (11,-1) {18};
	\node [] at (10,-1) {17};
	\node [] at (9,-1) {16};
	\node [] at (8,-1) {15};
	\node [] at (7,-1) {14};
	\node [] at (6,-1) {13};
	\node [] at (5,-1) {12};
	\node [] at (4,-1) {11};
	\node [] at (3,-1) {10};
	\node [] at (2,-1) {9};
	\node [] at (1,-1) {8};
	\node [] at (0,-1) {7};
	\node [] at (-1,-1) {6};
	\node [] at (-2,-1) {5};
	\node [] at (-3,-1) {4};
	\node [] at (-4,-1) {3};
	\node [] at (-5,-1) {2};
	\node [] at (-6,-1) {1};
	\node [] at (-7,-1) {0};
	\node [] at (-8,-1) {-1};
	\node [] at (-9,-1) {-2};

	\node [wb] at (11,0) {};
	\node [wb] at (10,0) {};
	\node [wb] at (9,0) {};
	\node [wb] at (8,0) {};
	\node [wb] at (7,0) {};
	\node [wb] at (6,0) {};
	\node [wb] at (5,0) {};
	\node [wb] at (4,0) {};
	\node [wb] at (3,0) {};
	\node [wb] at (2,0) {};
	\node [wb] at (1,0) {};
	\node [wb] at (0,0) {};
	\node [wb] at (-1,0) {};
	\node [wb] at (-2,0) {};
	\node [wb] at (-3,0) {};
	\node [wb] at (-4,0) {};
	\node [bb] at (-5,0) {};
	\node [bb] at (-6,0) {};
	\node [bb] at (-7,0) {};
	\node [bb] at (-8,0) {};
	\node [bb] at (-9,0) {};
		\draw (-10,0) node[]{$\ldots$};
	\end{tikzpicture}
\end{center}
which is the abacus of the trivial symbol of charge $3$. 

\end{exa}

\subsection{$l$-Symbol}

An {\it $l$-symbol} is a collection of $l$  symbols:
 $${\bf X}=(X^0,\ldots,X^{l-1})$$
The {\it multicharge} (or $l$-charge) of the symbol is the $l$-tuple $(m_0,\ldots,m_{l-1})\in \mathbb{Z}^l$ where for all $j=0,\ldots,l-1$, the number $m_j$ is the charge of $X^j =(\beta^j_i)_{i< m_j}$.  It will sometimes be  convenient in the following to write such a symbol from bottom to top, taking into account the  gaps between the charges of each symbol  (see the example below).  $${\bf X}=\left(\begin{array}{c}
 X^{l-1}\\
 \vdots\\
 X^0
 \end{array} \right)$$
If ${\bf X}$ is an $l$-symbol and if $r\in \mathbb{N}$, we write ${\bf X}[r]$ for the symbol 
  $(X^0[r],\ldots,X^{l-1}[r])$. Again,  we write ${\bf X}\simeq {\bf X}[r]$. 
 
\begin{exa}
Let us consider the $3$-symbol: 
$${\bf X}=((\ldots,-1,0,2,4,6),(,\ldots,-1,0,3,4),(\ldots,-1,0,2,5)).$$
 Then we can write it as follows:
 $${\bf X}=\left(\begin{array}{cccccc}
\ldots & -1 & 0 & 2 & 5 \\
 \ldots & -1 & 0 & 3 & 4 \\
  \ldots & -1 & 0 & 2 & 4 & 6 
 \end{array} \right)$$
Then  the associated multicharge is $(4,3,3)$. 
The $3$-symbol 
$${\bf Y}=((\ldots,0,1,2,3,5,7,9),(\ldots,0,1,2,3,6,7),(\ldots,0,1,2,3,5,8))$$
 has multicharge $(7,6,6)$ and we have ${\bf Y}={\bf X} [3]$ so that 
${\bf X} \simeq {\bf Y}$. 

\end{exa}

An {\it $l$-symbol} ${\bf X}=(X^0,\ldots,X^{l-1})$ can be conveniently represented using its abacus configuration. In this way, we associate to each $X^j$ from $j=0$ to $l-1$   an abacus as above and we  write them  from bottom to top so that the beads in the same column are numbered by the same integer.   We call the associated object an {\it $l$-abacus}.
  \begin{exa}\label{firstabacus}
  Let $l=3$ and let us consider the following $3$-symbol:
  $$X=((\ldots,-1,0,2,4,6),(\ldots,-1,0,3,4),(\ldots,-1,0,2,5)).$$
  The associated $3$-abacus is:
\begin{center}
\begin{tikzpicture}[scale=0.5, bb/.style={draw,circle,fill,minimum size=2.5mm,inner sep=0pt,outer sep=0pt}, wb/.style={draw,circle,fill=white,minimum size=2.5mm,inner sep=0pt,outer sep=0pt}]
	
	\node [] at (11,-1) {20};
	\node [] at (10,-1) {19};
	\node [] at (9,-1) {18};
	\node [] at (8,-1) {17};
	\node [] at (7,-1) {16};
	\node [] at (6,-1) {15};
	\node [] at (5,-1) {14};
	\node [] at (4,-1) {13};
	\node [] at (3,-1) {12};
	\node [] at (2,-1) {11};
	\node [] at (1,-1) {10};
	\node [] at (0,-1) {9};
	\node [] at (-1,-1) {8};
	\node [] at (-2,-1) {7};
	\node [] at (-3,-1) {6};
	\node [] at (-4,-1) {5};
	\node [] at (-5,-1) {4};
	\node [] at (-6,-1) {3};
	\node [] at (-7,-1) {2};
	\node [] at (-8,-1) {1};
	\node [] at (-9,-1) {0};
	\draw (-10,-1) node[]{$\ldots$};
	
	\node [wb] at (11,0) {};
	\node [wb] at (10,0) {};
	\node [wb] at (9,0) {};
	\node [wb] at (8,0) {};
	\node [wb] at (7,0) {};
	\node [wb] at (6,0) {};
	\node [wb] at (5,0) {};
	\node [wb] at (4,0) {};
	\node [wb] at (3,0) {};
	\node [wb] at (2,0) {};
	\node [wb] at (1,0) {};
	\node [wb] at (0,0) {};
	\node [wb] at (-1,0) {};
	\node [wb] at (-2,0) {};
	\node [bb] at (-3,0) {};
	\node [wb] at (-4,0) {};
	\node [bb] at (-5,0) {};
	\node [wb] at (-6,0) {};
	\node [bb] at (-7,0) {};
	\node [wb] at (-8,0) {};
	\node [bb] at (-9,0) {};
		\draw (-10,0) node[]{$\ldots$};
	
	\node [wb] at (11,1) {};
	\node [wb] at (10,1) {};
	\node [wb] at (9,1) {};
	\node [wb] at (8,1) {};
	\node [wb] at (7,1) {};
	\node [wb] at (6,1) {};
	\node [wb] at (5,1) {};
	\node [wb] at (4,1) {};
	\node [wb] at (3,1) {};
	\node [wb] at (2,1) {};
	\node [wb] at (1,1) {};
	\node [wb] at (0,1) {};
	\node [wb] at (-1,1) {};
	\node [wb] at (-2,1) {};
	\node [wb] at (-3,1) {};
	\node [wb] at (-4,1) {};
	\node [bb] at (-5,1) {};
	\node [bb] at (-6,1) {};
	\node [wb] at (-7,1) {};
	\node [wb] at (-8,1) {};
	\node [bb] at (-9,1) {};
		\draw (-10,1) node[]{$\ldots$};
	
	\node [wb] at (11,2) {};
	\node [wb] at (10,2) {};
	\node [wb] at (9,2) {};
	\node [wb] at (8,2) {};
	\node [wb] at (7,2) {};
	\node [wb] at (6,2) {};
	\node [wb] at (5,2) {};
	\node [wb] at (4,2) {};
	\node [wb] at (3,2) {};
	\node [wb] at (2,2) {};
	\node [wb] at (1,2) {};
	\node [wb] at (0,2) {};
	\node [wb] at (-1,2) {};
	\node [wb] at (-2,2) {};
	\node [wb] at (-3,2) {};
	\node [bb] at (-4,2) {};
	\node [wb] at (-5,2) {};
	\node [wb] at (-6,2) {};
	\node [bb] at (-7,2) {};
	\node [wb] at (-8,2) {};
	\node [bb] at (-9,2) {};
		\draw (-10,2) node[]{$\ldots$};
	\end{tikzpicture}
\end{center}
\end{exa}
We can recover the associated multicharge by moving the black beads at the left in each runner.

\subsection{Symbol of a multipartition}
  To each  symbol $X =(\beta_i)_{i<m}$  (and thus to each abacus)   of charge $m$ we can canonically associate 
  a partition $\lambda (X)=(\lambda_1,\ldots,\lambda_r)$ such that for all $i\geq 1$, we have $\lambda_i=\beta_{m-i}+i-m$. Note that if  $k>>1$ then  $\lambda_k=0$. Regarding the abacus associated to the set of $\beta$-numbers, the parts of the partition are easily obtained by counting the numbers of white beads at the left of each black bead.

        \begin{exa}\label{aba1} Let  $X:=(\ldots,-1,0,3,4,6,8)$. The associated charge is $m=5$.   Then we have $\lambda (X)=(4,3,2,2)$. The abacus configuration gives:
        \begin{center}
\begin{tikzpicture}[scale=0.5, bb/.style={draw,circle,fill,minimum size=2.5mm,inner sep=0pt,outer sep=0pt}, wb/.style={draw,circle,fill=white,minimum size=2.5mm,inner sep=0pt,outer sep=0pt}]
	
	\node [] at (11,-1) {20};
	\node [] at (10,-1) {19};
	\node [] at (9,-1) {18};
	\node [] at (8,-1) {17};
	\node [] at (7,-1) {16};
	\node [] at (6,-1) {15};
	\node [] at (5,-1) {14};
	\node [] at (4,-1) {13};
	\node [] at (3,-1) {12};
	\node [] at (2,-1) {11};
	\node [] at (1,-1) {10};
	\node [] at (0,-1) {9};
	\node [] at (-1,-1) {8};
	\node [] at (-2,-1) {7};
	\node [] at (-3,-1) {6};
	\node [] at (-4,-1) {5};
	\node [] at (-5,-1) {4};
	\node [] at (-6,-1) {3};
	\node [] at (-7,-1) {2};
	\node [] at (-8,-1) {1};
	\node [] at (-9,-1) {0};
		\draw (-10,-1) node[]{$\ldots$};

	\node [wb] at (11,0) {};
	\node [wb] at (10,0) {};
	\node [wb] at (9,0) {};
	\node [wb] at (8,0) {};
	\node [wb] at (7,0) {};
	\node [wb] at (6,0) {};
	\node [wb] at (5,0) {};
	\node [wb] at (4,0) {};
	\node [wb] at (3,0) {};
	\node [wb] at (2,0) {};
	\node [wb] at (1,0) {};
	\node [wb] at (0,0) {};
	\node [bb] at (-1,0) {};
	\node [wb] at (-2,0) {};
	\node [bb] at (-3,0) {};
	\node [wb] at (-4,0) {};
	\node [bb] at (-5,0) {};
	\node [bb] at (-6,0) {};
	\node [wb] at (-7,0) {};
	\node [wb] at (-8,0) {};
	\node [bb] at (-9,0) {};
			\draw (-10,0) node[]{$\ldots$};
	\end{tikzpicture}
\end{center}

        \end{exa}
  It is easy to see that $\lambda (X)=\lambda (Y)$ if and only if $X \simeq Y$. 
    Conversely, to any partition $(\lambda_1,\ldots,\lambda_r)$, we can associate a set of $\beta$-numbers  (and thus an abacus). Let  $m\in \mathbb{Z}$.  Then we define:  
  $$X^m (\lambda) =(\beta_i)_{i< m}$$
  where  for all $j\geq 1$, we have $\beta_{m-j}=\lambda_j-j+m$ (with the convention that $\lambda_k=0$ if $k>r$).  
  Note that for all $(m_1,m_2)\in \mathbb{Z}^2$, we have $X^{m_1} (\lambda)\simeq 
  X^{m_2} (\lambda)$.

  To each $l$-symbol ${\bf X}=(X^0,\ldots,X^{l-1})$, we can associate an $l$-partition $\ulambda ({\bf X})=(\lambda (X^0),\ldots,\lambda (X^{l-1}))$ attached to this symbol  together with a multicharge ${\bf s} (X)\in \mathbb{Z}^l$ which is the multicharge of the symbol. Reciprocally, to each  multipartition   \footnote{The notation $\ulambda^l$ has not to be confused with $\lambda^l$ which is the $l$th component of $\ulambda$. It will be convenient to use this notation in the following.}  $\ulambda^l=(\lambda^0,\ldots,\lambda^{l-1})$ and multicharge ${\bf s}^l=(s_0,\ldots,s_{l-1})$, one can attach an $l$-symbol ${\bf X}^{{\bf s}^l} (\ulambda^l)=(X^{s_0} (\lambda^0),\ldots,X^{s_{l-1}} (\lambda^{l-1}))$. 
 
         \begin{exa}\label{A1} Let  $X=((,\ldots,0,3,4), (,\ldots,0,1,2,5))$ then we have  $\lambda (X)=((2,2),(2))$ and ${\bf s} (X)=(3,4)$.  The abacus configuration gives:
         
         
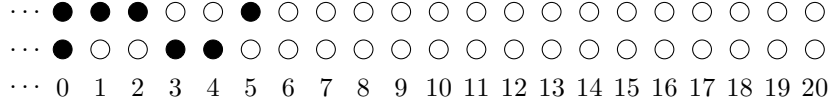
\begin{center}
\begin{tikzpicture}[scale=0.5, bb/.style={draw,circle,fill,minimum size=2.5mm,inner sep=0pt,outer sep=0pt}, wb/.style={draw,circle,fill=white,minimum size=2.5mm,inner sep=0pt,outer sep=0pt}]
	
	\node [] at (11,-1) {20};
	\node [] at (10,-1) {19};
	\node [] at (9,-1) {18};
	\node [] at (8,-1) {17};
	\node [] at (7,-1) {16};
	\node [] at (6,-1) {15};
	\node [] at (5,-1) {14};
	\node [] at (4,-1) {13};
	\node [] at (3,-1) {12};
	\node [] at (2,-1) {11};
	\node [] at (1,-1) {10};
	\node [] at (0,-1) {9};
	\node [] at (-1,-1) {8};
	\node [] at (-2,-1) {7};
	\node [] at (-3,-1) {6};
	\node [] at (-4,-1) {5};
	\node [] at (-5,-1) {4};
	\node [] at (-6,-1) {3};
	\node [] at (-7,-1) {2};
	\node [] at (-8,-1) {1};
	\node [] at (-9,-1) {0};
			\draw (-10,-1) node[]{$\ldots$};

	\node [wb] at (11,0) {};
	\node [wb] at (10,0) {};
	\node [wb] at (9,0) {};
	\node [wb] at (8,0) {};
	\node [wb] at (7,0) {};
	\node [wb] at (6,0) {};
	\node [wb] at (5,0) {};
	\node [wb] at (4,0) {};
	\node [wb] at (3,0) {};
	\node [wb] at (2,0) {};
	\node [wb] at (1,0) {};
	\node [wb] at (0,0) {};
	\node [wb] at (-1,0) {};
	\node [wb] at (-2,0) {};
	\node [wb] at (-3,0) {};
	\node [wb] at (-4,0) {};
	\node [bb] at (-5,0) {};
	\node [bb] at (-6,0) {};
	\node [wb] at (-7,0) {};
	\node [wb] at (-8,0) {};
	\node [bb] at (-9,0) {};
				\draw (-10,0) node[]{$\ldots$};
	
	\node [wb] at (11,1) {};
	\node [wb] at (10,1) {};11
	\node [wb] at (9,1) {};
	\node [wb] at (8,1) {};
	\node [wb] at (7,1) {};
	\node [wb] at (6,1) {};
	\node [wb] at (5,1) {};
	\node [wb] at (4,1) {};
	\node [wb] at (3,1) {};
	\node [wb] at (2,1) {};
	\node [wb] at (1,1) {};
	\node [wb] at (0,1) {};
	\node [wb] at (-1,1) {};
	\node [wb] at (-2,1) {};
	\node [wb] at (-3,1) {};
	\node [bb] at (-4,1) {};
	\node [wb] at (-5,1) {};
	\node [wb] at (-6,1) {};
	\node [bb] at (-7,1) {};
	\node [bb] at (-8,1) {};
	\node [bb] at (-9,1) {};
				\draw (-10,1) node[]{$\ldots$};

	\end{tikzpicture}

\end{center}

        \end{exa}

\subsection{Core and quotient}  
With all  these notions in mind, we will now be able to define the core and the quotient of a partition.  We fix $m\in \mathbb{Z}$. 
 Importantly, our notion of $e$-quotient  will depend  on this choice. 
Let $e\in \mathbb{N}$ and let $n\in \mathbb{N}$ and let 
 $\lambda \in \Pi (n)$.

First recall that   we  can  associate to $\lambda$ a 
  set of $\beta$-numbers $X^m (\lambda)$.  As usual, we take 
    $$X^m (\lambda) =(\beta_i)_{i< m}).$$
We then define:
  $${\bf X}^m (\lambda)=(X^0,\ldots,X^{e-1}),$$
  where for all $j\in \{0,\ldots,e-1\}$, $X^j$ is the  set of increasing integers obtained by reordering the set :
  $$\{ k\in \mathbb{N}\ |\ j+ke \in X^m (\lambda)\}.$$
 We denote 
  $$\mathbb{Z}^e [m]:=\{ (s_0,\ldots,s_{e-1}) \in \mathbb{Z}^e\ |\ \sum_{1\leq i\leq e-1} s_i=m\}.$$ 
Let $\ulambda_e$ be the multipartition and ${\bf s}_e:=(s_0,\ldots,s_{e-1}) \in \mathbb{Z}^e [m]$ be\footnote{the notation $\ulambda_e$ has not be be confused with the $e$th part of a partition. It is convenient to use this notation because of the level rank duality which will be define later.}  such that ${\bf X}^{{\bf s}_e} (\ulambda_e)={\bf X}^m (\lambda)$.  We denote 
$$\tau_e (\lambda) =(\ulambda_e,{\bf s}_e).$$   
We will use the following definitions:
\begin{itemize}
\item  The $e$-partition $\ulambda_e$ is the {\it $e$-quotient} of $(\lambda,m)$. 
\item Let  $\lambda^{\circ}$ be  such that ${\bf  X}^m (\lambda^{\circ})={\bf X}^{{\bf s}_e} (\uemptyset)$.   It is known as the {\it $e$-core partition} of $(\lambda,m)$.  
\item   The pair $(\lambda^{\circ},m)\in \Pi \times \mathbb{Z}$ is the {\it $e$-core} of $(\lambda,m)$ (this notations seems to be strange but it  will be convenient to use it in order to be consistent with our generalizations later in the paper). 
\item We see that the datum of the $e$-core partition (and $m$)  is equivalent to the datum of 
 ${\bf s}_e$. It will be in fact convenient for our purpose to also name the multicharge ${\bf s}_e\in \mathbb{Z}^e [m]$  itself as 
  the {\it $e$-core multicharge} of $\lambda$ and this is what will be done in the rest of the paper.  
  \end{itemize}
  In fact, the notion of $e$-core partition does not depend on $m$ but the $e$-quotient does.

\begin{exa}
Let $\lambda=(6,3,2,1,1)$.  Take $e=3$.  We have:
$$X^0(\lambda)=(\ldots,-7,-6,-4,-3,-1,1,5).$$
Then  the associated abacus is given as follows:
        
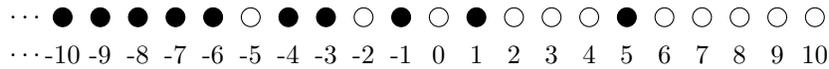
\begin{center}
\begin{tikzpicture}[scale=0.5, bb/.style={draw,circle,fill,minimum size=2.5mm,inner sep=0pt,outer sep=0pt}, wb/.style={draw,circle,fill=white,minimum size=2.5mm,inner sep=0pt,outer sep=0pt}]
	
	\node [] at (11,-1) {10};
	\node [] at (10,-1) {9};
	\node [] at (9,-1) {8};
	\node [] at (8,-1) {7};
	\node [] at (7,-1) {6};
	\node [] at (6,-1) {5};
	\node [] at (5,-1) {4};
	\node [] at (4,-1) {3};
	\node [] at (3,-1) {2};
	\node [] at (2,-1) {1};
	\node [] at (1,-1) {0};
	\node [] at (0,-1) {-1};
	\node [] at (-1,-1) {-2};
	\node [] at (-2,-1) {-3};
	\node [] at (-3,-1) {-4};
	\node [] at (-4,-1) {-5};
	\node [] at (-5,-1) {-6};
	\node [] at (-6,-1) {-7};
	\node [] at (-7,-1) {-8};
	\node [] at (-8,-1) {-9};
	\node [] at (-9,-1) {-10};
			\draw (-10,-1) node[]{$\ldots$};

	\node [wb] at (11,0) {};
	\node [wb] at (10,0) {};
	\node [wb] at (9,0) {};
	\node [wb] at (8,0) {};
	\node [wb] at (7,0) {};
	\node [bb] at (6,0) {};
	\node [wb] at (5,0) {};
	\node [wb] at (4,0) {};
	\node [wb] at (3,0) {};
	\node [bb] at (2,0) {};
	\node [wb] at (1,0) {};
	\node [bb] at (0,0) {};
	\node [wb] at (-1,0) {};
	\node [bb] at (-2,0) {};
	\node [bb] at (-3,0) {};
	\node [wb] at (-4,0) {};
	\node [bb] at (-5,0) {};
	\node [bb] at (-6,0) {};
	\node [bb] at (-7,0) {};
	\node [bb] at (-8,0) {};
	\node [bb] at (-9,0) {};
				\draw (-10,0) node[]{$\ldots$};
	\end{tikzpicture}
\end{center}
The associated $3$-abacus is :
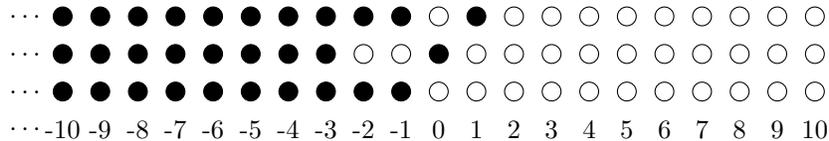
\begin{center}
\begin{tikzpicture}[scale=0.5, bb/.style={draw,circle,fill,minimum size=2.5mm,inner sep=0pt,outer sep=0pt}, wb/.style={draw,circle,fill=white,minimum size=2.5mm,inner sep=0pt,outer sep=0pt}]
	
	\node [] at (11,-1) {10};
	\node [] at (10,-1) {9};
	\node [] at (9,-1) {8};
	\node [] at (8,-1) {7};
	\node [] at (7,-1) {6};
	\node [] at (6,-1) {5};
	\node [] at (5,-1) {4};
	\node [] at (4,-1) {3};
	\node [] at (3,-1) {2};
	\node [] at (2,-1) {1};
	\node [] at (1,-1) {0};
	\node [] at (0,-1) {-1};
	\node [] at (-1,-1) {-2};
	\node [] at (-2,-1) {-3};
	\node [] at (-3,-1) {-4};
	\node [] at (-4,-1) {-5};
	\node [] at (-5,-1) {-6};
	\node [] at (-6,-1) {-7};
	\node [] at (-7,-1) {-8};
	\node [] at (-8,-1) {-9};
	\node [] at (-9,-1) {-10};
			\draw (-10,-1) node[]{$\ldots$};
	
	\node [wb] at (11,0) {};
	\node [wb] at (10,0) {};
	\node [wb] at (9,0) {};
	\node [wb] at (8,0) {};
	\node [wb] at (7,0) {};
	\node [wb] at (6,0) {};
	\node [wb] at (5,0) {};
	\node [wb] at (4,0) {};
	\node [wb] at (3,0) {};
	\node [wb] at (2,0) {};
	\node [wb] at (1,0) {};
	\node [bb] at (0,0) {};
	\node [bb] at (-1,0) {};
	\node [bb] at (-2,0) {};
	\node [bb] at (-3,0) {};
	\node [bb] at (-4,0) {};
	\node [bb] at (-5,0) {};
	\node [bb] at (-6,0) {};
	\node [bb] at (-7,0) {};
	\node [bb] at (-8,0) {};
	\node [bb] at (-9,0) {};
				\draw (-10,0) node[]{$\ldots$};
	
	\node [wb] at (11,1) {};
	\node [wb] at (10,1) {};
	\node [wb] at (9,1) {};
	\node [wb] at (8,1) {};
	\node [wb] at (7,1) {};
	\node [wb] at (6,1) {};
	\node [wb] at (5,1) {};
	\node [wb] at (4,1) {};
	\node [wb] at (3,1) {};
	\node [wb] at (2,1) {};
	\node [bb] at (1,1) {};
	\node [wb] at (0,1) {};
	\node [wb] at (-1,1) {};
	\node [bb] at (-2,1) {};
	\node [bb] at (-3,1) {};
	\node [bb] at (-4,1) {};
	\node [bb] at (-5,1) {};
	\node [bb] at (-6,1) {};
	\node [bb] at (-7,1) {};
	\node [bb] at (-8,1) {};
	\node [bb] at (-9,1) {};
				\draw (-10,1) node[]{$\ldots$};
	
	\node [wb] at (11,2) {};
	\node [wb] at (10,2) {};
	\node [wb] at (9,2) {};
	\node [wb] at (8,2) {};
	\node [wb] at (7,2) {};
	\node [wb] at (6,2) {};
	\node [wb] at (5,2) {};
	\node [wb] at (4,2) {};
	\node [wb] at (3,2) {};
	\node [bb] at (2,2) {};
	\node [wb] at (1,2) {};
	\node [bb] at (0,2) {};
	\node [bb] at (-1,2) {};
	\node [bb] at (-2,2) {};
	\node [bb] at (-3,2) {};
	\node [bb] at (-4,2) {};
	\node [bb] at (-5,2) {};
	\node [bb] at (-6,2) {};
	\node [bb] at (-7,2) {};
	\node [bb] at (-8,2) {};
	\node [bb] at (-9,2) {};
				\draw (-10,2) node[]{$\ldots$};
	\end{tikzpicture}
\end{center}
Each runner of the above $3$-abacus corresponds to an abacus which itself corresponds to a partition. The $3$-partition associated to this is the $3$-quotient: 
$$(\emptyset,(2),(1)).$$
To obtain the associated $3$-core, we  move all the black beads at the left end of each runner, we obtain the following $3$-abacus:
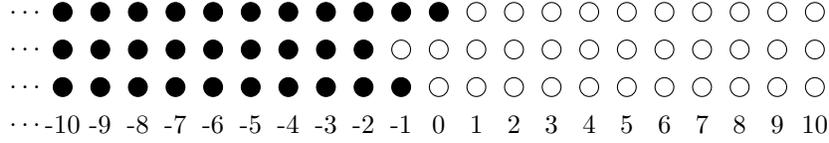
\begin{center}
\begin{tikzpicture}[scale=0.5, bb/.style={draw,circle,fill,minimum size=2.5mm,inner sep=0pt,outer sep=0pt}, wb/.style={draw,circle,fill=white,minimum size=2.5mm,inner sep=0pt,outer sep=0pt}]
	
	\node [] at (11,-1) {10};
	\node [] at (10,-1) {9};
	\node [] at (9,-1) {8};
	\node [] at (8,-1) {7};
	\node [] at (7,-1) {6};
	\node [] at (6,-1) {5};
	\node [] at (5,-1) {4};
	\node [] at (4,-1) {3};
	\node [] at (3,-1) {2};
	\node [] at (2,-1) {1};
	\node [] at (1,-1) {0};
	\node [] at (0,-1) {-1};
	\node [] at (-1,-1) {-2};
	\node [] at (-2,-1) {-3};
	\node [] at (-3,-1) {-4};
	\node [] at (-4,-1) {-5};
	\node [] at (-5,-1) {-6};
	\node [] at (-6,-1) {-7};
	\node [] at (-7,-1) {-8};
	\node [] at (-8,-1) {-9};
	\node [] at (-9,-1) {-10};
			\draw (-10,-1) node[]{$\ldots$};

	\node [wb] at (11,0) {};
	\node [wb] at (10,0) {};
	\node [wb] at (9,0) {};
	\node [wb] at (8,0) {};
	\node [wb] at (7,0) {};
	\node [wb] at (6,0) {};
	\node [wb] at (5,0) {};
	\node [wb] at (4,0) {};
	\node [wb] at (3,0) {};
	\node [wb] at (2,0) {};
	\node [wb] at (1,0) {};
	\node [bb] at (0,0) {};
	\node [bb] at (-1,0) {};
	\node [bb] at (-2,0) {};
	\node [bb] at (-3,0) {};
	\node [bb] at (-4,0) {};
	\node [bb] at (-5,0) {};
	\node [bb] at (-6,0) {};
	\node [bb] at (-7,0) {};
	\node [bb] at (-8,0) {};
	\node [bb] at (-9,0) {};
				\draw (-10,0) node[]{$\ldots$};

	\node [wb] at (11,1) {};
	\node [wb] at (10,1) {};
	\node [wb] at (9,1) {};
	\node [wb] at (8,1) {};
	\node [wb] at (7,1) {};
	\node [wb] at (6,1) {};
	\node [wb] at (5,1) {};
	\node [wb] at (4,1) {};
	\node [wb] at (3,1) {};
	\node [wb] at (2,1) {};
	\node [wb] at (1,1) {};
	\node [wb] at (0,1) {};
	\node [bb] at (-1,1) {};
	\node [bb] at (-2,1) {};
	\node [bb] at (-3,1) {};
	\node [bb] at (-4,1) {};
	\node [bb] at (-5,1) {};
	\node [bb] at (-6,1) {};
	\node [bb] at (-7,1) {};
	\node [bb] at (-8,1) {};
	\node [bb] at (-9,1) {};
				\draw (-10,1) node[]{$\ldots$};
	
	\node [wb] at (11,2) {};
	\node [wb] at (10,2) {};
	\node [wb] at (9,2) {};
	\node [wb] at (8,2) {};
	\node [wb] at (7,2) {};
	\node [wb] at (6,2) {};
	\node [wb] at (5,2) {};
	\node [wb] at (4,2) {};
	\node [wb] at (3,2) {};
	\node [wb] at (2,2) {};
	\node [bb] at (1,2) {};
	\node [bb] at (0,2) {};
	\node [bb] at (-1,2) {};
	\node [bb] at (-2,2) {};
	\node [bb] at (-3,2) {};
	\node [bb] at (-4,2) {};
	\node [bb] at (-5,2) {};
	\node [bb] at (-6,2) {};
	\node [bb] at (-7,2) {};
	\node [bb] at (-8,2) {};
	\node [bb] at (-9,2) {};
				\draw (-10,2) node[]{$\ldots$};
	\end{tikzpicture}
\end{center}
We obtain the multicharge $(0,-1,1)$ which is thus the $e$-core multicharge of $(\lambda,0)$. We then come back to the $1$-abacus associated to this $3$-abacus :
        
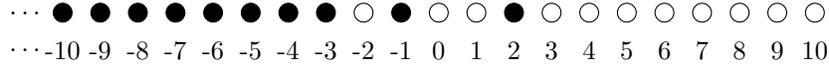
\begin{center}
\begin{tikzpicture}[scale=0.5, bb/.style={draw,circle,fill,minimum size=2.5mm,inner sep=0pt,outer sep=0pt}, wb/.style={draw,circle,fill=white,minimum size=2.5mm,inner sep=0pt,outer sep=0pt}]
	\node [] at (11,-1) {10};
	\node [] at (10,-1) {9};
	\node [] at (9,-1) {8};
	\node [] at (8,-1) {7};
	\node [] at (7,-1) {6};
	\node [] at (6,-1) {5};
	\node [] at (5,-1) {4};
	\node [] at (4,-1) {3};
	\node [] at (3,-1) {2};
	\node [] at (2,-1) {1};
	\node [] at (1,-1) {0};
	\node [] at (0,-1) {-1};
	\node [] at (-1,-1) {-2};
	\node [] at (-2,-1) {-3};
	\node [] at (-3,-1) {-4};
	\node [] at (-4,-1) {-5};
	\node [] at (-5,-1) {-6};
	\node [] at (-6,-1) {-7};
	\node [] at (-7,-1) {-8};
	\node [] at (-8,-1) {-9};
	\node [] at (-9,-1) {-10};
			\draw (-10,-1) node[]{$\ldots$};

	\node [wb] at (11,0) {};
	\node [wb] at (10,0) {};
	\node [wb] at (9,0) {};
	\node [wb] at (8,0) {};
	\node [wb] at (7,0) {};
	\node [wb] at (6,0) {};
	\node [wb] at (5,0) {};
	\node [wb] at (4,0) {};
	\node [bb] at (3,0) {};
	\node [wb] at (2,0) {};
	\node [wb] at (1,0) {};
	\node [bb] at (0,0) {};
	\node [wb] at (-1,0) {};
	\node [bb] at (-2,0) {};
	\node [bb] at (-3,0) {};
	\node [bb] at (-4,0) {};
	\node [bb] at (-5,0) {};
	\node [bb] at (-6,0) {};
	\node [bb] at (-7,0) {};
	\node [bb] at (-8,0) {};
	\node [bb] at (-9,0) {};
				\draw (-10,0) node[]{$\ldots$};
	\end{tikzpicture}
\end{center}
We have $\lambda^{\circ}=(3,1)$ which is the $e$-core partition of $(\lambda,0)$.  The $e$-core of $(\lambda,0)$ is 
   $(\lambda^{\circ},0)$.

\end{exa}

We say that a partition $\lambda$ is an {\it $e$-core partition} if the $e$-core  partition of $\lambda$ is $\lambda$  itself. We denote by $\operatorname{Cor} (e)$ the set of $e$-core partitions. The discussion above shows that  there is  a bijection between 
$\operatorname{Cor} (e)$ and $\mathbb{Z}^e [m]$.

Given a $m\in \mathbb{Z}$, an $e$-partition $\ulambda=(\lambda^0,\ldots,\lambda^{e-1})$ and an $e$-core $\lambda^{\circ}$ partition (or an element in $\mathbb{Z}^e [m]$), there is a unique pair $(\lambda,m)\in \Pi \times \mathbb{Z}$ with $e$-quotient $\ulambda$ and $e$-core $\lambda^{\circ}$.

\begin{exa}
Take the multicharge ${\bf s}=(0,1,-1)$ (which is of sum $m=0$), the abacus associated  to $X^{{\bf s}} (\emptyset)$ is 
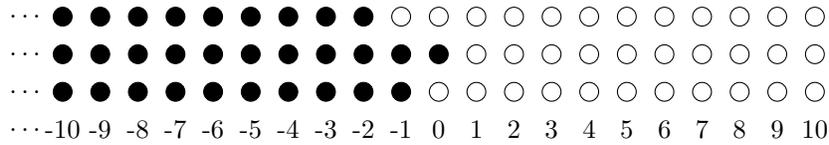
\begin{center}
\begin{tikzpicture}[scale=0.5, bb/.style={draw,circle,fill,minimum size=2.5mm,inner sep=0pt,outer sep=0pt}, wb/.style={draw,circle,fill=white,minimum size=2.5mm,inner sep=0pt,outer sep=0pt}]
	
	\node [] at (11,-1) {10};
	\node [] at (10,-1) {9};
	\node [] at (9,-1) {8};
	\node [] at (8,-1) {7};
	\node [] at (7,-1) {6};
	\node [] at (6,-1) {5};
	\node [] at (5,-1) {4};
	\node [] at (4,-1) {3};
	\node [] at (3,-1) {2};
	\node [] at (2,-1) {1};
	\node [] at (1,-1) {0};
	\node [] at (0,-1) {-1};
	\node [] at (-1,-1) {-2};
	\node [] at (-2,-1) {-3};
	\node [] at (-3,-1) {-4};
	\node [] at (-4,-1) {-5};
	\node [] at (-5,-1) {-6};
	\node [] at (-6,-1) {-7};
	\node [] at (-7,-1) {-8};
	\node [] at (-8,-1) {-9};
	\node [] at (-9,-1) {-10};
			\draw (-10,-1) node[]{$\ldots$};

	\node [wb] at (11,0) {};
	\node [wb] at (10,0) {};
	\node [wb] at (9,0) {};
	\node [wb] at (8,0) {};
	\node [wb] at (7,0) {};
	\node [wb] at (6,0) {};
	\node [wb] at (5,0) {};
	\node [wb] at (4,0) {};
	\node [wb] at (3,0) {};
	\node [wb] at (2,0) {};
	\node [wb] at (1,0) {};
	\node [bb] at (0,0) {};
	\node [bb] at (-1,0) {};
	\node [bb] at (-2,0) {};
	\node [bb] at (-3,0) {};
	\node [bb] at (-4,0) {};
	\node [bb] at (-5,0) {};
	\node [bb] at (-6,0) {};
	\node [bb] at (-7,0) {};
	\node [bb] at (-8,0) {};
	\node [bb] at (-9,0) {};
				\draw (-10,0) node[]{$\ldots$};

	\node [wb] at (11,1) {};
	\node [wb] at (10,1) {};
	\node [wb] at (9,1) {};
	\node [wb] at (8,1) {};
	\node [wb] at (7,1) {};
	\node [wb] at (6,1) {};
	\node [wb] at (5,1) {};
	\node [wb] at (4,1) {};
	\node [wb] at (3,1) {};
	\node [wb] at (2,1) {};
	\node [bb] at (1,1) {};
	\node [bb] at (0,1) {};
	\node [bb] at (-1,1) {};
	\node [bb] at (-2,1) {};
	\node [bb] at (-3,1) {};
	\node [bb] at (-4,1) {};
	\node [bb] at (-5,1) {};
	\node [bb] at (-6,1) {};
	\node [bb] at (-7,1) {};
	\node [bb] at (-8,1) {};
	\node [bb] at (-9,1) {};
				\draw (-10,1) node[]{$\ldots$};
	
	\node [wb] at (11,2) {};
	\node [wb] at (10,2) {};
	\node [wb] at (9,2) {};
	\node [wb] at (8,2) {};
	\node [wb] at (7,2) {};
	\node [wb] at (6,2) {};
	\node [wb] at (5,2) {};
	\node [wb] at (4,2) {};
	\node [wb] at (3,2) {};
	\node [wb] at (2,2) {};
	\node [wb] at (1,2) {};
	\node [wb] at (0,2) {};
	\node [bb] at (-1,2) {};
	\node [bb] at (-2,2) {};
	\node [bb] at (-3,2) {};
	\node [bb] at (-4,2) {};
	\node [bb] at (-5,2) {};
	\node [bb] at (-6,2) {};
	\node [bb] at (-7,2) {};
	\node [bb] at (-8,2) {};
	\node [bb] at (-9,2) {};
				\draw (-10,2) node[]{$\ldots$};
	\end{tikzpicture}
\end{center}
and we have $X(\lambda^{\circ})=X^{{\bf s}} (\emptyset)$ for $\lambda^\circ=(2)$ (which is thus a $3$-core). 
 We are interested to find the partition with $3$-core $(2)$ and $3$-quotient $((2,1),(2),(2,1,1))$. Then, we simply move the black beads in the above abacus to obtain the abacus of ${\bf X}^{(0,1,-1)} ((2,1),(2),(2,1,1))$. We obtain:
 
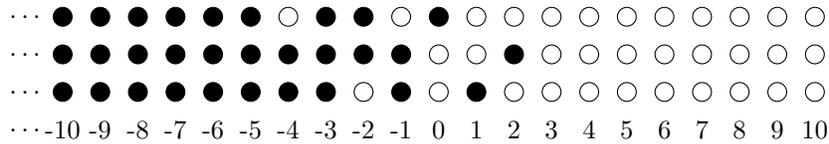
\begin{center}
\begin{tikzpicture}[scale=0.5, bb/.style={draw,circle,fill,minimum size=2.5mm,inner sep=0pt,outer sep=0pt}, wb/.style={draw,circle,fill=white,minimum size=2.5mm,inner sep=0pt,outer sep=0pt}]
	
	\node [] at (11,-1) {10};
	\node [] at (10,-1) {9};
	\node [] at (9,-1) {8};
	\node [] at (8,-1) {7};
	\node [] at (7,-1) {6};
	\node [] at (6,-1) {5};
	\node [] at (5,-1) {4};
	\node [] at (4,-1) {3};
	\node [] at (3,-1) {2};
	\node [] at (2,-1) {1};
	\node [] at (1,-1) {0};
	\node [] at (0,-1) {-1};
	\node [] at (-1,-1) {-2};
	\node [] at (-2,-1) {-3};
	\node [] at (-3,-1) {-4};
	\node [] at (-4,-1) {-5};
	\node [] at (-5,-1) {-6};
	\node [] at (-6,-1) {-7};
	\node [] at (-7,-1) {-8};
	\node [] at (-8,-1) {-9};
	\node [] at (-9,-1) {-10};
			\draw (-10,-1) node[]{$\ldots$};

	\node [wb] at (11,0) {};
	\node [wb] at (10,0) {};
	\node [wb] at (9,0) {};
	\node [wb] at (8,0) {};
	\node [wb] at (7,0) {};
	\node [wb] at (6,0) {};
	\node [wb] at (5,0) {};
	\node [wb] at (4,0) {};
	\node [wb] at (3,0) {};
	\node [bb] at (2,0) {};
	\node [wb] at (1,0) {};
	\node [bb] at (0,0) {};
	\node [wb] at (-1,0) {};
	\node [bb] at (-2,0) {};
	\node [bb] at (-3,0) {};
	\node [bb] at (-4,0) {};
	\node [bb] at (-5,0) {};
	\node [bb] at (-6,0) {};
	\node [bb] at (-7,0) {};
	\node [bb] at (-8,0) {};
	\node [bb] at (-9,0) {};
				\draw (-10,0) node[]{$\ldots$};

	\node [wb] at (11,1) {};
	\node [wb] at (10,1) {};
	\node [wb] at (9,1) {};
	\node [wb] at (8,1) {};
	\node [wb] at (7,1) {};
	\node [wb] at (6,1) {};
	\node [wb] at (5,1) {};
	\node [wb] at (4,1) {};
	\node [bb] at (3,1) {};
	\node [wb] at (2,1) {};
	\node [wb] at (1,1) {};
	\node [bb] at (0,1) {};
	\node [bb] at (-1,1) {};
	\node [bb] at (-2,1) {};
	\node [bb] at (-3,1) {};
	\node [bb] at (-4,1) {};
	\node [bb] at (-5,1) {};
	\node [bb] at (-6,1) {};
	\node [bb] at (-7,1) {};
	\node [bb] at (-8,1) {};
	\node [bb] at (-9,1) {};
				\draw (-10,1) node[]{$\ldots$};
	
	\node [wb] at (11,2) {};
	\node [wb] at (10,2) {};
	\node [wb] at (9,2) {};
	\node [wb] at (8,2) {};
	\node [wb] at (7,2) {};
	\node [wb] at (6,2) {};
	\node [wb] at (5,2) {};
	\node [wb] at (4,2) {};
	\node [wb] at (3,2) {};
	\node [wb] at (2,2) {};
	\node [bb] at (1,2) {};
	\node [wb] at (0,2) {};
	\node [bb] at (-1,2) {};
	\node [bb] at (-2,2) {};
	\node [wb] at (-3,2) {};
	\node [bb] at (-4,2) {};
	\node [bb] at (-5,2) {};
	\node [bb] at (-6,2) {};
	\node [bb] at (-7,2) {};
	\node [bb] at (-8,2) {};
	\node [bb] at (-9,2) {};
				\draw (-10,2) node[]{$\ldots$};
	\end{tikzpicture}
\end{center}
Now, we perform the associated $1$-abacus: 
        
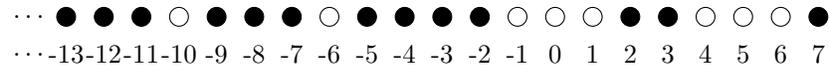
\begin{center}
\begin{tikzpicture}[scale=0.5, bb/.style={draw,circle,fill,minimum size=2.5mm,inner sep=0pt,outer sep=0pt}, wb/.style={draw,circle,fill=white,minimum size=2.5mm,inner sep=0pt,outer sep=0pt}]
	\node [] at (11,-1) {7};
	\node [] at (10,-1) {6};
	\node [] at (9,-1) {5};
	\node [] at (8,-1) {4};
	\node [] at (7,-1) {3};
	\node [] at (6,-1) {2};
	\node [] at (5,-1) {1};
	\node [] at (4,-1) {0};
	\node [] at (3,-1) {-1};
	\node [] at (2,-1) {-2};
	\node [] at (1,-1) {-3};
	\node [] at (0,-1) {-4};
	\node [] at (-1,-1) {-5};
	\node [] at (-2,-1) {-6};
	\node [] at (-3,-1) {-7};
	\node [] at (-4,-1) {-8};
	\node [] at (-5,-1) {-9};
	\node [] at (-6,-1) {-10};
	\node [] at (-7,-1) {-11};
	\node [] at (-8,-1) {-12};
	\node [] at (-9,-1) {-13};
			\draw (-10,-1) node[]{$\ldots$};

	\node [bb] at (11,0) {};
	\node [wb] at (10,0) {};
	\node [wb] at (9,0) {};
	\node [wb] at (8,0) {};
	\node [bb] at (7,0) {};
	\node [bb] at (6,0) {};
	\node [wb] at (5,0) {};
	\node [wb] at (4,0) {};
	\node [wb] at (3,0) {};
	\node [bb] at (2,0) {};
	\node [bb] at (1,0) {};
	\node [bb] at (0,0) {};
	\node [bb] at (-1,0) {};
	\node [wb] at (-2,0) {};
	\node [bb] at (-3,0) {};
	\node [bb] at (-4,0) {};
	\node [bb] at (-5,0) {};
	\node [wb] at (-6,0) {};
	\node [bb] at (-7,0) {};
	\node [bb] at (-8,0) {};
	\node [bb] at (-9,0) {};
				\draw (-10,0) node[]{$\ldots$};
	\end{tikzpicture}
\end{center}
which is the abacus of the partition $(8,5,5,2,2,2,2,1,1,1)$. This is the desired partition. 
\end{exa}

\subsection{Nodes}\label{nodes}  Assume that $\ulambda^l$ is an  $l$-partition and fix ${\bf s}^l=(s_0,\ldots,s_{l-1})$.
We can define its Young diagram 
$$[\ulambda^l]:=\{ (a,b,c)  \ | \ a\geq 1,\ c\in \{0,\ldots,l-1\},\ 1\leq b\leq \lambda_a^c\} \subset \mathbb{Z}_{>0}\times 
  \mathbb{Z}_{>0} \times \{0,\ldots,l-1\}.$$
  The elements in $[\ulambda^l]$ are called the {\it nodes} of $\ulambda^l$. 
  The {\it content}  of  a node $\gamma=(a,b,c)$ of $\ulambda^l$ is the element $b-a+s_c$  of $\mathbb{Z}$ and the residue is the content modulo $e\mathbb{Z}$.    
 We say that:
 \begin{itemize}
 \item the node $(a,b,c)\in [\ulambda^l]$ is a {\it removable $j$-node} if $[\ulambda^l]\setminus \{(a,b,c)\}$ is the Young diagram of a well defined $l$-partition and if $b-a+s_c\equiv j+e\mathbb{Z}$ that is, the residue of $(a,b,c)$ is $j+e\mathbb{Z}$.
 \item  $(a,b,c)\in \mathbb{N}^2\times \{0,1,\ldots,l-1\} \setminus  [\ulambda^l]$ is an {\it addable $j$-node}  if $[\ulambda^l]\sqcup  \{(a,b,c)\}$ is the Young diagram of a well defined $l$-partition 
  and if $b-a+s_c\equiv j+e\mathbb{Z}$ that is, the residue of $(a,b,c)$ is $j+e\mathbb{Z}$.
  \end{itemize}
Consider the associated symbol ${\bf X}:={\bf X}^{{\bf s}^l} (\ulambda^l)$ and the associated $l$-abacus. Then note that: 
\begin{itemize}
\item A removable $i$-node $\gamma$ in $\ulambda^l$ in component $c$ is canonically associated to 
 a black bead in the runner $c$  numbered by an element $j$ such that $j-1$ is occupied with a white bead.   We have $j\equiv i+e\mathbb{Z}$. 
  The abacus associated to the  $l$-partition $\umu^l$ obtained by removing this node  is obtained by exchanging the white and the black bead. 
  \item An addable  $i$-node $\gamma$ in $\ulambda^l$ in component $c$ is canonically associated to 
 a white bead in the runner $c$  numbered by an element $j$ such that $j-1$ is occupied with a black bead.  We have $j\equiv i+e\mathbb{Z}$.
  The abacus associated to the  $l$-partition $\umu^l$ obtained by adding this node  is obtained by exchanging the white and the black bead.

\end{itemize}
\begin{exa}
Keep Example \ref{A1} with $\ulambda^l:=((2,2),(2))$ and ${\bf s}^l=(3,4)$ and take $e=4$. We write the Young diagram with the residue in each box associated with its node:
$$(\ \ytableausetup{centertableaux}
\begin{ytableau}
3 & 0  \\ 
2& 1 
\end{ytableau}, \begin{ytableau}
0 & 1  
\end{ytableau}\ ).$$
We have two removable $1$-nodes : $(2,2,0)$ and $(1,2,1)$.  They are associated to two black beads numbered by $5$ in runner $1$ and $3$ in runner $0$. We have two addable $1$-nodes: $(1,3,0)$ and $(3,1,0)$ which are associated to white beads numbered by $1$ and $5$ in runner $0$.  We have also one addable $2$-node $(1,3,1)$ and one addable $3$-node $(2,1,1)$ associated with white beads numbered by $6$ and $3$ in runner $1$. 

\end{exa}

 \section{Multipartitions and combinatorial Level-rank duality}
 We have already seen how one can associate to each partition its $e$-core and its $e$-quotient.  In this part, we introduce a variation of these notions, defined by Uglov \cite{U} and we study the case of $l$-partitions.
 
\subsection{A variant of the notion of core and quotient} In \cite{U}, Uglov has given a variant of the notion of core and quotient.   Assume that $\lambda$ is a partition. First,   assume that we get the set 
$$X^m(\lambda)=(\beta_i)_{i<m}.$$
For all $j<m$, we write 
$$\beta_j=c_j+ed_j+elm_j,$$
where $c_j \in \{0,\ldots,e-1\}$, $d_i \in \{0,\ldots,l-1\}$ and $m_ij\in \mathbb{Z}$. Then we set
 $X^d$ to be   the  set of increasing integers obtained by reordering the set :
  $$\{ c_j+em_j \  |\  d_j=d\}.$$
  We obtain an $l$-symbol 
  $$(X^{l-1},\ldots,X^{0}),$$
  which will be called the {\it Uglov $l$-symbol}.  From this, we can associate to a partition $\lambda$ 
   an $l$-partition $\ulambda^l$   together with an $l$-multicharge ${\bf s}^l\in \mathbb{Z}^l [m]$, and of course, one can do the reverse process.   
We denote $\tau^l (\lambda,m)=(\ulambda^l,{\bf s}^l)$, $\tau^l$ is a thus a  bijection from the set of all partitions 
 to the set $\Pi^l \times \mathbb{Z}^l [m]$. 
  \begin{exa}
  Let $\lambda=(6,3,2,1,1)$ and $l=2$ and $e=3$, we have:
  $$X^0(\lambda)=(\ldots,-7,-6,-4,-3,-1,1,5).$$
  with abacus:
        \begin{center}
\begin{tikzpicture}[scale=0.5, bb/.style={draw,circle,fill,minimum size=2.5mm,inner sep=0pt,outer sep=0pt}, wb/.style={draw,circle,fill=white,minimum size=2.5mm,inner sep=0pt,outer sep=0pt}]
	
	\node [] at (11,-1) {10};
	\node [] at (10,-1) {9};
	\node [] at (9,-1) {8};
	\node [] at (8,-1) {7};
	\node [] at (7,-1) {6};
	\node [] at (6,-1) {5};
	\node [] at (5,-1) {4};
	\node [] at (4,-1) {3};
	\node [] at (3,-1) {2};
	\node [] at (2,-1) {1};
	\node [] at (1,-1) {0};
	\node [] at (0,-1) {-1};
	\node [] at (-1,-1) {-2};
	\node [] at (-2,-1) {-3};
	\node [] at (-3,-1) {-4};
	\node [] at (-4,-1) {-5};
	\node [] at (-5,-1) {-6};
	\node [] at (-6,-1) {-7};
	\node [] at (-7,-1) {-8};
	\node [] at (-8,-1) {-9};
	\node [] at (-9,-1) {-10};
			\draw (-10,-1) node[]{$\ldots$};

	\node [wb] at (11,0) {};
	\node [wb] at (10,0) {};
	\node [wb] at (9,0) {};
	\node [wb] at (8,0) {};
	\node [wb] at (7,0) {};
	\node [bb] at (6,0) {};
	\node [wb] at (5,0) {};
	\node [wb] at (4,0) {};
	\node [wb] at (3,0) {};
	\node [bb] at (2,0) {};
	\node [wb] at (1,0) {};
	\node [bb] at (0,0) {};
	\node [wb] at (-1,0) {};
	\node [bb] at (-2,0) {};
	\node [bb] at (-3,0) {};
	\node [wb] at (-4,0) {};
	\node [bb] at (-5,0) {};
	\node [bb] at (-6,0) {};
	\node [bb] at (-7,0) {};
	\node [bb] at (-8,0) {};
	\node [bb] at (-9,0) {};
				\draw (-10,0) node[]{$\ldots$};
	\end{tikzpicture}
\end{center}
The associated $l$-symbol is :
$$((\ldots,-3, -1,2),(\ldots, -3,-1,1)). $$
with abacus:
 \begin{center}
\begin{tikzpicture}[scale=0.5, bb/.style={draw,circle,fill,minimum size=2.5mm,inner sep=0pt,outer sep=0pt}, wb/.style={draw,circle,fill=white,minimum size=2.5mm,inner sep=0pt,outer sep=0pt}]
	
	\node [] at (11,-1) {10};
	\node [] at (10,-1) {9};
	\node [] at (9,-1) {8};
	\node [] at (8,-1) {7};
	\node [] at (7,-1) {6};
	\node [] at (6,-1) {5};
	\node [] at (5,-1) {4};
	\node [] at (4,-1) {3};
	\node [] at (3,-1) {2};
	\node [] at (2,-1) {1};
	\node [] at (1,-1) {0};
	\node [] at (0,-1) {-1};
	\node [] at (-1,-1) {-2};
	\node [] at (-2,-1) {-3};
	\node [] at (-3,-1) {-4};
	\node [] at (-4,-1) {-5};
	\node [] at (-5,-1) {-6};
	\node [] at (-6,-1) {-7};
	\node [] at (-7,-1) {-8};
	\node [] at (-8,-1) {-9};
	\node [] at (-9,-1) {-10};
			\draw (-10,-1) node[]{$\ldots$};
	\draw (-12,1) node[]{$X^1$};

	\node [wb] at (11,0) {};
	\node [wb] at (10,0) {};
	\node [wb] at (9,0) {};
	\node [wb] at (8,0) {};
	\node [wb] at (7,0) {};
	\node [wb] at (6,0) {};
	\node [wb] at (5,0) {};
	\node [wb] at (4,0) {};
	\node [bb] at (3,0) {};
	\node [wb] at (2,0) {};
	\node [wb] at (1,0) {};
	\node [bb] at (0,0) {};
	\node [wb] at (-1,0) {};
	\node [bb] at (-2,0) {};
	\node [bb] at (-3,0) {};
	\node [bb] at (-4,0) {};
	\node [bb] at (-5,0) {};
	\node [bb] at (-6,0) {};
	\node [bb] at (-7,0) {};
	\node [bb] at (-8,0) {};
	\node [bb] at (-9,0) {};
				\draw (-10,0) node[]{$\ldots$};
					\draw (-12,0) node[]{$X^0$};
	
	\node [wb] at (11,1) {};
	\node [wb] at (10,1) {};
	\node [wb] at (9,1) {};
	\node [wb] at (8,1) {};
	\node [wb] at (7,1) {};
	\node [wb] at (6,1) {};
	\node [wb] at (5,1) {};
	\node [wb] at (4,1) {};
	\node [wb] at (3,1) {};
	\node [bb] at (2,1) {};
	\node [wb] at (1,1) {};
	\node [bb] at (0,1) {};
	\node [wb] at (-1,1) {};
	\node [bb] at (-2,1) {};
	\node [bb] at (-3,1) {};
	\node [bb] at (-4,1) {};
	\node [bb] at (-5,1) {};
	\node [bb] at (-6,1) {};
	\node [bb] at (-7,1) {};
	\node [bb] at (-8,1) {};
	\node [bb] at (-9,1) {};
				\draw (-10,1) node[]{$\ldots$};
	\end{tikzpicture}
\end{center}
and we thus get $\tau^l (\lambda)=(((3,1),(2,1)),(0,0))$.

  \end{exa}

\subsection{Relations between the two notions of quotients}\label{proc}

Assume that we have a partition $\lambda$  and that $\tau^l (\lambda)=(\ulambda^l,{\bf s}^l)$,  
$\tau_e (\lambda)=(\ulambda_e,{\bf s}_e)$. Then one can go easily from the abacus of $\tau^l (\lambda)$
 to the abacus of $\tau_e (\lambda)$ and reciprocally.  To do that, start with the abacus 
  associated to  $(\ulambda^l,{\bf s}^l)$
 \begin{itemize}
 \item  We define a rectangle  on the $l$-abacus, containing $e$ beads in each abacus. This rectangle  starts with the beads numbered  with $0$ and finishes with the beads 
numbered  with $e-1$. We get a rectangle with $el$ beads, then again define a second rectangle with the beads numbered with $e$ to the beads numbered by  $2e-1$ an so on, even with the beads marked with negative integers. 
\item Rotate each rectangle $90$ degree anticlockwise. 
\item We get a new $e$-abacus, which is the $e$-abacus of $(\ulambda_e,{\bf s}_e)$. 
 \end{itemize}
 Thus take $l=2$ and $e=3$ and take the abacus of a partition $\lambda$. Then if we keep writing the labelling of the beads of the abacus of $\lambda$, the  $2$-abacus 
  of $\tau^l (\lambda)$ becomes:
   \begin{center}
\begin{tikzpicture}[scale=0.5, bb/.style={draw,circle,fill,minimum size=2.5mm,inner sep=0pt,outer sep=0pt}, wb/.style={draw,circle,fill=white,minimum size=2.5mm,inner sep=0pt,outer sep=0pt}]
	
	\node [] at (11,1) {19};
	\node [] at (10,1) {18};
	\node [] at (9,1) {14};
	\node [] at (8,1) {13};
	\node [] at (7,1) {12};
	\node [] at (6,1) {8};
	\node [] at (5,1) {7};
	\node [] at (4,1) {6};
	\node [] at (3,1) {2};
	\node [] at (2,1) {1};
	\node [] at (1,1) {0};
	\node [] at (0,1) {-4};
	\node [] at (-1,1) {-5};
	\node [] at (-2,1) {-6};
	\node [] at (-3,1) {-10};
	\node [] at (-4,1) {-11};
	\node [] at (-5,1) {-12};
	\node [] at (-6,1) {-16};
	\node [] at (-7,1) {-17};
	\node [] at (-8,1) {-18};
	\node [] at (-9,1) {-22};
			\draw (-10,1) node[]{$\ldots$};

	\node [] at (11,0) {22};
	\node [] at (10,0) {21};
	\node [] at (9,0) {17};
	\node [] at (8,0) {16};
	\node [] at (7,0) {15};
	\node [] at (6,0) {11};
	\node [] at (5,0) {10};
	\node [] at (4,0) {9};
	\node [] at (3,0) {5};
	\node [] at (2,0) {4};
	\node [] at (1,0) {3};
	\node [] at (0,0) {-1};
	\node [] at (-1,0) {-2};
	\node [] at (-2,0) {-3};
	\node [] at (-3,0) {-7};
	\node [] at (-4,0) {-8};
	\node [] at (-5,0) {-9};
	\node [] at (-6,0) {-13};
	\node [] at (-7,0) {-14};
	\node [] at (-8,0) {-15};
	\node [] at (-9,0) {-19};
			\draw (-10,-1) node[]{$\ldots$};
						\draw[dashed](0.5,-0.5)--node[]{}(0.5,1.5);
			\draw (-10,1) node[]{$\ldots$};
					\draw[dashed](3.5,-0.5)--node[]{}(3.5,1.5);		
						\draw (-10,1) node[]{$\ldots$};
					\draw[dashed](-2.5,-0.5)--node[]{}(-2.5,1.5);	
						\draw[dashed](-5.5,-0.5)--node[]{}(-5.5,1.5);	
							\draw[dashed](6.5,-0.5)--node[]{}(6.5,1.5);
								\draw[dashed](9.5,-0.5)--node[]{}(9.5,1.5);
									\draw[dashed](-8.5,-0.5)--node[]{}(-8.5,1.5);
	\end{tikzpicture}
\end{center}
 The $3$-abacus of $\tau_e (\lambda)$ then gives:
 
    \begin{center}
\begin{tikzpicture}[scale=0.5, bb/.style={draw,circle,fill,minimum size=2.5mm,inner sep=0pt,outer sep=0pt}, wb/.style={draw,circle,fill=white,minimum size=2.5mm,inner sep=0pt,outer sep=0pt}]
	
	\node [] at (10,2) {35};
	\node [] at (9,2) {32};
	\node [] at (8,2) {29};
	\node [] at (7,2) {26};
	\node [] at (6,2) {23};
	\node [] at (5,2) {20};
	\node [] at (4,2) {17};
	\node [] at (3,2) {14};
	\node [] at (2,2) {11};
	\node [] at (1,2) {8};
	\node [] at (0,2) {5};
	\node [] at (-1,2) {2};
	\node [] at (-2,2) {-1};
	\node [] at (-3,2) {-4};
	\node [] at (-4,2) {-7};
	\node [] at (-5,2) {-10};
	\node [] at (-6,2) {-13};
	\node [] at (-7,2) {-16};
	\node [] at (-8,2) {-19};
	\node [] at (-9,2) {-22};
			\draw (-10,2) node[]{$\ldots$};
	
	\node [] at (10,1) {34};
	\node [] at (9,1) {31};
	\node [] at (8,1) {28};
	\node [] at (7,1) {25};
	\node [] at (6,1) {22};
	\node [] at (5,1) {19};
	\node [] at (4,1) {16};
	\node [] at (3,1) {13};
	\node [] at (2,1) {10};
	\node [] at (1,1) {7};
	\node [] at (0,1) {4};
	\node [] at (-1,1) {1};
	\node [] at (-2,1) {-2};
	\node [] at (-3,1) {-5};
	\node [] at (-4,1) {-8};
	\node [] at (-5,1) {-11};
	\node [] at (-6,1) {-14};
	\node [] at (-7,1) {-17};
	\node [] at (-8,1) {-20};
	\node [] at (-9,1) {-23};
			\draw (-10,1) node[]{$\ldots$};

	\node [] at (10,0) {33};
	\node [] at (9,0) {30};
	\node [] at (8,0) {27};
	\node [] at (7,0) {24};
	\node [] at (6,0) {21};
	\node [] at (5,0) {18};
	\node [] at (4,0) {15};
	\node [] at (3,0) {12};
	\node [] at (2,0) {9};
	\node [] at (1,0) {6};
	\node [] at (0,0) {3};
	\node [] at (-1,0) {0};
	\node [] at (-2,0) {-3};
	\node [] at (-3,0) {-6};
	\node [] at (-4,0) {-9};
	\node [] at (-5,0) {-12};
	\node [] at (-6,0) {-15};
	\node [] at (-7,0) {-18};
	\node [] at (-8,0) {-21};
	\node [] at (-9,0) {-24};
			\draw (-10,0) node[]{$\ldots$};
					\draw (-10,2) node[]{$\ldots$};
				\draw[dashed](-7.5,-0.5)--node[]{}(-7.5,2.5);\draw (-10,2) node[]{$\ldots$};
				\draw[dashed](-5.5,-0.5)--node[]{}(-5.5,2.5);\draw (-10,2) node[]{$\ldots$};
				\draw[dashed](-3.5,-0.5)--node[]{}(-3.5,2.5);\draw (-10,2) node[]{$\ldots$};
				\draw[dashed](-1.5,-0.5)--node[]{}(-1.5,2.5);\draw (-10,2) node[]{$\ldots$};
				\draw[dashed](0.5,-0.5)--node[]{}(0.5,2.5);\draw (-10,2) node[]{$\ldots$};
				\draw[dashed](2.5,-0.5)--node[]{}(2.5,2.5);\draw (-10,2) node[]{$\ldots$};
				\draw[dashed](4.5,-0.5)--node[]{}(4.5,2.5);\draw (-10,2) node[]{$\ldots$};
				\draw[dashed](6.5,-0.5)--node[]{}(6.5,2.5);\draw (-10,2) node[]{$\ldots$};				
				\draw[dashed](8.5,-0.5)--node[]{}(8.5,2.5);\draw (-10,2) node[]{$\ldots$};

	\end{tikzpicture}
\end{center}

 \begin{exa}\label{2111}
 We keep the above example. Recall that we have $\tau^l (\lambda)=(((3,1),(2,1)),(0,0))$ with thus the above abacus. 
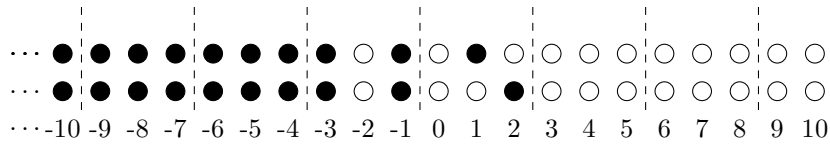
  \begin{center}
\begin{tikzpicture}[scale=0.5, bb/.style={draw,circle,fill,minimum size=2.5mm,inner sep=0pt,outer sep=0pt}, wb/.style={draw,circle,fill=white,minimum size=2.5mm,inner sep=0pt,outer sep=0pt}]
	
	\node [] at (11,-1) {10};
	\node [] at (10,-1) {9};
	\node [] at (9,-1) {8};
	\node [] at (8,-1) {7};
	\node [] at (7,-1) {6};
	\node [] at (6,-1) {5};
	\node [] at (5,-1) {4};
	\node [] at (4,-1) {3};
	\node [] at (3,-1) {2};
	\node [] at (2,-1) {1};
	\node [] at (1,-1) {0};
	\node [] at (0,-1) {-1};
	\node [] at (-1,-1) {-2};
	\node [] at (-2,-1) {-3};
	\node [] at (-3,-1) {-4};
	\node [] at (-4,-1) {-5};
	\node [] at (-5,-1) {-6};
	\node [] at (-6,-1) {-7};
	\node [] at (-7,-1) {-8};
	\node [] at (-8,-1) {-9};
	\node [] at (-9,-1) {-10};
			\draw (-10,-1) node[]{$\ldots$};

	\node [wb] at (11,0) {};
	\node [wb] at (10,0) {};
	\node [wb] at (9,0) {};
	\node [wb] at (8,0) {};
	\node [wb] at (7,0) {};
	\node [wb] at (6,0) {};
	\node [wb] at (5,0) {};
	\node [wb] at (4,0) {};
	\node [bb] at (3,0) {};
	\node [wb] at (2,0) {};
	\node [wb] at (1,0) {};
	\node [bb] at (0,0) {};
	\node [wb] at (-1,0) {};
	\node [bb] at (-2,0) {};
	\node [bb] at (-3,0) {};
	\node [bb] at (-4,0) {};
	\node [bb] at (-5,0) {};
	\node [bb] at (-6,0) {};
	\node [bb] at (-7,0) {};
	\node [bb] at (-8,0) {};
	\node [bb] at (-9,0) {};
				\draw (-10,0) node[]{$\ldots$};

	\node [wb] at (11,1) {};
	\node [wb] at (10,1) {};
	\node [wb] at (9,1) {};
	\node [wb] at (8,1) {};
	\node [wb] at (7,1) {};
	\node [wb] at (6,1) {};
	\node [wb] at (5,1) {};
	\node [wb] at (4,1) {};
	\node [wb] at (3,1) {};
	\node [bb] at (2,1) {};
	\node [wb] at (1,1) {};
	\node [bb] at (0,1) {};
	\node [wb] at (-1,1) {};
	\node [bb] at (-2,1) {};
	\node [bb] at (-3,1) {};
	\node [bb] at (-4,1) {};
	\node [bb] at (-5,1) {};
	\node [bb] at (-6,1) {};
	\node [bb] at (-7,1) {};
	\node [bb] at (-8,1) {};
	\node [bb] at (-9,1) {};
				\draw (-10,1) node[]{$\ldots$};
					\draw[dashed](0.5,-0.5)--node[]{}(0.5,2.5);
			\draw (-10,1) node[]{$\ldots$};
					\draw[dashed](3.5,-0.5)--node[]{}(3.5,2.5);		
						\draw (-10,1) node[]{$\ldots$};
					\draw[dashed](-2.5,-0.5)--node[]{}(-2.5,2.5);	
						\draw[dashed](-5.5,-0.5)--node[]{}(-5.5,2.5);	
							\draw[dashed](6.5,-0.5)--node[]{}(6.5,2.5);
								\draw[dashed](9.5,-0.5)--node[]{}(9.5,2.5);
									\draw[dashed](-8.5,-0.5)--node[]{}(-8.5,2.5);
	\end{tikzpicture}
\end{center}
and then after rotation:
 
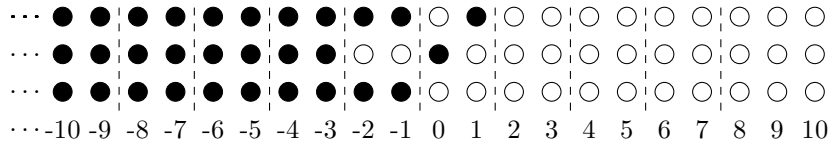
\begin{center}
\begin{tikzpicture}[scale=0.5, bb/.style={draw,circle,fill,minimum size=2.5mm,inner sep=0pt,outer sep=0pt}, wb/.style={draw,circle,fill=white,minimum size=2.5mm,inner sep=0pt,outer sep=0pt}]
	
	\node [] at (11,-1) {10};
	\node [] at (10,-1) {9};
	\node [] at (9,-1) {8};
	\node [] at (8,-1) {7};
	\node [] at (7,-1) {6};
	\node [] at (6,-1) {5};
	\node [] at (5,-1) {4};
	\node [] at (4,-1) {3};
	\node [] at (3,-1) {2};
	\node [] at (2,-1) {1};
	\node [] at (1,-1) {0};
	\node [] at (0,-1) {-1};
	\node [] at (-1,-1) {-2};
	\node [] at (-2,-1) {-3};
	\node [] at (-3,-1) {-4};
	\node [] at (-4,-1) {-5};
	\node [] at (-5,-1) {-6};
	\node [] at (-6,-1) {-7};
	\node [] at (-7,-1) {-8};
	\node [] at (-8,-1) {-9};
	\node [] at (-9,-1) {-10};
			\draw (-10,-1) node[]{$\ldots$};

	\node [wb] at (11,0) {};
	\node [wb] at (10,0) {};
	\node [wb] at (9,0) {};
	\node [wb] at (8,0) {};
	\node [wb] at (7,0) {};
	\node [wb] at (6,0) {};
	\node [wb] at (5,0) {};
	\node [wb] at (4,0) {};
	\node [wb] at (3,0) {};
	\node [wb] at (2,0) {};
	\node [wb] at (1,0) {};
	\node [bb] at (0,0) {};
	\node [bb] at (-1,0) {};
	\node [bb] at (-2,0) {};
	\node [bb] at (-3,0) {};
	\node [bb] at (-4,0) {};
	\node [bb] at (-5,0) {};
	\node [bb] at (-6,0) {};
	\node [bb] at (-7,0) {};
	\node [bb] at (-8,0) {};
	\node [bb] at (-9,0) {};
				\draw (-10,0) node[]{$\ldots$};

	\node [wb] at (11,1) {};
	\node [wb] at (10,1) {};
	\node [wb] at (9,1) {};
	\node [wb] at (8,1) {};
	\node [wb] at (7,1) {};
	\node [wb] at (6,1) {};
	\node [wb] at (5,1) {};
	\node [wb] at (4,1) {};
	\node [wb] at (3,1) {};
	\node [wb] at (2,1) {};
	\node [bb] at (1,1) {};
	\node [wb] at (0,1) {};
	\node [wb] at (-1,1) {};
	\node [bb] at (-2,1) {};
	\node [bb] at (-3,1) {};
	\node [bb] at (-4,1) {};
	\node [bb] at (-5,1) {};
	\node [bb] at (-6,1) {};
	\node [bb] at (-7,1) {};
	\node [bb] at (-8,1) {};
	\node [bb] at (-9,1) {};
				\draw (-10,1) node[]{$\ldots$};
	
	\node [wb] at (11,2) {};
	\node [wb] at (10,2) {};
	\node [wb] at (9,2) {};
	\node [wb] at (8,2) {};
	\node [wb] at (7,2) {};
	\node [wb] at (6,2) {};
	\node [wb] at (5,2) {};
	\node [wb] at (4,2) {};
	\node [wb] at (3,2) {};
	\node [bb] at (2,2) {};
	\node [wb] at (1,2) {};
	\node [bb] at (0,2) {};
	\node [bb] at (-1,2) {};
	\node [bb] at (-2,2) {};
	\node [bb] at (-3,2) {};
	\node [bb] at (-4,2) {};
	\node [bb] at (-5,2) {};
	\node [bb] at (-6,2) {};
	\node [bb] at (-7,2) {};
	\node [bb] at (-8,2) {};
	\node [bb] at (-9,2) {};
				\draw (-10,2) node[]{$\ldots$};
				\draw[dashed](-7.5,-0.5)--node[]{}(-7.5,2.5);\draw (-10,2) node[]{$\ldots$};
				\draw[dashed](-5.5,-0.5)--node[]{}(-5.5,2.5);\draw (-10,2) node[]{$\ldots$};
				\draw[dashed](-3.5,-0.5)--node[]{}(-3.5,2.5);\draw (-10,2) node[]{$\ldots$};
				\draw[dashed](-1.5,-0.5)--node[]{}(-1.5,2.5);\draw (-10,2) node[]{$\ldots$};
				\draw[dashed](0.5,-0.5)--node[]{}(0.5,2.5);\draw (-10,2) node[]{$\ldots$};
				\draw[dashed](2.5,-0.5)--node[]{}(2.5,2.5);\draw (-10,2) node[]{$\ldots$};
				\draw[dashed](4.5,-0.5)--node[]{}(4.5,2.5);\draw (-10,2) node[]{$\ldots$};
				\draw[dashed](6.5,-0.5)--node[]{}(6.5,2.5);\draw (-10,2) node[]{$\ldots$};				
				\draw[dashed](8.5,-0.5)--node[]{}(8.5,2.5);\draw (-10,2) node[]{$\ldots$};				
	\end{tikzpicture}
\end{center}
which is the $3$-abacus if $\tau_e (\lambda)=((\emptyset, (2),  (1)),(0,-1,1))$. Note that $(\emptyset,(2),(1))$ is indeed the $3$-quotient of $\lambda$.

 \end{exa}

\subsection{Nodes again} 
Let $\lambda$ be a partition and write $\tau^l (\lambda)=(\ulambda^l,{\bf s}^l)$. 
 Let $\tau_e (\lambda)=(\ulambda_e,{\bf s}_e)$ and denote the associated symbol by :
 $$(X^0,\ldots,X^{e-1}).$$ 
 Assume that there exists $i\in \{0,\ldots,e-2\}$ such that $j\in X^i$ and such that $j\notin X^{i+1}$. Let us consider the symbol
 ${\bf Y}=(Y^0,\ldots,Y^{e-1})$ such that $Y^k=X^k$ if $k\neq i,i+1$, $Y^i=X^i\setminus \{j\}$ and $Y^{i+1}=X^{i+1}\cup \{j\}$. The associated abacus is obtained  by moving the black bead numbered by $j$ from the $i$th runner to the $i+1$th one. 
We obtain an  $e$-symbol which is itself associated to a multipartition   $\umu^l$. This multipartition  is obtained from $\ulambda^l$ by removing a node with residue $i+e\mathbb{Z}$. Reciprocally, adding a node  with residue  $i+e\mathbb{Z}$ consists in doing the above manipulation on abacus/symbol.

 If now  $j\in X^{0}$ and  $j-l\notin X^{e-1}$, on can consider 
 ${\bf Y}=(X^0\setminus \{j\},\ldots,X^{e-1}\sqcup \{j-l\})$. The associated abacus is obtained by moving the black bead numbered by $j$ from the $0$th runner to the $e-1$th one at the place $j-l$. 
 Then we have the $e$-abacus associated to an $l$-partition  $\umu^l$ is obtained from $\ulambda^l$ by removing a node with residue $0+e\mathbb{Z}$. Reciprocally, adding  a node with residue $0+e\mathbb{Z}$  consists in doing the above manipulation on abacus/symbol. 

 
 \begin{exa}
 Take $l=2$, ${\bf s}_l=(0,0)$, $e=3$ and the $2$-partition $\ulambda^l=((3,1),(2,1))$. The associated $l$-abacus is 
  represented in Example \ref{2111}.  The Young diagram is : 
 $$(\ \ytableausetup{centertableaux}
\begin{ytableau}
0 & 1  & 2 \\ 
2
\end{ytableau}, \begin{ytableau}
0 & 1\\
2
\end{ytableau}\ ).$$
 
 We have $3$ removable $3$-nodes, and this means that, in the associated $e$-abacus $(X^0,X^1,X^3)$ ,  we must have exactly three beads in $X^2$   with no bead in the same position in $X^1$ which is the case for positions  $-2$, $-1$ and $1$. We have one addable $2$-node which means that we must have a bead in position $X^1$ with no bead in the same position in $X^2$. This is indeed the case for position $0$. 
 
 Let us  look at the removable / addable $0$-nodes. We have no removable  $0$-nodes.  Indeed, in $X^0$, we have  beads in position $x$ for $x\leq -1$ but we also have beads in position $x-2$ for $x\leq -1$. We have $3$ addable $0$ nodes and indeed, we have $3$ beads in $X^2$ in position  $1$, $-1$ and $-2$ with no bead in $X^0$ in position $3$, $1$ and $0$. If we move the bead from position $21$ in $X^2$ to position $3$ in $X^0$, we get the abacus:
  \begin{center}
\begin{tikzpicture}[scale=0.5, bb/.style={draw,circle,fill,minimum size=2.5mm,inner sep=0pt,outer sep=0pt}, wb/.style={draw,circle,fill=white,minimum size=2.5mm,inner sep=0pt,outer sep=0pt}]
	
	\node [] at (11,-1) {10};
	\node [] at (10,-1) {9};
	\node [] at (9,-1) {8};
	\node [] at (8,-1) {7};
	\node [] at (7,-1) {6};
	\node [] at (6,-1) {5};
	\node [] at (5,-1) {4};
	\node [] at (4,-1) {3};
	\node [] at (3,-1) {2};
	\node [] at (2,-1) {1};
	\node [] at (1,-1) {0};
	\node [] at (0,-1) {-1};
	\node [] at (-1,-1) {-2};
	\node [] at (-2,-1) {-3};
	\node [] at (-3,-1) {-4};
	\node [] at (-4,-1) {-5};
	\node [] at (-5,-1) {-6};
	\node [] at (-6,-1) {-7};
	\node [] at (-7,-1) {-8};
	\node [] at (-8,-1) {-9};
	\node [] at (-9,-1) {-10};
			\draw (-10,-1) node[]{$\ldots$};

	\node [wb] at (11,0) {};
	\node [wb] at (10,0) {};
	\node [wb] at (9,0) {};
	\node [wb] at (8,0) {};
	\node [wb] at (7,0) {};
	\node [wb] at (6,0) {};
	\node [wb] at (5,0) {};
	\node [bb] at (4,0) {};
	\node [wb] at (3,0) {};
	\node [wb] at (2,0) {};
	\node [wb] at (1,0) {};
	\node [bb] at (0,0) {};
	\node [bb] at (-1,0) {};
	\node [bb] at (-2,0) {};
	\node [bb] at (-3,0) {};
	\node [bb] at (-4,0) {};
	\node [bb] at (-5,0) {};
	\node [bb] at (-6,0) {};
	\node [bb] at (-7,0) {};
	\node [bb] at (-8,0) {};
	\node [bb] at (-9,0) {};
				\draw (-10,0) node[]{$\ldots$};

	\node [wb] at (11,1) {};
	\node [wb] at (10,1) {};
	\node [wb] at (9,1) {};
	\node [wb] at (8,1) {};
	\node [wb] at (7,1) {};
	\node [wb] at (6,1) {};
	\node [wb] at (5,1) {};
	\node [wb] at (4,1) {};
	\node [wb] at (3,1) {};
	\node [wb] at (2,1) {};
	\node [bb] at (1,1) {};
	\node [wb] at (0,1) {};
	\node [wb] at (-1,1) {};
	\node [bb] at (-2,1) {};
	\node [bb] at (-3,1) {};
	\node [bb] at (-4,1) {};
	\node [bb] at (-5,1) {};
	\node [bb] at (-6,1) {};
	\node [bb] at (-7,1) {};
	\node [bb] at (-8,1) {};
	\node [bb] at (-9,1) {};
				\draw (-10,1) node[]{$\ldots$};
	
	\node [wb] at (11,2) {};
	\node [wb] at (10,2) {};
	\node [wb] at (9,2) {};
	\node [wb] at (8,2) {};
	\node [wb] at (7,2) {};
	\node [wb] at (6,2) {};
	\node [wb] at (5,2) {};
	\node [wb] at (4,2) {};
	\node [wb] at (3,2) {};
	\node [wb] at (2,2) {};
	\node [wb] at (1,2) {};
	\node [bb] at (0,2) {};
	\node [bb] at (-1,2) {};
	\node [bb] at (-2,2) {};
	\node [bb] at (-3,2) {};
	\node [bb] at (-4,2) {};
	\node [bb] at (-5,2) {};
	\node [bb] at (-6,2) {};
	\node [bb] at (-7,2) {};
	\node [bb] at (-8,2) {};
	\node [bb] at (-9,2) {};
				\draw (-10,2) node[]{$\ldots$};
				\draw[dashed](-7.5,-0.5)--node[]{}(-7.5,2.5);\draw (-10,2) node[]{$\ldots$};
				\draw[dashed](-5.5,-0.5)--node[]{}(-5.5,2.5);\draw (-10,2) node[]{$\ldots$};
				\draw[dashed](-3.5,-0.5)--node[]{}(-3.5,2.5);\draw (-10,2) node[]{$\ldots$};
				\draw[dashed](-1.5,-0.5)--node[]{}(-1.5,2.5);\draw (-10,2) node[]{$\ldots$};
				\draw[dashed](0.5,-0.5)--node[]{}(0.5,2.5);\draw (-10,2) node[]{$\ldots$};
				\draw[dashed](2.5,-0.5)--node[]{}(2.5,2.5);\draw (-10,2) node[]{$\ldots$};
				\draw[dashed](4.5,-0.5)--node[]{}(4.5,2.5);\draw (-10,2) node[]{$\ldots$};
				\draw[dashed](6.5,-0.5)--node[]{}(6.5,2.5);\draw (-10,2) node[]{$\ldots$};				
				\draw[dashed](8.5,-0.5)--node[]{}(8.5,2.5);\draw (-10,2) node[]{$\ldots$};				
	\end{tikzpicture}
\end{center}
The associated $l$-abacus is 
   \begin{center}
\begin{tikzpicture}[scale=0.5, bb/.style={draw,circle,fill,minimum size=2.5mm,inner sep=0pt,outer sep=0pt}, wb/.style={draw,circle,fill=white,minimum size=2.5mm,inner sep=0pt,outer sep=0pt}]
	
	\node [] at (11,-1) {10};
	\node [] at (10,-1) {9};
	\node [] at (9,-1) {8};
	\node [] at (8,-1) {7};
	\node [] at (7,-1) {6};
	\node [] at (6,-1) {5};
	\node [] at (5,-1) {4};
	\node [] at (4,-1) {3};
	\node [] at (3,-1) {2};
	\node [] at (2,-1) {1};
	\node [] at (1,-1) {0};
	\node [] at (0,-1) {-1};
	\node [] at (-1,-1) {-2};
	\node [] at (-2,-1) {-3};
	\node [] at (-3,-1) {-4};
	\node [] at (-4,-1) {-5};
	\node [] at (-5,-1) {-6};
	\node [] at (-6,-1) {-7};
	\node [] at (-7,-1) {-8};
	\node [] at (-8,-1) {-9};
	\node [] at (-9,-1) {-10};
			\draw (-10,-1) node[]{$\ldots$};

	\node [wb] at (11,0) {};
	\node [wb] at (10,0) {};
	\node [wb] at (9,0) {};
	\node [wb] at (8,0) {};
	\node [wb] at (7,0) {};
	\node [wb] at (6,0) {};
	\node [wb] at (5,0) {};
	\node [bb] at (4,0) {};
	\node [wb] at (3,0) {};
	\node [wb] at (2,0) {};
	\node [wb] at (1,0) {};
	\node [bb] at (0,0) {};
	\node [wb] at (-1,0) {};
	\node [bb] at (-2,0) {};
	\node [bb] at (-3,0) {};
	\node [bb] at (-4,0) {};
	\node [bb] at (-5,0) {};
	\node [bb] at (-6,0) {};
	\node [bb] at (-7,0) {};
	\node [bb] at (-8,0) {};
	\node [bb] at (-9,0) {};
				\draw (-10,0) node[]{$\ldots$};

	\node [wb] at (11,1) {};
	\node [wb] at (10,1) {};
	\node [wb] at (9,1) {};
	\node [wb] at (8,1) {};
	\node [wb] at (7,1) {};
	\node [wb] at (6,1) {};
	\node [wb] at (5,1) {};
	\node [wb] at (4,1) {};
	\node [wb] at (3,1) {};
	\node [bb] at (2,1) {};
	\node [wb] at (1,1) {};
	\node [bb] at (0,1) {};
	\node [wb] at (-1,1) {};
	\node [bb] at (-2,1) {};
	\node [bb] at (-3,1) {};
	\node [bb] at (-4,1) {};
	\node [bb] at (-5,1) {};
	\node [bb] at (-6,1) {};
	\node [bb] at (-7,1) {};
	\node [bb] at (-8,1) {};
	\node [bb] at (-9,1) {};
				\draw (-10,1) node[]{$\ldots$};
					\draw[dashed](0.5,-0.5)--node[]{}(0.5,2.5);
			\draw (-10,1) node[]{$\ldots$};
					\draw[dashed](3.5,-0.5)--node[]{}(3.5,2.5);		
						\draw (-10,1) node[]{$\ldots$};
					\draw[dashed](-2.5,-0.5)--node[]{}(-2.5,2.5);	
						\draw[dashed](-5.5,-0.5)--node[]{}(-5.5,2.5);	
							\draw[dashed](6.5,-0.5)--node[]{}(6.5,2.5);
								\draw[dashed](9.5,-0.5)--node[]{}(9.5,2.5);
									\draw[dashed](-8.5,-0.5)--node[]{}(-8.5,2.5);
	\end{tikzpicture}
\end{center}
 which is the $l$-abacus  of $((4,1),(2,1))$, a bipartition obtained from $((3,1),(2,1))$ by adding a $0$-node. 
 \end{exa}

\subsection{Size}

The following result is implicit in Uglov \cite{U} and Yvonne's papers \cite{Y}. We give a formal proof  for the convenience of the reader as such a proof seems not to appear in the literature.\footnote{We thank C.Lecouvey for discussion on this proof.} There is a self-contained  and purely combinatorial (but more complicated) proof which is given in \cite{Jsize}. 
 
 \begin{Prop}\label{rank}
 Assume that $\lambda$ is a partition and  that $\tau^l (\lambda)=(\ulambda^l,{\bf s}^l)$ and $\tau_e (\lambda)=(\ulambda_e,{\bf s}_e)$. Assume that $\mu$ is a partition and that $\tau^l (\mu)=(\umu^l,{\bf s}^l)$ and $\tau_e (\mu)=(\umu_e,{\bf s}_e)$.
  Assume in addition that $|\umu^l|=|\ulambda^l|$. Then we have $|\umu_e|=|\ulambda_e|$.
  \end{Prop}
\begin{proof}

For $x=l$ or $x=e$ and $0\leq i \leq e-1$, we set $N_i (\unu,{\bf t})$ to be the number of $i$-nodes (associated with the multicharge ${\bf t}$) for the $x$-partition $\unu$. 
 Let $\lambda^{\circ}$ be the common $e$-core of $\lambda$ and $\mu$. 
By  \cite[Prop. 2.26]{JL}, we have:
$$|\lambda|-l |\lambda^{\circ}|=|\ulambda^l|+el N_0(\ulambda^l,{\bf s}^l),$$
and 
$$|\mu|-l |\lambda^{\circ}|=|\umu^l|+el N_0(\umu^l,{\bf s}^l).$$
So we have 
$$|\mu|- el N_0(\umu^l,{\bf s}^l)=|\lambda|- el N_0(\ulambda^l,{\bf s}^l).$$
Now, as $\lambda^{\circ}$ is the $e$-core of $\mu$  and of $\lambda$, we have:
$$|\mu|=e|\umu_e|+|\lambda^{\circ}|,$$
and 
$$|\lambda|=e|\ulambda_e|+|\lambda^{\circ}|.$$
So we obtain that:
$$e|\umu_e|- el N_0(\umu^l,{\bf s}^l)=e |\ulambda_e|- el N_0(\ulambda^l,{\bf s}^l).$$
So we need to show that $N_0(\umu^l,{\bf s}^l)=N_0(\ulambda^l,{\bf s}^l)$. 
By  \cite[(W1) p. 57]{Y} combined with  \cite[(W4) p. 57]{Y}, for all $0\leq i \leq e-1$, we have: 
$$N_i (\ulambda^l,{\bf s}^l)- N_i (\umu^l,{\bf s}^l)=N_0 (\ulambda_e,{\bf s}_e)- N_0 (\umu_e,{\bf s}_e).$$
So summing over all $0\leq i\leq e-1$ leads to :
$$|\umu^l|-|\ulambda^l|=e(N_0 (\ulambda_e,{\bf s}_e)- N_0 (\umu_e,{\bf s}_e)).$$
We deduce $N_0 (\ulambda_e,{\bf s}_e)=N_0 (\umu_e,{\bf s}_e)$.  By \cite[Proof of Prop. 3.24]{Y} (see also \cite[Lemma 3.4.1]{J}), we have that:
$$N_0 (\ulambda_e,{\bf s}_e)- N_0 (\umu_e,{\bf s}_e)=N_0 (\ulambda^l,{\bf s}^l)- N_0 (\umu^l,{\bf s}^l).$$
We thus have:
$N_0 (\ulambda^l,{\bf s}^l)=N_0 (\umu^l,{\bf s}^l)$,  as desired.

\end{proof}

\subsection{Generalized Core}
 We end this section by defining a notion of core for the $l$-partitions, as done in \cite{JL}.  
We fix  ${\bf s}^l\in \mathbb{Z}^l[m]$ and an associated $l$-partition $\ulambda^l$, 
 we here see how one can generalize the notion of core and quotient for this multipartition. To do this, we set:
     $$\overline{\mathcal{A}}^l_e:=\{ (s_0,\ldots,s_{l-1})\in \mathbb{Z}^l\ |\ \forall (i,j)\in \{0,\ldots,l-1\},\ i<j,\ 0\leq s_j-s_i\leq e\},$$
   $${\mathcal{A}}^l_e:=\{ (s_0,\ldots,s_{l-1})\in \mathbb{Z}^l\ |\ \forall (i,j)\in \{0,\ldots,l-1\},\ i<j,\ 0\leq s_j-s_i<e\}.$$


\begin{Def}\label{escore2} Let ${\bf s}^l\in{\mathcal{A}}^l_e$, 
we say that the  pair  $({\boldsymbol{\lambda}}^l,{\bf s}^l)$ is a an  $e$-core  if  this pair satisfies the following property.  Let 
 $(X^0,\ldots,X^{l-1})$ be its associated $l$-symbol then we have 
 $$X^0 \subset X^1 \subset \ldots X^{l-1} \subset  X^0[e].$$ 
\end{Def}
\begin{Rem}
In \cite{JL}, the definition is a little bit more general by considering arbitrary multicharges. 

\end{Rem}
Of course such an  $e$-core  is canonically associated with 
 the partition $(\tau^l)^{-1} (\ulambda^l,{\bf s}^l)$, but also with 
a multicharge ${\bf s}_e\in \mathbb{Z}^e [m]$
 which is defined as $\tau_e (\tau^l)^{-1} (\ulambda^l,{\bf s}^l)=(\uemptyset,{\bf s}_e)$.

 Now, we explain how one can obtain the core of an $l$-partition together with its multicharge. First, by \cite{JL}, one can assume that  
 $${\bf s}^l \in \mathcal{A}^l_e. $$
  We compute the core of the $l$-partition $\ulambda^l$  as follows.  
We consider the  $l$-abacus $(X^{0} ,\ldots,X^{l-1} )$ of $\ulambda^l$ associated with the multicharge ${\bf s}^l$. 
An {\it  elementary operation} on this abacus is defined as a move of one black bead from one runner of the abacus to another satisfying the following rule.
\begin{enumerate}
	\item If this black bead is not in the top runner, then we can do such an elementary operation on this black bead only if there is no black bead immediately above (that is in the same position on the runner just above). In this case, we slide 
	the black bead from its initial position, in a runner $i$,   to the runner $i+1$ located above in the same position.
	\item If this black bead is in the top runner in position $x$,  then we can do such an elementary operation only if there is no black bead in  position $x-e$ on the lowest runner.
\end{enumerate}
  At the end (that is, when we cannot perform any further operations, by construction, we obtain an $l$-abacus
  which is canonically associated with an $l$-symbol  
$$(L_{0},\ldots,L_{l-1}),$$ 
satisfying:
$$L_{0}\subset L_{1}\subset \ldots \subset L_{l-1}
\subset L_{0}[e].$$
this corresponds to an $l$-partition $\umu^l$ and a multicharge ${\bf v}^l\in \overline{\mathcal{A}}^l_e$  such that 
 $(\umu^l,{\bf v}^l)$ is an $e$-core. 
  This is called the $e$-core  of $(\ulambda^l,{\bf s}^l)$. By the discussion above, we see that 
   doing an elementary operation consists in removing a removable node in the associated $e$-partition.  Thus, if
  the number $w(\ulambda^l)$  which is number of elementary operations needed to reach this core is also the size of this $e$-partition and this is called the weight of $(\ulambda^l,{\bf s}^l)$ (this has been first defined by Fayers \cite{Fa}).   The $e$-partition  $\ulambda_e$ may be seen as a natural analogue of the  $e$-quotient 
   for $\ulambda^l$. The $e$-multicharge ${\bf s}_e$ is called {\it the $e$-core multicharge} of $(\ulambda^l,{\bf s}^l)$ (and recall that this is bijectively associated with the 
    $e$-core of  $(\ulambda^l,{\bf s}^l)$, as in the case $l=1$).

  As a consequence, a core is a partition (with its multicharge) with weight $0$.  We have the following proposition which give a nice unified properties of  cores:
  \begin{Prop}\label{core}
  Let  ${\bf s}^l \in \mathcal{A}^l_e$ and $\ulambda^l \in \Pi^l$  be such that $(\ulambda^l,{\bf s}^l$ is an $e$-core. Then   all $i\in \{0,\ldots,e-1\}$,  we are in one of the following two situation :
   \begin{itemize}
   \item $\ulambda^l $ admits  no addable $i$-node.
  \item $\ulambda^l $ admits  no removable $i$-node.
  \end{itemize}
  \end{Prop}
  \begin{proof}
  Take the $e$-abacus of $(\ulambda^l ,{\bf s}^l )$ 
   and assume that $\ulambda^l $ admits an  addable $i$-node. Assume that  it is on component $c\in \{0,\ldots,l-1\}$. 
    Thus, we have a black bead in position  $j\in \mathbb{Z}$ such that $j\equiv i-1+e\mathbb{Z}$ in component $c$ 
     and 
     a white bead in component $c$ in position  $j+1$. As a consequence,  the position $j$ in component 
     $k>c$  are occupied with beads.  This is also the case for the   positions  $j-te$ with $t>0$ in all components. On the other hand, 
      the position $j+1$  in component $k<c$  and the position  $j+1+te$ with $t>0$  in all the component cannot be occupied by beads (otherwise, this contradicts the definition of $e$-core). This implies that there are no removable $i$-nodes. 
       If $\ulambda^l $ admits a removable $i$-node, the proof is similar.

  \end{proof}

 \begin{exa}
 Tale $\ulambda=((3,1),(2,1)$ with $e=3$ and ${\bf s}=(0,0)$. We have
 
 
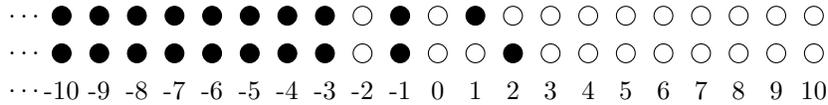
\begin{center}
\begin{tikzpicture}[scale=0.5, bb/.style={draw,circle,fill,minimum size=2.5mm,inner sep=0pt,outer sep=0pt}, wb/.style={draw,circle,fill=white,minimum size=2.5mm,inner sep=0pt,outer sep=0pt}]
	
	\node [] at (11,-1) {10};
	\node [] at (10,-1) {9};
	\node [] at (9,-1) {8};
	\node [] at (8,-1) {7};
	\node [] at (7,-1) {6};
	\node [] at (6,-1) {5};
	\node [] at (5,-1) {4};
	\node [] at (4,-1) {3};
	\node [] at (3,-1) {2};
	\node [] at (2,-1) {1};
	\node [] at (1,-1) {0};
	\node [] at (0,-1) {-1};
	\node [] at (-1,-1) {-2};
	\node [] at (-2,-1) {-3};
	\node [] at (-3,-1) {-4};
	\node [] at (-4,-1) {-5};
	\node [] at (-5,-1) {-6};
	\node [] at (-6,-1) {-7};
	\node [] at (-7,-1) {-8};
	\node [] at (-8,-1) {-9};
	\node [] at (-9,-1) {-10};
			\draw (-10,-1) node[]{$\ldots$};

	\node [wb] at (11,0) {};
	\node [wb] at (10,0) {};
	\node [wb] at (9,0) {};
	\node [wb] at (8,0) {};
	\node [wb] at (7,0) {};
	\node [wb] at (6,0) {};
	\node [wb] at (5,0) {};
	\node [wb] at (4,0) {};
	\node [bb] at (3,0) {};
	\node [wb] at (2,0) {};
	\node [wb] at (1,0) {};
	\node [bb] at (0,0) {};
	\node [wb] at (-1,0) {};
	\node [bb] at (-2,0) {};
	\node [bb] at (-3,0) {};
	\node [bb] at (-4,0) {};
	\node [bb] at (-5,0) {};
	\node [bb] at (-6,0) {};
	\node [bb] at (-7,0) {};
	\node [bb] at (-8,0) {};
	\node [bb] at (-9,0) {};
				\draw (-10,0) node[]{$\ldots$};

	\node [wb] at (11,1) {};
	\node [wb] at (10,1) {};
	\node [wb] at (9,1) {};
	\node [wb] at (8,1) {};
	\node [wb] at (7,1) {};
	\node [wb] at (6,1) {};
	\node [wb] at (5,1) {};
	\node [wb] at (4,1) {};
	\node [wb] at (3,1) {};
	\node [bb] at (2,1) {};
	\node [wb] at (1,1) {};
	\node [bb] at (0,1) {};
	\node [wb] at (-1,1) {};
	\node [bb] at (-2,1) {};
	\node [bb] at (-3,1) {};
	\node [bb] at (-4,1) {};
	\node [bb] at (-5,1) {};
	\node [bb] at (-6,1) {};
	\node [bb] at (-7,1) {};
	\node [bb] at (-8,1) {};
	\node [bb] at (-9,1) {};
				\draw (-10,1) node[]{$\ldots$};
	\end{tikzpicture}
\end{center}
We have $\ulambda_e= (\emptyset,(2),(1))$ thus the weight of $\ulambda$ is $3$  (the size of $\ulambda_e$) and the $e$-core multicharge is 
$(0,-1,1)$ with abacus:
 
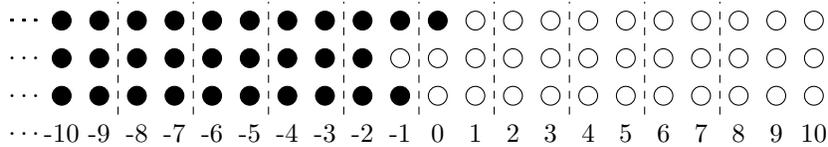
\begin{center}
\begin{tikzpicture}[scale=0.5, bb/.style={draw,circle,fill,minimum size=2.5mm,inner sep=0pt,outer sep=0pt}, wb/.style={draw,circle,fill=white,minimum size=2.5mm,inner sep=0pt,outer sep=0pt}]
	
	\node [] at (11,-1) {10};
	\node [] at (10,-1) {9};
	\node [] at (9,-1) {8};
	\node [] at (8,-1) {7};
	\node [] at (7,-1) {6};
	\node [] at (6,-1) {5};
	\node [] at (5,-1) {4};
	\node [] at (4,-1) {3};
	\node [] at (3,-1) {2};
	\node [] at (2,-1) {1};
	\node [] at (1,-1) {0};
	\node [] at (0,-1) {-1};
	\node [] at (-1,-1) {-2};
	\node [] at (-2,-1) {-3};
	\node [] at (-3,-1) {-4};
	\node [] at (-4,-1) {-5};
	\node [] at (-5,-1) {-6};
	\node [] at (-6,-1) {-7};
	\node [] at (-7,-1) {-8};
	\node [] at (-8,-1) {-9};
	\node [] at (-9,-1) {-10};
			\draw (-10,-1) node[]{$\ldots$};

	\node [wb] at (11,0) {};
	\node [wb] at (10,0) {};
	\node [wb] at (9,0) {};
	\node [wb] at (8,0) {};
	\node [wb] at (7,0) {};
	\node [wb] at (6,0) {};
	\node [wb] at (5,0) {};
	\node [wb] at (4,0) {};
	\node [wb] at (3,0) {};
	\node [wb] at (2,0) {};
	\node [wb] at (1,0) {};
	\node [bb] at (0,0) {};
	\node [bb] at (-1,0) {};
	\node [bb] at (-2,0) {};
	\node [bb] at (-3,0) {};
	\node [bb] at (-4,0) {};
	\node [bb] at (-5,0) {};
	\node [bb] at (-6,0) {};
	\node [bb] at (-7,0) {};
	\node [bb] at (-8,0) {};
	\node [bb] at (-9,0) {};
				\draw (-10,0) node[]{$\ldots$};

	\node [wb] at (11,1) {};
	\node [wb] at (10,1) {};
	\node [wb] at (9,1) {};
	\node [wb] at (8,1) {};
	\node [wb] at (7,1) {};
	\node [wb] at (6,1) {};
	\node [wb] at (5,1) {};
	\node [wb] at (4,1) {};
	\node [wb] at (3,1) {};
	\node [wb] at (2,1) {};
	\node [wb] at (1,1) {};
	\node [wb] at (0,1) {};
	\node [bb] at (-1,1) {};
	\node [bb] at (-2,1) {};
	\node [bb] at (-3,1) {};
	\node [bb] at (-4,1) {};
	\node [bb] at (-5,1) {};
	\node [bb] at (-6,1) {};
	\node [bb] at (-7,1) {};
	\node [bb] at (-8,1) {};
	\node [bb] at (-9,1) {};
				\draw (-10,1) node[]{$\ldots$};
	
	\node [wb] at (11,2) {};
	\node [wb] at (10,2) {};
	\node [wb] at (9,2) {};
	\node [wb] at (8,2) {};
	\node [wb] at (7,2) {};
	\node [wb] at (6,2) {};
	\node [wb] at (5,2) {};
	\node [wb] at (4,2) {};
	\node [wb] at (3,2) {};
	\node [wb] at (2,2) {};
	\node [bb] at (1,2) {};
	\node [bb] at (0,2) {};
	\node [bb] at (-1,2) {};
	\node [bb] at (-2,2) {};
	\node [bb] at (-3,2) {};
	\node [bb] at (-4,2) {};
	\node [bb] at (-5,2) {};
	\node [bb] at (-6,2) {};
	\node [bb] at (-7,2) {};
	\node [bb] at (-8,2) {};
	\node [bb] at (-9,2) {};
				\draw (-10,2) node[]{$\ldots$};
				\draw[dashed](-7.5,-0.5)--node[]{}(-7.5,2.5);\draw (-10,2) node[]{$\ldots$};
				\draw[dashed](-5.5,-0.5)--node[]{}(-5.5,2.5);\draw (-10,2) node[]{$\ldots$};
				\draw[dashed](-3.5,-0.5)--node[]{}(-3.5,2.5);\draw (-10,2) node[]{$\ldots$};
				\draw[dashed](-1.5,-0.5)--node[]{}(-1.5,2.5);\draw (-10,2) node[]{$\ldots$};
				\draw[dashed](0.5,-0.5)--node[]{}(0.5,2.5);\draw (-10,2) node[]{$\ldots$};
				\draw[dashed](2.5,-0.5)--node[]{}(2.5,2.5);\draw (-10,2) node[]{$\ldots$};
				\draw[dashed](4.5,-0.5)--node[]{}(4.5,2.5);\draw (-10,2) node[]{$\ldots$};
				\draw[dashed](6.5,-0.5)--node[]{}(6.5,2.5);\draw (-10,2) node[]{$\ldots$};				
				\draw[dashed](8.5,-0.5)--node[]{}(8.5,2.5);\draw (-10,2) node[]{$\ldots$};				
	\end{tikzpicture}
\end{center}
The $l$-abacus of the associated $e$-core multipartition is :

 
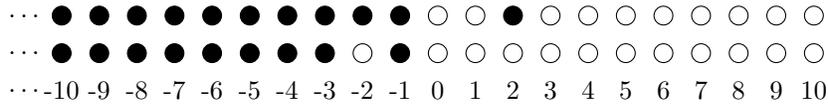
\begin{center}
\begin{tikzpicture}[scale=0.5, bb/.style={draw,circle,fill,minimum size=2.5mm,inner sep=0pt,outer sep=0pt}, wb/.style={draw,circle,fill=white,minimum size=2.5mm,inner sep=0pt,outer sep=0pt}]
	
	\node [] at (11,-1) {10};
	\node [] at (10,-1) {9};
	\node [] at (9,-1) {8};
	\node [] at (8,-1) {7};
	\node [] at (7,-1) {6};
	\node [] at (6,-1) {5};
	\node [] at (5,-1) {4};
	\node [] at (4,-1) {3};
	\node [] at (3,-1) {2};
	\node [] at (2,-1) {1};
	\node [] at (1,-1) {0};
	\node [] at (0,-1) {-1};
	\node [] at (-1,-1) {-2};
	\node [] at (-2,-1) {-3};
	\node [] at (-3,-1) {-4};
	\node [] at (-4,-1) {-5};
	\node [] at (-5,-1) {-6};
	\node [] at (-6,-1) {-7};
	\node [] at (-7,-1) {-8};
	\node [] at (-8,-1) {-9};
	\node [] at (-9,-1) {-10};
			\draw (-10,-1) node[]{$\ldots$};

	\node [wb] at (11,0) {};
	\node [wb] at (10,0) {};
	\node [wb] at (9,0) {};
	\node [wb] at (8,0) {};
	\node [wb] at (7,0) {};
	\node [wb] at (6,0) {};
	\node [wb] at (5,0) {};
	\node [wb] at (4,0) {};
	\node [wb] at (3,0) {};
	\node [wb] at (2,0) {};
	\node [wb] at (1,0) {};
	\node [bb] at (0,0) {};
	\node [wb] at (-1,0) {};
	\node [bb] at (-2,0) {};
	\node [bb] at (-3,0) {};
	\node [bb] at (-4,0) {};
	\node [bb] at (-5,0) {};
	\node [bb] at (-6,0) {};
	\node [bb] at (-7,0) {};
	\node [bb] at (-8,0) {};
	\node [bb] at (-9,0) {};
				\draw (-10,0) node[]{$\ldots$};

	\node [wb] at (11,1) {};
	\node [wb] at (10,1) {};
	\node [wb] at (9,1) {};
	\node [wb] at (8,1) {};
	\node [wb] at (7,1) {};
	\node [wb] at (6,1) {};
	\node [wb] at (5,1) {};
	\node [wb] at (4,1) {};
	\node [bb] at (3,1) {};
	\node [wb] at (2,1) {};
	\node [wb] at (1,1) {};
	\node [bb] at (0,1) {};
	\node [bb] at (-1,1) {};
	\node [bb] at (-2,1) {};
	\node [bb] at (-3,1) {};
	\node [bb] at (-4,1) {};
	\node [bb] at (-5,1) {};
	\node [bb] at (-6,1) {};
	\node [bb] at (-7,1) {};
	\node [bb] at (-8,1) {};
	\node [bb] at (-9,1) {};
				\draw (-10,1) node[]{$\ldots$};
	\end{tikzpicture}
\end{center}
The $e$-core multipartition is thus $((1),(2))$ with multicharge $(-1,1)$. 
 \end{exa}
\begin{exa}
The converse proposition is false even if $l=1$. Take $\lambda=(4)$ with $e=3$. This is not an $e$-core but we have one removable $0$-node with no removable $0$-node.  We have one addable $1$-node with no removable $1$-node.  We have one addable $2$-node but no removable $1$-node. 
 $$\ytableausetup{centertableaux}
\begin{ytableau}
0 & 1  & 2 & 0
\end{ytableau}$$
 
\end{exa}

\begin{Rem}\label{mvide}
Assume that the $e$-core multicharge ${\bf s}_e=(s_0,\ldots,s_{e-1})$ of a pair $(\ulambda^l,{\bf s}^l)$ satisfies:
$$l>s_0 \geq \ldots \geq s_{e-1} \geq 0$$
Then this implies that the $l$-abacus associated to $(\emptyset,{\bf s}_e)$ is such that for all bead in position $j$, we have a bead in position $j-1$. This implies that this is the $l$-abacus of the empty $l$-partition. As a consequence, the $e$-core of $(\ulambda^l,{\bf s}^l)$ is of the form $(\emptyset,{{\bf s}^l} ')$ for a certain 
 multicharge ${{\bf s}^l}'$. 
\end{Rem}

\section{Blocks of Ariki-Koike algebras}

In this section, we recall the main objects around  the modular representation theory of Ariki-Koike algebras. We refer to \cite{Ar,GJ} for details. Then we use our previous combinatorial definition to describe  the notion of blocks   for Ariki-Koike algebras as already stated in \cite{JL} and explore basic properties around them.

\subsection{Ariki-Koike algebras}

Let $n\in \mathbb{N}$, $l\in \mathbb{N}$, $e\in \mathbb{N}_{>1}$ and let ${\bf s}^l=(s_0,\ldots,s_{l-1})\in \mathbb{Z}^l$.  We set $\eta_e:=\operatorname{exp} (2i\pi/e)$.  The Ariki-Koike algebra
 $\mathcal{H}_n^{{\bf s}^l} (\eta)$ is the unital associative  $\mathbb{C}$-algebra with
 \begin{itemize}
\item   generators $T_1,\ldots,T_{n-1}$,
\item  relations: 
\begin{align*}
& T_0 T_1 T_0 T_1=T_1 T_0 T_1 T_0, \\
& T_iT_{i+1}T_i=T_{i+1}T_i T_{i+1}\ (i=1,...,n-2), \\
& T_i T_j =T_j T_i\ (|j-i|>1), \\
&(T_0-\eta^{s_0})(T_0-\eta^{s_2})...(T_0- \eta^{s_{l-1}}) = 0, \\
&(T_i-\eta)(T_i+1) = 0\ (i=1,...,n-1).
\end{align*}
\end{itemize}
 The representation theory of  $\mathcal{H}_n^{{\bf s}^l}  (\eta)$ 
is controlled by its decomposition matrix. 
For all $\ulambda^l\in \Pi^l$; one can associate a certain finite dimensional 
$\mathcal{H}_n^{{\bf s}^l} (\eta)$-module $S^{\ulambda}$ called a Specht module. For each $M\in \operatorname{Irr} (\mathcal{H}_n^{{\bf s}^l}(\eta))$, we have the composition factor 
$[S^{\ulambda^l}:M]$. The matrix:
$$\mathcal{D}:=([S^{\ulambda^l}:M])_{\ulambda^l \in \Pi^{l}(n),M\in \operatorname{Irr} (\mathcal{H}_n^{{\bf s}^l}(\eta))}$$
is  the {\it decomposition matrix}.  

\subsection{Parametrizing the simple modules} The study of the parametrization of the simple modules for Ariki-Koike algebras have a long story and there are different way to solve this. We here use the concept of basic sets, notion that we quickly recall.  There exists a natural pre-ordrer $\prec_{{\bf s}^l}$ on the set of $l$-partitions (see \cite{GJ}).  Then by \cite{GJ}, it can be shown that 
 there exists a  subset $\Phi^{{\bf s}^l} (n)\subset \Pi^{l} (n)$  and a bijective map 
 $$\mathcal{F} : \operatorname{Irr} (\mathcal{H}_n^{{\bf s}^l}(\eta)) \to \Phi^{{\bf s}^l} (n),$$
 such that 
for all  $M\in \operatorname{Irr} (\mathcal{H}_n^{{\bf s}^l}(\eta))$ we have 
$$[S^{ \mathcal{F} (M)}:M]=1 \text{ and } [S^{\umu}:M_i]\neq 0 \text{ only if }\umu \prec_{{\bf s}^l} \mathcal{F} (M).$$
The set $\Phi^{{\bf s}^l} (n) $ thus gives a natural indexation for the set of simple modules. These $l$-partitions are known as Uglov $l$-partitions.  If ${\bf s}^l \in \mathcal{A}_e^l$, then they are known as FLOTW $l$-partitions and they have an easy non recursive definition. This is not the case in general. However,  one can go from one parametrization to another thanks to an easy algorithm that we will describe later.

\subsection{Uglov $l$-partitions} Let us give a quick definition of the set of Uglov $l$-partitions.  For two nodes, we write  $\gamma<_{({\bf s}^l,e)}\gamma'$
 if we have  $b-a+s_c<b'-a'+s_{c'}\ \textrm{or } \textrm{if}\ b-a+s_c=b'-a'+s_{c'}\textrm{ and }c>c'.$ 
 
Let  ${{\ulambda}^l}$ be an  $l$-partition. We can consider its set of addable and
removable $i$-nodes. Let $w_{i}(\ulambda^l)$ be the word obtained first by writing the
addable and removable $i$-nodes of ${{\ulambda^l}}$ in {increasing}
order with respect to $\prec _{(e,{{{\bf s}^l}})}$ 
next by encoding each addable $i$-node by the letter $A$ and each removable $%
i$-node by the letter $R$.\ Write $\widetilde{w}_{i}(\ulambda^l,{\bf s}^l)=A^{p}R^{q}$ for the
word derived from $w_{i}$ by deleting as many  subwords of type $RA$ as
possible. $w_{i}(\ulambda^l,{\bf s}^l)$ is called the  {\it $i$-signature}  of $(\ulambda^l,{\bf s}^l)$ and 
 $\widetilde{w}_{i}(\ulambda^l)$ the  {\it reduced $i$-signature} of $(\ulambda^l,{\bf s}^l)$ . 
The addable $i$-nodes in  $\widetilde{w}_{i}(\ulambda^l,{\bf s}^l)$ are called the 
  {\it normal addable $i$-nodes}.  
The removable $i$-nodes in  $\widetilde{w}_{i}(\ulambda^l,{\bf s}^l)$ are called the 
  {\it normal removable $i$-nodes}. 
If $p>0,$ let $\gamma $ be the rightmost addable $i$-node in $\widetilde{w}_{i}(\ulambda^l,{\bf s}^l)$. The node $%
\gamma $ is called the  {\it good addable $i$-node}. If $q>0$, the leftmost removable $i$-node in 
$\widetilde{w}_{i}(\ulambda^l,{\bf s}^l)$ is called the  {\it good removable $i$-node}. Note that this notion depends on the order $<_{e,{\bf s}^l}$ and thus on the choice of ${\bf s}\in \Z^l$. 

The set $\Phi^{{\bf s}^l}$ of  {\it Uglov $l$-partitions}  is defined   recursively  as follows.
\begin{itemize}
   \item We have $\uemptyset:=(\emptyset,\emptyset,\ldots ,\emptyset)\in{ \Phi^{{\bf s}^l}}$.
    \item If $\ulambda^l\in\Phi^{{\bf s}^l}$ with $\ulambda^l \neq \uemptyset$, there exist $i\in{\{0,\ldots ,e-1\}}$ and a good removable $i$-node $\gamma$ such that if we remove $\gamma$ from  $\ulambda^l$, the resulting  $l$-partition is in $\Phi^{{\bf s}^l}$. We then denote  $\widetilde{f}_i .\ulambda^l =\umu^l$, or equivalently $\widetilde{e}_i .\umu^l =\ulambda^l$
\end{itemize}

\subsection{Affine symmetric groups}\label{algoiso}
Let $\widetilde{\mathfrak{S}}_r$ be the affine symmetric group.  This is the Coxeter group with a presentation by 
\begin{enumerate}
\item generators: $\sigma_i$, $i=0,\ldots,r-1$,
\item relations: for all indices $i$  and $j$ (which are read modulo $e$):
$$\sigma_i \sigma_{i+1} \sigma_i = \sigma_{i+1} \sigma_{i} \sigma_{i+1} ,$$
$$\sigma_i \sigma_j =\sigma_j \sigma_i\ (\text{if }i-j \neq 1+r\mathbb{Z}),$$
$$\sigma_i^2=1.$$
\end{enumerate}

The extended affine symmetric group $\widehat{\mathfrak{S}}_r$  is the semi-direct product $\widetilde{\mathfrak{S}_r}\rtimes \langle \tau \rangle$
 where $\langle \tau \rangle\simeq \mathbb{Z}$ where the product is defined by the relation $\tau \sigma_{i} =\sigma_{i+1} \tau$. 
Then  then $\widehat{\mathfrak{S}}_r$, is generated by $\tau$ and the 
$\sigma_i$ for $i=1,\ldots,r-1$.
We will now consider two types of action:
\begin{enumerate}
\item  We have that $\widehat{\mathfrak{S}}_e$, acts faithfully on $\mathbb{Z}^e$  (at the left) as follows: for any ${{{\bf s}}}_e=(s_{0},\ldots ,s_{e-1})\in 
\mathbb{Z}^{e}$, we have that%
$$\begin{array}{rcll}
\sigma _{c}.{{{\bf s}}}_e&=&(s_{0},\ldots ,s_{c-1},s_{c},\ldots ,s_{e-1})&\text{for }c=1,\ldots,l-1 \text{ and }\\
\tau . {{{\bf s}}}_e &=&(s_{e-1}+l,s_0,\ldots,s_{e-2})\end{array}.$$
Note that we have 
$$\sigma_0.  {\bf s}_e =(s_{e-1}+l,s_{1},\ldots, s_{e-2},s_{0}-l).$$
\item  We have that $\widehat{\mathfrak{S}}_l$ acts faithfully on $\mathbb{Z}^l$  (at the right) as follows: for any ${{{\bf s}}}^l=(s_{0},\ldots ,s_{l-1})\in 
\mathbb{Z}^{l}$, we have that%
$$\begin{array}{rcll}
{{{\bf s}}}^l.\sigma _{c}&=&(s_{0},\ldots ,s_{c-1},s_{c},\ldots ,s_{l})&\text{for }c=1,\ldots,l-1 \text{ and }\\
{{{\bf s}}}^l. \tau &=&(s_1,\ldots,s_{l-1},s_{0}+e)\end{array}.$$
Note that we have 
$$ {\bf s}^l. \sigma_0 =(s_{l-1}-e,s_{1},\ldots, s_{l-2},s_{0}+e).$$
\end{enumerate}

If ${\bf s}^l$ and ${\bf s'}^l$ are in the same orbit modulo the action of $\widehat{\mathfrak{S}}_l$, then the Ariki-Koike algebras 
 $\mathcal{H}_n^{{\bf s}^l}(\eta)$ and $\mathcal{H}_n^{{\bf s'}^l}(\eta)$ are isomorphic. This induces a 
 bijection: 
 $$\Psi^{{\bf s}^l \to {\bf s'}^l}:\Phi^{{\bf s}^l} \to \Phi^{{\bf s '}^l}.$$
In fact, this bijection is the restriction of a bijection:
 $$\Psi^{{\bf s}^l \to {\bf s'}^l}:\Pi^l  \to \Pi^l,$$
which  has been described in \cite{JL0} and may be described
just by describing two types of bijections :
 $$\Psi^{{\bf s}^l \to {\bf s}^l.\tau }\text{ and }\Psi^{{\bf s}^l \to {\bf s}^l.\sigma_i}.$$
\begin{Prop}[J-Lecouvey \cite{JL}]\label{tau}
For all ${\bf s}^l=(s_0,\ldots,s_{l-1})\in \mathbb{Z}^l$ and $\ulambda^l \in\Phi^{{\bf s}^l}$  we have 
$$\Psi^{{\bf s}^l \to  {\bf s}^l.\tau}  (\ulambda^l)=(\lambda^1,\ldots,\lambda^{l-1},\lambda^0).$$
\end{Prop}
Now assume that $i\in \{1,\ldots,l-1\}$, then we will describe $\Psi^{{\bf s}^l \to {\bf s}^l.\sigma_{i} }$ as it is stated in \cite{JL}. We have 
$$\Psi^{{\bf s}^l \to {\bf s}^l. \sigma_{i}} (\ulambda)=(\lambda^0,\ldots,\widetilde{\lambda}^{i},\widetilde{\lambda}^{i-1},\ldots,
\lambda^{l-1})$$
and we explain now how one can obtain $(\widetilde{\lambda}^{i},\widetilde{\lambda}^{i-1})$ 
 from   $({\lambda}^{i-1}, {\lambda}^{i})$.   To do this, consider the symbol associated 
  with $\ulambda^l$:
  $$(X^0,\ldots,X^{l-1})$$
The description of the bijections   
essentially rests on the following basic procedure on the pair $(X^{i-1} ,X^{i})$.
  Let $X^{i-1}=[,\ldots, x_{1},\ldots ,x_{r}]$ and 
$X^{i}=[\ldots, y_{1},\ldots ,y_{s}]$ where we assume that $x_k=y_k$ if $k<1$. 
 We
compute from $(X^{i-1} ,X^{i})$  a new pair $({X^{i-1}} ' ,{X^{i} }')$ with ${X^{i}} '=[\ldots,x_{-1},x_0,x_{1}^{\prime },\ldots ,x_{r}^{\prime }]$ and ${X^{i-1}}'=[\ldots,y_{-1},y_0,y_{1}^{\prime
},\ldots ,y_{s}^{\prime }]$ of such sequences by applying the following
algorithm :

\begin{itemize}
\item  Assume $r\geq s$. We associate to $y_{1}$ the integer $x_{i_{1}}\in X^{i-1} $
such that 
\begin{equation}
x_{i_{1}}=\left\{ 
\begin{array}{l}
\mathrm{min}\{x\in X^i \mid y_{1}\leq x\}\text{ if }y_{1}\geq x_{1} \\ 
x_{r}\text{ otherwise}
\end{array}
\right. .  \label{algo1}
\end{equation}
We say that $y_1$ and $x_{i_1}$ are in pairs. 
We repeat the same procedure to the ordered pair $(X^{i-1} \setminus
\{x_{i_{1}}\},X^{i}\setminus \{y_{1}\}).$ By induction this yields a subset $%
\{x_{i_{1}},\ldots ,x_{i_{s}}\}\subset X^{i-1} .\;$Then we define ${X^{i-1}} '$ as
the increasing reordering $\{x_{i_{1}},\ldots ,x_{i_{s}}\}$ and ${X^{i}} '$
as the increasing reordering of $X^{i-1} \setminus \{x_{i_{1}},\ldots
,x_{i_{s}}\}\sqcup X^{i}.$

\item  Assume $r<s.\;$We associate to $x_{1}$ the integer $y_{i_{1}}\in X^{i}$
such that 
\begin{equation}
y_{i_{1}}=\left\{ 
\begin{array}{l}
\mathrm{max}\{y\in X^{i}\mid x_{1}\geq y\}\text{ if }x_{1}\geq y_{1} \\ 
y_{1}\text{ otherwise}
\end{array}
\right. .  \label{algo2}
\end{equation}
We say that $x_1$ and $y_{i_1}$ are in pairs. 
We repeat the same procedure to the ordered sequences $X^{i-1} \setminus \{x_{1}\}$
and $X^{i} \setminus \{y_{i_{1}}\}$ and obtain a subset $\{y_{i_{1}},\ldots
,y_{i_{r}}\}\subset X^{i}.\;$Then we define ${X^{i}} '$ as the increasing
reordering $\{y_{i_{1}},\ldots ,y_{i_{r}}\}$ and ${X^{i-1}  }'$ as the
increasing reordering of $X^{i} \setminus \{y_{i_{1}},\ldots ,y_{i_{r}}\}\sqcup
U. $
\end{itemize}

\begin{Prop}[J-Lecouvey \cite{JL}]
Under the above notation, if $\ulambda^l \in \Phi^{{\bf s}^l}$, then we have $\umu^l=\Psi^{{\bf s}^l \to{\bf s}^l .\sigma_{i} } (\ulambda^l)$ 
 where the $el$-symbol of $\umu^l$ is $(X^0,\ldots, {X^{i}}',{X^{i-1}} ',\ldots, X^{l-1})$.
\end{Prop}

\begin{exa}
Take $\ulambda=  ((8, 2, 1),( 4, 3, 2 ))$ with ${\bf s}=(0,1)$ then the associates symbol is :
$$\left( \begin{array}{ccccccc}
                            \ldots &-5&-4&\textcolor{blue}{-3}&\textcolor{red}{0}&2&\textcolor{yellow}{4} \\                        
\ldots &-5 &  -4&\textcolor{red}{-2}&\textcolor{blue}{0}&\textcolor{yellow}{7 }
\end{array} \right)$$
So $-2$ in $X^0$  is in pair with $0$ in $X^1$, $0$ in $X^0$  is in pair with $-3$ in $X^1$ and $7$ in $X^0$ is in pair with $4$. 
 The new symbol is:
 $$\left( \begin{array}{ccccccc}
                            \ldots &-5&-4&-3&0&4 \\                        
\ldots &-5 &  -4&-2&0& 2 & 7 
\end{array} \right)$$
which is the symbol of $((7,3,2,1),(5,2))$. Thus we have 
$$\Psi^{(0,1) \to (1,0)} ((8, 2, 1),( 4, 3, 2 ))= ((7,3,2,1),(5,2))$$

\end{exa}

\subsection{Blocks} We fix a multicharge ${\bf s}^l$ and the associated Ariki-Koike algebra $\mathcal{H}_n^{{\bf s}^l}(\eta)$ . 
By definition, two $l$-partitions $\ulambda^l$ and $\umu^l$ lie in the {\it same  ordinary block} if 
there exists a sequence $(M_1,\ldots,M_r)$  of simple ${\mathbb{F}}\mathcal{H}^{{\bf s}}_n(\eta)$-modules and 
a sequence  of $l$-partitions $(\ulambda[1],\ldots,\ulambda[{r+1}])$ with $\ulambda[1]=\ulambda^l$, $\ulambda[{r+1}]=\umu^l$ and 
for all $i\in \{1,\ldots,r\}$, we have $[S^{\ulambda[i]}:M_i]\neq 0$ and 
$[S^{\ulambda[{i+1}]}:M_i]\neq 0$.  
The ordinary blocks are the thus equivalence classes under the above equivalence class. Thus, we have a partition of the set of $l$-partitions by the ordinary blocks:
$$\Pi^l (n)=\mathcal{B}_1 \sqcup \ldots \sqcup \mathcal{B}_r.$$
Similarly, we have a notion of {\it modular blocks} : two $l$-partitions in $\Phi^{{\bf s}^l} (n) $ are in the same ordinary block 
 if and only  there exist a sequence of 
 $(M_1,\ldots,M_r)$ of simple $\mathcal{H}_n^{{\bf s}}(\eta)$-modules
  with $\mathcal{F} (M_1)=\ulambda^l$ and $\mathcal{F} (M_r)=\umu^l$ 
  and 
a sequence  of $l$-partitions $(\ulambda[1],\ldots,\ulambda[{r+1}])$ $\ulambda[1]=\ulambda^l$, $\ulambda[{r+1}]=\umu^l$ such that we have $[S^{\ulambda[i]}:M_i]\neq 0$ and 
$[S^{\ulambda[{i+1}]}:M_i]\neq 0$.  
Hence a block may be think as a couple $(\mathcal{B},\mathfrak{B})$ where $\mathcal{B}$ is a set of $l$-partitions and 
$\mathfrak{B}$ is a set of  Uglov $l$-partitions. We have $\mathfrak{B}\subset \mathcal{B}$. So the datum of the ordinary block suffices to obtain the modular blocks.  
  We will now see in details how one can describe these blocks.  The following is the main Theorem of \cite{JL}.

 \begin{Th}
 Let ${\bf s}^l\in  \mathcal{A}^l_e$ then 
 two $l$-partitions  $\ulambda^l$ and $\umu^l$ are in the same block if and only if they have the same $e$-core.   Moreover, all the $l$-partitions in the same block have the same weight which is thus called the weight of the block. 
 \end{Th}

This theorem shows that one can also parameterize a block with its core and its weight. A block thus may be denoted $B ({\bf s}_e, w)$  where ${\bf s}_e \in \mathbb{Z}^e [m]$  denotes the $e$-core multicharge. We write 
$$B ({\bf s}_e, w)=(\mathcal{B},\mathfrak{B}),$$
if $\mathcal{B}$ is the associated ordinary block and $\mathfrak{B}$ is the associated modular block.

\begin{exa}
 Assume that $l=2$, ${\bf s}=(-1,1)$  and $e=4$.  The associated Uglov $2$-partitions are given by:
$$\Phi^{(0,1)}(4)=\{(\emptyset,(4)),((1),(2,1)),((1,1),(1,1)),((1),(3)),((1,1),(2)),((2),(1,1)),$$
$$ ((2),(2)),((2,1),(1)),((2,1,1),\emptyset),((2,2),\emptyset),((3),(1)),((3,1),\emptyset),((4),\emptyset)\}.$$
Thanks to \cite{J}, one can compute the associated decomposition matrix. It is given as follows:
$$
\begin{array}{c}
     \textcolor{red}{ ((4),\emptyset) }\\
     \textcolor{blue}{ ( ((3),(1)) } \\
       \textcolor{red}{  (\emptyset,(4))} \\
      \textcolor{red}{((3,1),\emptyset) } \\
   \textcolor{blue}{ (   ((2),(2)) } \\
       \textcolor{green}{  ((2,2),\emptyset) } \\
     \textcolor{red}{  ((1),(3))}  \\
   ((2,1),(1))  \\
     \textcolor{red}{  ((2,1,1),\emptyset) } \\
    ((2),(1,1))  \\
     \textcolor{red}{ ((1,1),(2)) } \\
    ((1),(2,1))  \\
     \textcolor{green}{  ((1,1),(1,1)) } \\
     \textcolor{red}{ (\emptyset,(3,1))}  \\
   \textcolor{red}{  ((1,1,1),(1))  }\\
    \textcolor{blue}{   (\emptyset,(2,2))  } \\
    \textcolor{red}{ ((1,1,1,1),\emptyset) } \\
    \textcolor{red}{ (\emptyset,(2,1,1)) } \\
    \textcolor{green}{   ((1),(1,1,1)) } \\
  \textcolor{red}{ (\emptyset,(1,1,1,1))  }
\end{array}
\left(
\begin{array}{cccccccccccccc}
  &1&.&.&.&.&.&.&.&.&.&.&.&.\\
  &.&1&.&.&.&.&.&.&.&.&.&.&.\\
  &.&.&1&.&.&.&.&.&.&.&.&.&.\\
  &1&.&.&1&.&.&.&.&.&.&.&.&.\\
  &.&1&.&.&1&.&.&.&.&.&.&.&.\\
  &.&.&.&.&.&1&.&.&.&.&.&.&.\\
 &1&.&1&.&.&.&1&.&.&.&.&.&.\\
 &.&.&.&.&.&.&.&1&.&.&.&.&.\\
  &.&.&.&1&.&.&.&.&1&.&.&.&.\\
  &.&.&.&.&.&.&.&.&.&1&.&.&.\\
  &1&.&.&1&.&.&1&.&.&.&1&.&.\\
  &.&.&.&.&.&.&.&.&.&.&.&1&.\\
  &.&.&.&.&.&1&.&.&.&.&.&.&1\\
  &.&.&1&.&.&.&1&.&.&.&.&.&.\\
  &.&.&.&1&.&.&.&.&1&.&1&.&.\\
  &.&.&.&.&1&.&.&.&.&.&.&.&.\\
  &.&.&.&.&.&.&.&.&1&.&.&.&.\\
  &.&.&.&.&.&.&1&.&.&.&1&.&.\\
  &.&.&.&.&.&.&.&.&.&.&.&.&1\\
  &.&.&.&.&.&.&.&.&.&.&1&.&.\\
\end{array}\right)$$
We see that we have the following ordinary blocks:
$$\mathcal{B}_1=\{ ((4),\emptyset), (3,1),\emptyset),((1),(3)), ((1,1),(2)), ((2,1,1),\emptyset),((1,1,1),(1)),((1,1,1,1),\emptyset),(\emptyset,(2,1,1)), $$
$$,(\emptyset,(1,1,1,1), (\emptyset,(4))),(\emptyset,(3,1))  \},$$
$$\mathcal{B}_2=\{((3),(1)), ((2),(2)), (\emptyset, (2,2))\},$$
$$\mathcal{B}_3=\{((2,2),\emptyset), ((1,1),(1,1)), ((1), (1,1,1))\},$$
$$\mathcal{B}_4=\{((1),(2,1)))\},$$
and thus the following modular blocks :
$$\mathcal{B}_1=\{ ((4),\emptyset), ((3,1),\emptyset),((1),(3)), ((1,1),(2)), ((2,1,1),\emptyset) (\emptyset,(4))) \},$$
$$\mathcal{B}_2=\{((3),(1)), ((2),(2))\},$$
$$\mathcal{B}_3=\{((2,2),\emptyset), ((1,1),(1,1))\},$$
$$\mathcal{B}_4=\{((1),(2,1)))\},$$
One can check that:
$$B((1,0,0,0),2))=(\mathcal{B}_1,\mathfrak{B}_1),
B( (0,1,1,-1), 1)=(\mathcal{B}_2,\mathfrak{B}_2),$$
$$B((2,1,-1,-1),1) )=(\mathcal{B}_3,\mathfrak{B}_3),
B((2,0,1,-1),0)=(\mathcal{B}_4,\mathfrak{B}_4).$$

\end{exa}

\section{Action of the affine Weyl group on blocks }

We now see how the affine symmetric group acts on the set of blocks and develop this action thanks to our parametrization of blocks. 


%

\subsection{Action on the set of $e$-cores}   
Assume that ${\bf s}^l\in \mathbb{Z}^l [m]$.  There is an  action of $\widehat{\mathfrak{S}}_e$ on the sets of blocks 
 $$\{ B({\bf s}_e,w)\ |\ {\bf s}_e \in \mathbb{Z}^e [m],\ w \in \mathbb{N} \}.$$
Let ${\bf s}_e \in \mathbb{Z}^e[m]$, then we define:
$$\sigma. B({\bf s}_e,w)=B(\sigma.{\bf s}_e,w).$$
 In the two following section, 
  we see how one can give two actions on the set of partitions and on the set of $l$-partitions which are compatible with this action.

\subsection{Action on ordinary blocks}

Let $\ulambda^l \in \Pi^l (n)$. Let $i\in \{0,1\ldots,e-1\}$. Let 
$$\{\gamma_1,\ldots,\gamma_k\},$$
be the set of addable $j$-nodes of $\ulambda^l$. Let 
$$\{\eta_1,\ldots,\eta_s\},$$
be the set of removable  $j$-nodes of $\ulambda^l$.  Then we define 
$$\sigma _i. \ulambda^l:=\umu^l,$$
where $$[\umu^l]=[\ulambda^l] \cup  \left\{\gamma_1,\ldots,\gamma_k\} \right\}\setminus\{\eta_1,\ldots,\eta_s\}.$$
One can easily perform the computation using the abacus configuration. 
 Let us consider the $e$-symbol ${\bf X}^{{\bf s}_e} (\ulambda_e)=(X^0,\ldots,X^{e-1})$. 
\begin{enumerate}
\item if $1\leq i\leq e-1$. Then we have  
$${\bf X}^{\sigma_i. {\bf s}_e} (\umu_e )=(X^0,\ldots,X^{i},X^{i-1},\ldots,X^{e-1})$$
\item if $i=0$ then we have:
$${\bf X}^{\sigma_i. {\bf s}_e} (\umu_e )=(X^{e-1}[l],\ldots,X^{1},\ldots, X^{e-2},\ldots,X^{0}[-l])$$
\end{enumerate}

\subsection{Action on modular blocks} 
 We now give an action of the set of Uglov $l$-partitions $\Phi^{{\bf s}^l}$. This  action is in fact the restriction of an action on the set of multipartitions. So one can take.   Let $\ulambda^l \in \Pi^l$. We give the definition of  $\sigma_i \star  \ulambda^l$. We set :
 $$\varepsilon_i (\ulambda^l)= \operatorname{max} (i\geq 0,\ \widetilde{f}_i^k.\ulambda^l \neq 0),$$
  $$\varphi_i (\ulambda^l)= \operatorname{max} (i\geq 0,\ \widetilde{e}_i^k.\ulambda^l \neq 0).$$
Then we set 
$$\sigma_i  \star  \ulambda^l=\left\{ \begin{array}{cc} \widetilde{e}_i^{\varphi (\ulambda^l)-\varepsilon (\ulambda^l)} \ulambda^l & \text{ if } \varphi (\ulambda^l)\geq \varepsilon (\ulambda^l), \\
 \widetilde{f}_i^{\varepsilon (\ulambda^l)-\varphi (\ulambda^l)} \ulambda^l & \text{ if } \varphi (\ulambda^l)\leq \varepsilon (\ulambda^l).\end{array}  \right.
$$
Contrary to the above case, it is not easy to see that the action is well defined  and this comes  in fact from a general result proved by Kashiwara \cite{Kashi}.

Consider the reduced $i$-signature of $\lambda^l$.  
$$\underbrace{A\ldots A}_{\varepsilon_i  (\ulambda^l)} \underbrace{R\ldots R}_{\varphi_i  (\ulambda^l)}$$
There exists $m\in \mathbb{N}$ such that   the number of black beads in runner $i-1$ is equal to $\varepsilon_i  (\ulambda^l)+m$
 and the number of black beads in runner $i$ is equal to $\varphi_i  (\ulambda^l)+m$ (the number $m$ corresponds to the number of occurrences ``RA'' we need to remove to reach the reduced signature from the signature). 
  If we look at the $i$-signature of $\sigma_i \star \ulambda^l$, we now have: 
$$\underbrace{A\ldots A}_{\varphi_i  (\ulambda^l)} \underbrace{R\ldots R}_{\varepsilon_i  (\ulambda^l)}$$
and the number of black beads in runner $i-1$ is  $\varphi_i  (\ulambda^l)+m$ where as the number of black beads in runner 
 $i$ is $\varepsilon_i  (\ulambda^l)+m$. We conclude that 
$$\sigma_i  (\mathcal{B},\mathcal{B})=(\sigma_i . \mathcal{B},\sigma_i \star\mathfrak{B}).$$
If $\ulambda^l \in \Phi^{{\bf s}^l}$ then $\sigma_i \star \ulambda^l \in \Phi^{{\bf s}^l}$.

It is important to note that the above action is not the same as in the last section, even if  we restrict the action to the set of  Uglov $l$-partitions, this is not even true for the set of $e$-regular partitions ! 
 We see in the action on the partition that an $e$-regular partition is not necessarily sent to an $e$-regular partition (for example 
  $\sigma_2 (3,2,1,1)=(2,2,2,1,1)$ for $p=3$).
 However, if $\ulambda^l$ is an Uglov $l$-partition then $\sigma_i  \star  \ulambda^l$ and $\sigma_i (\ulambda^l)$ have the same size. Thus, the weight of the two partitions are the same and this implies that the action on Uglov $l$-partitions also preserves the weight.

\subsection{Chuang-Rouquier equivalences}
 The above action has an interpretation in terms of the determination of blocks.  In \cite{CR},  Chuang and Rouquier have shown that 
  the action  of the affine symmetric group on the set of ordinary blocks induces an equivalence of derived categories between these blocks.  This equivalence is even a Morita equivalence in some cases. 
  Assume that we have 
 $$B ({\bf s}_e, w)=(\mathcal{B},\mathfrak{B})$$
 then we say that $B ({\bf s}_e, w)$ is an $i$-Scopes block if for all $\umu^l \in \mathcal{B}$, $\umu^l$ admits no addable  $i$-nodes. 
  If $B ({\bf s}_e, w)$  is an $i$-Scopes block then $B ({\bf s}_e, w)$ and $\sigma_i B ({\bf s}_e, w)$ are Morita-equivalent. 
  The following result is an easy extension of what happens in the case $l=1$. Such a result has been already studied by  Dell'Arciprete \cite{AD}, Lyle \cite{Ly}, Webster \cite{We} and recently by Li and Tan \cite{LiT} but without our notion of generalized core. 
  \begin{Prop}\label{scopes}  We fix a multicharge ${\bf s}^l \in \mathbb{Z}^l [m]$ and the associated Ariki-Koike algebra $\mathcal{H}_n^{{\bf s}^l}(\eta)$ .   Let ${\bf s}_e =(s_0,\ldots,s_{e-1})\in \mathbb{Z}^e [m]$ be such that $B ({\bf s}_e, w)=(\mathcal{B},\mathfrak{B})$  is a block of our Ariki-Koike algebra. Let $0\leq i\leq e-1$.
Assume that 
$$(*) \qquad s_i-s_{i-1} \geq \left\{ \begin{array}{rc} 
 w & \text{ if }i=1,\ldots,e-1 \\
 w+e & \text{ if }i=0 
 \end{array} \right.$$
then  for all $\umu^l \in \mathcal{B}$, $\umu^l$ admits no addable  $i$-nodes. 

\end{Prop}
\begin{proof}
 Assume that ${\bf s}_e =(s_0,\ldots,s_{e-1})$ satisfies the conditions in $(2)$.  Assume $i\neq 0$ and thus  that 
 $s_i-s_{i-1}\geq w$. If $\umu^l$ is in  $\mathcal{B}$, then the $l$-abacus of $\umu^l$ is obtained from the $e$-core $l$-partition associated to ${\bf s}_e$  by doing $w$ elementary moves. Because of the above condition, this implies that in the $l$-abacus of $\ulambda$, if we have a  bead in runner $i-1$ then we must have one in runner $i$. As a consequence, we do not have any addable $i$-node in $\umu^l$.
  If $s_0-s_{e-1}\geq w+e$, this is the same proof.

\end{proof}
\begin{Prop}\label{scopes2}
 Let ${\bf s}_e =(s_0,\ldots,s_{e-1})\in \mathbb{Z}^e [m]$  and let $0\leq i\leq e-1$. 
   Assume that for $w\in \mathbb{N}$,  condition $(*)$ is not satisfied  then there exists ${\bf s}^l\in  \mathbb{Z}^l [m]$,
    an  associated block  $B ({\bf s}_e, w)=(\mathcal{B},\mathfrak{B})$  for  $\mathcal{H}_n^{{\bf s}^l}(\eta)$ and $\umu^l\in \mathcal{B}$ such that $\umu^l$ has at least one addable $i$-node. 

\end{Prop}
\begin{proof}
Assume that $i\neq 0$. If  $s_i-s_{i-1}< w$ then  take the associated $e$-core $l$-partition. Take the associated $e$-abacus. Then in runner $i-1$, one can move the rightmost bead $w$ times at the right.  Taking the associated $l$-abacus, we see that this bead has no bead at its right, it is thus associated 
 with an addable $i$-node. 
\end{proof}

\begin{exa}
Take $l=2$ and $e=3$ with ${\bf s}_e=(7,0,2)$.  Write the abacus of the associated $e$-core $2$-partition:
 
 \begin{center}
\begin{tikzpicture}[scale=0.5, bb/.style={draw,circle,fill,minimum size=2.5mm,inner sep=0pt,outer sep=0pt}, rb/.style={draw,circle,fill=red,minimum size=2.5mm,inner sep=0pt,outer sep=0pt}, wb/.style={draw,circle,fill=white,minimum size=2.5mm,inner sep=0pt,outer sep=0pt}]

	\node [] at (11,-1) {10};
	\node [] at (10,-1) {9};
	\node [] at (9,-1) {8};
	\node [] at (8,-1) {7};
	\node [] at (7,-1) {6};
	\node [] at (6,-1) {5};
	\node [] at (5,-1) {4};
	\node [] at (4,-1) {3};
	\node [] at (3,-1) {2};
	\node [] at (2,-1) {1};
	\node [] at (1,-1) {0};
	\node [] at (0,-1) {-1};
	\node [] at (-1,-1) {-2};
	\node [] at (-2,-1) {-3};
	\node [] at (-3,-1) {-4};
	\node [] at (-4,-1) {-5};
	\node [] at (-5,-1) {-6};
	\node [] at (-6,-1) {-7};
	\node [] at (-7,-1) {-8};
	\node [] at (-8,-1) {-9};
	\node [] at (-9,-1) {-10};
			\draw (-10,-1) node[]{$\ldots$};

	\node [wb] at (11,0) {};
	\node [wb] at (10,0) {};
	\node [wb] at (9,0) {};
	\node [wb] at (8,0) {};
	\node [bb] at (7,0) {};
	\node [wb] at (6,0) {};
	\node [wb] at (5,0) {};
	\node [bb] at (4,0) {};
	\node [bb] at (3,0) {};
	\node [wb] at (2,0) {};
	\node [bb] at (1,0) {};
	\node [bb] at (0,0) {};
	\node [rb] at (-1,0) {};
	\node [bb] at (-2,0) {};
	\node [bb] at (-3,0) {};
	\node [bb] at (-4,0) {};
	\node [bb] at (-5,0) {};
	\node [bb] at (-6,0) {};
	\node [bb] at (-7,0) {};
	\node [bb] at (-8,0) {};
	\node [bb] at (-9,0) {};
				\draw (-10,0) node[]{$\ldots$};

	\node [wb] at (11,1) {};
	\node [bb] at (10,1) {};
	\node [wb] at (9,1) {};
	\node [wb] at (8,1) {};
	\node [bb] at (7,1) {};
	\node [wb] at (6,1) {};
	\node [wb] at (5,1) {};
	\node [bb] at (4,1) {};
	\node [bb] at (3,1) {};
	\node [wb] at (2,1) {};
	\node [bb] at (1,1) {};
	\node [bb] at (0,1) {};
	\node [bb] at (-1,1) {};
	\node [bb] at (-2,1) {};
	\node [bb] at (-3,1) {};
	\node [bb] at (-4,1) {};
	\node [bb] at (-5,1) {};
	\node [bb] at (-6,1) {};
	\node [bb] at (-7,1) {};
	\node [bb] at (-8,1) {};
	\node [bb] at (-9,1) {};
				\draw (-10,1) node[]{$\ldots$};
							\draw (-10,-1) node[]{$\ldots$};
						\draw[dashed](0.5,-0.5)--node[]{}(0.5,1.5);
			\draw (-10,1) node[]{$\ldots$};
					\draw[dashed](3.5,-0.5)--node[]{}(3.5,1.5);		
						\draw (-10,1) node[]{$\ldots$};
					\draw[dashed](-2.5,-0.5)--node[]{}(-2.5,1.5);	
						\draw[dashed](-5.5,-0.5)--node[]{}(-5.5,1.5);	
							\draw[dashed](6.5,-0.5)--node[]{}(6.5,1.5);
								\draw[dashed](9.5,-0.5)--node[]{}(9.5,1.5);
									\draw[dashed](-8.5,-0.5)--node[]{}(-8.5,1.5);
	\end{tikzpicture}
\end{center}
This is the $2$-partition $((3,1),(5,3,1))$ with multicharge ${\bf s}_l=(4,5)$. Take $i=2$ then $s_2-s_1=2$. If $w=3$ one can do $3$ elementary moves with the red bead to obtain the abacus : 
 \begin{center}
\begin{tikzpicture}[scale=0.5, bb/.style={draw,circle,fill,minimum size=2.5mm,inner sep=0pt,outer sep=0pt}, rb/.style={draw,circle,fill=red,minimum size=2.5mm,inner sep=0pt,outer sep=0pt}, wb/.style={draw,circle,fill=white,minimum size=2.5mm,inner sep=0pt,outer sep=0pt}]

	\node [] at (11,-1) {10};
	\node [] at (10,-1) {9};
	\node [] at (9,-1) {8};
	\node [] at (8,-1) {7};
	\node [] at (7,-1) {6};
	\node [] at (6,-1) {5};
	\node [] at (5,-1) {4};
	\node [] at (4,-1) {3};
	\node [] at (3,-1) {2};
	\node [] at (2,-1) {1};
	\node [] at (1,-1) {0};
	\node [] at (0,-1) {-1};
	\node [] at (-1,-1) {-2};
	\node [] at (-2,-1) {-3};
	\node [] at (-3,-1) {-4};
	\node [] at (-4,-1) {-5};
	\node [] at (-5,-1) {-6};
	\node [] at (-6,-1) {-7};
	\node [] at (-7,-1) {-8};
	\node [] at (-8,-1) {-9};
	\node [] at (-9,-1) {-10};
			\draw (-10,-1) node[]{$\ldots$};

	\node [wb] at (11,0) {};
	\node [wb] at (10,0) {};
	\node [wb] at (9,0) {};
	\node [wb] at (8,0) {};
	\node [bb] at (7,0) {};
	\node [wb] at (6,0) {};
	\node [wb] at (5,0) {};
	\node [bb] at (4,0) {};
	\node [bb] at (3,0) {};
	\node [wb] at (2,0) {};
	\node [bb] at (1,0) {};
	\node [bb] at (0,0) {};
	\node [wb] at (-1,0) {};
	\node [bb] at (-2,0) {};
	\node [bb] at (-3,0) {};
	\node [bb] at (-4,0) {};
	\node [bb] at (-5,0) {};
	\node [bb] at (-6,0) {};
	\node [bb] at (-7,0) {};
	\node [bb] at (-8,0) {};
	\node [bb] at (-9,0) {};
				\draw (-10,0) node[]{$\ldots$};

	\node [wb] at (11,1) {};
	\node [bb] at (10,1) {};
	\node [wb] at (9,1) {};
	\node [wb] at (8,1) {};
	\node [bb] at (7,1) {};
	\node [wb] at (6,1) {};
	\node [rb] at (5,1) {};
	\node [bb] at (4,1) {};
	\node [bb] at (3,1) {};
	\node [wb] at (2,1) {};
	\node [bb] at (1,1) {};
	\node [bb] at (0,1) {};
	\node [bb] at (-1,1) {};
	\node [bb] at (-2,1) {};
	\node [bb] at (-3,1) {};
	\node [bb] at (-4,1) {};
	\node [bb] at (-5,1) {};
	\node [bb] at (-6,1) {};
	\node [bb] at (-7,1) {};
	\node [bb] at (-8,1) {};
	\node [bb] at (-9,1) {};
				\draw (-10,1) node[]{$\ldots$};
							\draw (-10,-1) node[]{$\ldots$};
						\draw[dashed](0.5,-0.5)--node[]{}(0.5,1.5);
			\draw (-10,1) node[]{$\ldots$};
					\draw[dashed](3.5,-0.5)--node[]{}(3.5,1.5);		
						\draw (-10,1) node[]{$\ldots$};
					\draw[dashed](-2.5,-0.5)--node[]{}(-2.5,1.5);	
						\draw[dashed](-5.5,-0.5)--node[]{}(-5.5,1.5);	
							\draw[dashed](6.5,-0.5)--node[]{}(6.5,1.5);
								\draw[dashed](9.5,-0.5)--node[]{}(9.5,1.5);
									\draw[dashed](-8.5,-0.5)--node[]{}(-8.5,1.5);
	\end{tikzpicture}
\end{center}

The associated bipartition is $((4,2,2,1,1),(4,2,1,1,1)$ with multicharge $(3,6)$ and this partition has indeed one addable $2$-node. 
Take $i=0$. 
 \begin{center}
\begin{tikzpicture}[scale=0.5, bb/.style={draw,circle,fill,minimum size=2.5mm,inner sep=0pt,outer sep=0pt}, rb/.style={draw,circle,fill=red,minimum size=2.5mm,inner sep=0pt,outer sep=0pt}, wb/.style={draw,circle,fill=white,minimum size=2.5mm,inner sep=0pt,outer sep=0pt}]

	\node [] at (11,-1) {10};
	\node [] at (10,-1) {9};
	\node [] at (9,-1) {8};
	\node [] at (8,-1) {7};
	\node [] at (7,-1) {6};
	\node [] at (6,-1) {5};
	\node [] at (5,-1) {4};
	\node [] at (4,-1) {3};
	\node [] at (3,-1) {2};
	\node [] at (2,-1) {1};
	\node [] at (1,-1) {0};
	\node [] at (0,-1) {-1};
	\node [] at (-1,-1) {-2};
	\node [] at (-2,-1) {-3};
	\node [] at (-3,-1) {-4};
	\node [] at (-4,-1) {-5};
	\node [] at (-5,-1) {-6};
	\node [] at (-6,-1) {-7};
	\node [] at (-7,-1) {-8};
	\node [] at (-8,-1) {-9};
	\node [] at (-9,-1) {-10};
			\draw (-10,-1) node[]{$\ldots$};

	\node [wb] at (11,0) {};
	\node [wb] at (10,0) {};
	\node [wb] at (9,0) {};
	\node [wb] at (8,0) {};
	\node [bb] at (7,0) {};
	\node [wb] at (6,0) {};
	\node [wb] at (5,0) {};
	\node [bb] at (4,0) {};
	\node [rb] at (3,0) {};
	\node [wb] at (2,0) {};
	\node [bb] at (1,0) {};
	\node [bb] at (0,0) {};
	\node [bb] at (-1,0) {};
	\node [bb] at (-2,0) {};
	\node [bb] at (-3,0) {};
	\node [bb] at (-4,0) {};
	\node [bb] at (-5,0) {};
	\node [bb] at (-6,0) {};
	\node [bb] at (-7,0) {};
	\node [bb] at (-8,0) {};
	\node [bb] at (-9,0) {};
				\draw (-10,0) node[]{$\ldots$};

	\node [wb] at (11,1) {};
	\node [bb] at (10,1) {};
	\node [wb] at (9,1) {};
	\node [wb] at (8,1) {};
	\node [bb] at (7,1) {};
	\node [wb] at (6,1) {};
	\node [wb] at (5,1) {};
	\node [bb] at (4,1) {};
	\node [bb] at (3,1) {};
	\node [wb] at (2,1) {};
	\node [bb] at (1,1) {};
	\node [bb] at (0,1) {};
	\node [bb] at (-1,1) {};
	\node [bb] at (-2,1) {};
	\node [bb] at (-3,1) {};
	\node [bb] at (-4,1) {};
	\node [bb] at (-5,1) {};
	\node [bb] at (-6,1) {};
	\node [bb] at (-7,1) {};
	\node [bb] at (-8,1) {};
	\node [bb] at (-9,1) {};
				\draw (-10,1) node[]{$\ldots$};
							\draw (-10,-1) node[]{$\ldots$};
						\draw[dashed](0.5,-0.5)--node[]{}(0.5,1.5);
			\draw (-10,1) node[]{$\ldots$};
					\draw[dashed](3.5,-0.5)--node[]{}(3.5,1.5);		
						\draw (-10,1) node[]{$\ldots$};
					\draw[dashed](-2.5,-0.5)--node[]{}(-2.5,1.5);	
						\draw[dashed](-5.5,-0.5)--node[]{}(-5.5,1.5);	
							\draw[dashed](6.5,-0.5)--node[]{}(6.5,1.5);
								\draw[dashed](9.5,-0.5)--node[]{}(9.5,1.5);
									\draw[dashed](-8.5,-0.5)--node[]{}(-8.5,1.5);
	\end{tikzpicture}
\end{center}
If we perform $4$ elementary moves with the read bead we obtain the following abacus:
 \begin{center}
\begin{tikzpicture}[scale=0.5, bb/.style={draw,circle,fill,minimum size=2.5mm,inner sep=0pt,outer sep=0pt}, rb/.style={draw,circle,fill=red,minimum size=2.5mm,inner sep=0pt,outer sep=0pt}, wb/.style={draw,circle,fill=white,minimum size=2.5mm,inner sep=0pt,outer sep=0pt}]

	\node [] at (11,-1) {10};
	\node [] at (10,-1) {9};
	\node [] at (9,-1) {8};
	\node [] at (8,-1) {7};
	\node [] at (7,-1) {6};
	\node [] at (6,-1) {5};
	\node [] at (5,-1) {4};
	\node [] at (4,-1) {3};
	\node [] at (3,-1) {2};
	\node [] at (2,-1) {1};
	\node [] at (1,-1) {0};
	\node [] at (0,-1) {-1};
	\node [] at (-1,-1) {-2};
	\node [] at (-2,-1) {-3};
	\node [] at (-3,-1) {-4};
	\node [] at (-4,-1) {-5};
	\node [] at (-5,-1) {-6};
	\node [] at (-6,-1) {-7};
	\node [] at (-7,-1) {-8};
	\node [] at (-8,-1) {-9};
	\node [] at (-9,-1) {-10};
			\draw (-10,-1) node[]{$\ldots$};

	\node [wb] at (11,0) {};
	\node [wb] at (10,0) {};
	\node [rb] at (9,0) {};
	\node [wb] at (8,0) {};
	\node [bb] at (7,0) {};
	\node [wb] at (6,0) {};
	\node [wb] at (5,0) {};
	\node [bb] at (4,0) {};
	\node [wb] at (3,0) {};
	\node [wb] at (2,0) {};
	\node [bb] at (1,0) {};
	\node [bb] at (0,0) {};
	\node [bb] at (-1,0) {};
	\node [bb] at (-2,0) {};
	\node [bb] at (-3,0) {};
	\node [bb] at (-4,0) {};
	\node [bb] at (-5,0) {};
	\node [bb] at (-6,0) {};
	\node [bb] at (-7,0) {};
	\node [bb] at (-8,0) {};
	\node [bb] at (-9,0) {};
				\draw (-10,0) node[]{$\ldots$};

	\node [wb] at (11,1) {};
	\node [bb] at (10,1) {};
	\node [wb] at (9,1) {};
	\node [wb] at (8,1) {};
	\node [bb] at (7,1) {};
	\node [wb] at (6,1) {};
	\node [wb] at (5,1) {};
	\node [bb] at (4,1) {};
	\node [bb] at (3,1) {};
	\node [wb] at (2,1) {};
	\node [bb] at (1,1) {};
	\node [bb] at (0,1) {};
	\node [bb] at (-1,1) {};
	\node [bb] at (-2,1) {};
	\node [bb] at (-3,1) {};
	\node [bb] at (-4,1) {};
	\node [bb] at (-5,1) {};
	\node [bb] at (-6,1) {};
	\node [bb] at (-7,1) {};
	\node [bb] at (-8,1) {};
	\node [bb] at (-9,1) {};
				\draw (-10,1) node[]{$\ldots$};
							\draw (-10,-1) node[]{$\ldots$};
						\draw[dashed](0.5,-0.5)--node[]{}(0.5,1.5);
			\draw (-10,1) node[]{$\ldots$};
					\draw[dashed](3.5,-0.5)--node[]{}(3.5,1.5);		
						\draw (-10,1) node[]{$\ldots$};
					\draw[dashed](-2.5,-0.5)--node[]{}(-2.5,1.5);	
						\draw[dashed](-5.5,-0.5)--node[]{}(-5.5,1.5);	
							\draw[dashed](6.5,-0.5)--node[]{}(6.5,1.5);
								\draw[dashed](9.5,-0.5)--node[]{}(9.5,1.5);
									\draw[dashed](-8.5,-0.5)--node[]{}(-8.5,1.5);
	\end{tikzpicture}
\end{center}
This is the $2$-abacus of the bipartition $((5,4,2),(5,3,1,1))$  with multicharge $(4,5)$ and we indeed have one addable $0$-node.

\end{exa}

\begin{Rem}
Assume that $\mathcal{B}$  satisfies one of the above property $(*)$. Let 
$\mathfrak{B}$ be the associated modular block. Let $\ulambda^l \in\mathfrak{B}$ then   $\ulambda^l$ has no removable $i$-nodes.  So we have $\sigma_i \star \ulambda=\sigma_i \ulambda$.

\end{Rem}
\begin{Rem}
Proposition \ref{scopes2} is not the converse of Proposition \ref{scopes}. It can happen that there exists one  multicharge   and $i$-Scopes block of the associated Ariki-Koike algebra   which does not satisfied $(*)$.  

\end{Rem}

\subsection{Explicit description}
Now we want to give  an interpretation of the action on the modular blocks via the level-rank duality. In fact, such an interpretation can already be derived from  \cite[th. 2.19]{LG} in the case $l>1$ and $i\neq 0$  using the crystal basis theory. Here, we give a purely combinatorial and self-contained proof which is also available for $l=1$.  
\begin{Th}\label{duality}
Assume that $0\leq i\leq e-1$.  Let $\lambda \in \Pi$ and let 
$$\tau^l (\lambda)=(\ulambda^l,{\bf s}^l), {\tau}_e (\lambda)=({\ulambda}_e,{\bf s}_e).$$
Assume that $\ulambda^l\in \Pi^l$.  For $i=1,\ldots,e-1$, we have:
 $$(\tau^l)^{-1}(\sigma_i \star  \ulambda^l, {\bf s}^l) =
{\tau}_e^{-1} ( \Psi^{ {{\bf s}}_e \to  {\bf s}_e.\sigma_{i}} ({\ulambda}_e),  \sigma_{i}.{\bf s}_e ).$$

\end{Th}

\begin{proof}
Let
$${\bf X}^l:=(X^0,\ldots,X^{l-1})$$
be the $l$-symbol associated to $(\ulambda^l,{\bf s}^l)$, and consider the $e$-symbol associated to 
$({\ulambda}_e,{\bf s}_e)$:
$${\bf Y}=(Y^0,\ldots,Y^{e-1}).$$
Recall that one can go from one abacus to the other thanks to the procedure described in \S \ref{proc}.

Assume that $0\leq i\leq e-1$. 
We want to compute $\sigma_i \star  \ulambda^l$. So we need to look at the set of removable and addable $i$-nodes of $\ulambda^l$.  
 In ${\bf X}^l$,  the removable $i$-nodes correspond to elements $\alpha \in X^k$ (for all $0\leq k\leq l-1$) such that $\alpha \equiv i+e\mathbb{Z}$ and $\alpha-1\notin X^k$. The addable $i$-nodes 
  corresponds to elements $\alpha \in X^k$ (for all $0\leq k\leq e-1$) such that $\alpha+1 \equiv i+e\mathbb{Z}$ and $\alpha+1\notin X^k$.
  
Thus, one can also determine the addable and removable $i$-nodes of $\ulambda$ in $Y$. 
\begin{itemize}
\item If $i\neq 0$, the removable $i$-nodes correspond to elements in $\alpha \in X^i$ such that $\alpha \notin X^{i-1}$, 
 the addable $i$-nodes correspond to the elements $\alpha \in X^{i-1}$ such that $\alpha \notin X{i}$.

\item If $i=0$, the removable $i$-nodes correspond to elements in $\alpha \in X^0$ such that $\alpha-e \notin X^{e-1}$, 
 the addable $i$-nodes correspond to the elements $\alpha \in X^{e-1}$ such that $\alpha+e \notin X^{0}$. 
 \end{itemize}

Set ${\bf s}_e:=(t_1,\ldots,t_{e-1})$. Assume that $i\neq 0$.  Because both the action of $\sigma_i$ and 
$\Psi^{ {{\bf s}}_e \to {\bf s}_e.\sigma_i}$ are involutive, one can assume that $t_{i-1}\leq t_i$.  This means that we have more removable $i$-nodes than addable $i$-nodes. We perform the algorithm for the computation of 
 $\Psi^{ {{\bf s}}_e \to {\bf s}_e.\sigma_i} ({\ulambda}_e)$.  
 In the associated $e$-symbol, we consider the couple $(X^{i-1},X^i)$.   
  Take the $i$-signature  $w_{i}(\ulambda^l,{\bf s}^l)$ and read it from left to right.  Then one may easily 
   obtain the $i$-signature of 
  $(\tau^l)^{-1} \circ \tau_e (\Psi^{ {{\bf s}}_e \to  {\bf s}_e.\sigma_{i}} ({\ulambda}_e))$ 
   from   $w_{i}(\ulambda^l,{\bf s}^l)$ by applying our algorithm recursively, starting  with the leftmost $A$ in the signature:
   \begin{itemize}
   \item If we have a letter $R$ at the left of $A$ then, the letter $A$ is in pair with the rightmost $R$ which is at the left of $A$ and which is not already in a pair. 
   \item Otherwise, $A$ is in pair with the rightmost $R$ which is not in a pair
   \end{itemize}
    In this way, one can define pairs between some of the letters $R$ and $A$ exactly as in the algorithm.  All of the letters $A$ are in pairs with a letter $R$ whereas some of the letters $R$ are not. 
    We will say that a letter $R$ or $A$   is a marked   letter if it is in  pair with another letter.    By our algorithm in \S \ref{algoiso}, all the non  marked letters are associated 
     to elements of the symbol which will move from $X_i$ to $X_{i-1}$, that is the associated letter $R$ will change to a letter $A$. 
    
    So all our letter $A$ are marked.   Let us first assume that there are now occurence of type $RA$ in this word so that the $i$-signature is equal to the reduced 
    $i$-signature:
$$\underbrace{A\ldots A}_{\varepsilon_i  (\ulambda^l)} \underbrace{R\ldots R}_{\varphi_i  (\ulambda^l)}$$
We here have ${\varphi_i  (\ulambda^l)}\geq \varepsilon_i  (\ulambda^l)$. So by our algorithm for computing the bijections, 
the nodes associated to the letters $A$ are in pairs with the nodes associated to the $\varepsilon_i  (\ulambda^l)$ rightmost
   $R$. Then, if we want to compute $\Psi^{ {{\bf s}}_e \to  {\bf s}_e.\sigma_{i}} ({\ulambda}_e)$, we need to 
     change the non marked nodes  from removable to addable nodes (recall than we have assumed that we have more removable than addable nodes). The resulting signature is:
     $$\underbrace{A\ldots A}_{\varphi_i  (\ulambda^l)} \underbrace{R\ldots R}_{\varepsilon_i  (\ulambda^l)}$$
  which is what we wanted.  
  Now, assume that we have  an occurence of type $RA$ in the $i$-signature that we denote by $w$.
   If we remove this occurence (that is the two associated beads in the $e$-abacus) then the associated abacus and its associated word $w'$ satisfies our result, by induction.  Now if we add the two beads, two cases can occur:
   \begin{itemize}
\item The two letters associated to the beads are in   pairs in which case the result is trivial,
     \item Otherwise, the letter $R$ is in pair with another letter $A$. This letter $A$ was in pair with another letter $R$ in $w'$ which should thus be at the left of the one we have added. But then this letter must also be marked because we have a new letter $A$ at its right. 
  \end{itemize}
  As a conclusion,  We see that marked letters in $w$ are exactly the marked letters of $w'$ together with the two that we have added. 
     The result follows. 

\end{proof}

\begin{Rem}
For $i=0$,  we denote, for an $e$-partition $\ulambda^e$ and ${\bf s}_e \in \mathbb{Z}^e$, 
$(\ulambda^e)^{\ast}:=(\lambda^{e-1},\lambda^0,\ldots,\lambda^{e-2})$ , 
$(\ulambda^e)^{\triangle}:=(\lambda^{1},\ldots,\lambda^{e-1},\lambda^{0})$ , 
and ${\bf s}_e^{\ast}:=(s_{e-1}+l,s_0, \ldots,s_{e-2})$. It  follows directly using the case $i\neq 1$  together with the construction of the $e$-abacus associate to an $l$-partition that:
 $$ (\tau^l)^{-1}(\sigma_0 \star  \ulambda^l,{\bf s}^l ) =
{\tau}_e^{-1} ( (\Psi^{  {{\bf s}}_e^{\ast} \to ({\bf s}_e)^{\ast}.\sigma_1} ({\ulambda}_e^{\ast}))^\triangle , \sigma_0.  {\bf s}_e    ).$$
\end{Rem}

\begin{exa}
Assume that $e=l=3$ and let $\ulambda^l=((3,2),(3,2),(1))$ with ${\bf s}^l=(0,0,1)$. Then $\ulambda^l$ is an Uglov $l$-partition.  The Young diagram is as follows: 
$$(\ \ytableausetup{centertableaux}
\begin{ytableau}
0 & 1 & 2   \\ 
2& 0
\end{ytableau}, 
\begin{ytableau}
0& 1 & 2  \\ 
2& 0 
\end{ytableau}, \begin{ytableau}
1
\end{ytableau}\ )$$
The associated $l$-abacus is 
\begin{center}
\begin{tikzpicture}[scale=0.5, bb/.style={draw,circle,fill,minimum size=2.5mm,inner sep=0pt,outer sep=0pt}, wb/.style={draw,circle,fill=white,minimum size=2.5mm,inner sep=0pt,outer sep=0pt}]
	
	\node [] at (11,-1) {10};
	\node [] at (10,-1) {9};
	\node [] at (9,-1) {8};
	\node [] at (8,-1) {7};
	\node [] at (7,-1) {6};
	\node [] at (6,-1) {5};
	\node [] at (5,-1) {4};
	\node [] at (4,-1) {3};
	\node [] at (3,-1) {2};
	\node [] at (2,-1) {1};
	\node [] at (1,-1) {0};
	\node [] at (0,-1) {-1};
	\node [] at (-1,-1) {-2};
	\node [] at (-2,-1) {-3};
	\node [] at (-3,-1) {-4};
	\node [] at (-4,-1) {-5};
	\node [] at (-5,-1) {-6};
	\node [] at (-6,-1) {-7};
	\node [] at (-7,-1) {-8};
	\node [] at (-8,-1) {-9};
	\node [] at (-9,-1) {-10};
			\draw (-10,-1) node[]{$\ldots$};

	\node [wb] at (11,0) {};
	\node [wb] at (10,0) {};
	\node [wb] at (9,0) {};
	\node [wb] at (8,0) {};
	\node [wb] at (7,0) {};
	\node [wb] at (6,0) {};
	\node [wb] at (5,0) {};
	\node [wb] at (4,0) {};
	\node [bb] at (3,0) {};
	\node [wb] at (2,0) {};
	\node [bb] at (1,0) {};
	\node [wb] at (0,0) {};
	\node [wb] at (-1,0) {};
	\node [bb] at (-2,0) {};
	\node [bb] at (-3,0) {};
	\node [bb] at (-4,0) {};
	\node [bb] at (-5,0) {};
	\node [bb] at (-6,0) {};
	\node [bb] at (-7,0) {};
	\node [bb] at (-8,0) {};
	\node [bb] at (-9,0) {};
				\draw (-10,0) node[]{$\ldots$};

	\node [wb] at (11,1) {};
	\node [wb] at (10,1) {};
	\node [wb] at (9,1) {};
	\node [wb] at (8,1) {};
	\node [wb] at (7,1) {};
	\node [wb] at (6,1) {};
	\node [wb] at (5,1) {};
	\node [wb] at (4,1) {};
	\node [bb] at (3,1) {};
	\node [wb] at (2,1) {};
	\node [bb] at (1,1) {};
	\node [wb] at (0,1) {};
	\node [wb] at (-1,1) {};
	\node [bb] at (-2,1) {};
	\node [bb] at (-3,1) {};
	\node [bb] at (-4,1) {};
	\node [bb] at (-5,1) {};
	\node [bb] at (-6,1) {};
	\node [bb] at (-7,1) {};
	\node [bb] at (-8,1) {};
	\node [bb] at (-9,1) {};
				\draw (-10,1) node[]{$\ldots$};
	
	\node [wb] at (11,2) {};
	\node [wb] at (10,2) {};
	\node [wb] at (9,2) {};
	\node [wb] at (8,2) {};
	\node [wb] at (7,2) {};
	\node [wb] at (6,2) {};
	\node [wb] at (5,2) {};
	\node [wb] at (4,2) {};
	\node [wb] at (3,2) {};
	\node [bb] at (2,2) {};
	\node [wb] at (1,2) {};
	\node [bb] at (0,2) {};
	\node [bb] at (-1,2) {};
	\node [bb] at (-2,2) {};
	\node [bb] at (-3,2) {};
	\node [bb] at (-4,2) {};
	\node [bb] at (-5,2) {};
	\node [bb] at (-6,2) {};
	\node [bb] at (-7,2) {};
	\node [bb] at (-8,2) {};
	\node [bb] at (-9,2) {};
				\draw (-10,2) node[]{$\ldots$};
				\draw[dashed](-8.5,-0.5)--node[]{}(-8.5,2.5);\draw (-10,2) node[]{$\ldots$};
				\draw[dashed](-5.5,-0.5)--node[]{}(-5.5,2.5);\draw (-10,2) node[]{$\ldots$};
				\draw[dashed](-2.5,-0.5)--node[]{}(-2.5,2.5);\draw (-10,2) node[]{$\ldots$};
				\draw[dashed](0.5,-0.5)--node[]{}(0.5,2.5);\draw (-10,2) node[]{$\ldots$};
				\draw[dashed](3.5,-0.5)--node[]{}(3.5,2.5);\draw (-10,2) node[]{$\ldots$};
				\draw[dashed](6.5,-0.5)--node[]{}(6.5,2.5);\draw (-10,2) node[]{$\ldots$};
				\draw[dashed](9.5,-0.5)--node[]{}(9.5,2.5);\draw (-10,2) node[]{$\ldots$};				
	\end{tikzpicture}
\end{center}
The associated $e$-abacus is thus:
\begin{center}
\begin{tikzpicture}[scale=0.5, bb/.style={draw,circle,fill,minimum size=2.5mm,inner sep=0pt,outer sep=0pt}, wb/.style={draw,circle,fill=white,minimum size=2.5mm,inner sep=0pt,outer sep=0pt}]
	
	\node [] at (11,-1) {10};
	\node [] at (10,-1) {9};
	\node [] at (9,-1) {8};
	\node [] at (8,-1) {7};
	\node [] at (7,-1) {6};
	\node [] at (6,-1) {5};
	\node [] at (5,-1) {4};
	\node [] at (4,-1) {3};
	\node [] at (3,-1) {2};
	\node [] at (2,-1) {1};
	\node [] at (1,-1) {0};
	\node [] at (0,-1) {-1};
	\node [] at (-1,-1) {-2};
	\node [] at (-2,-1) {-3};
	\node [] at (-3,-1) {-4};
	\node [] at (-4,-1) {-5};
	\node [] at (-5,-1) {-6};
	\node [] at (-6,-1) {-7};
	\node [] at (-7,-1) {-8};
	\node [] at (-8,-1) {-9};
	\node [] at (-9,-1) {-10};
			\draw (-10,-1) node[]{$\ldots$};

	\node [wb] at (11,0) {};
	\node [wb] at (10,0) {};
	\node [wb] at (9,0) {};
	\node [wb] at (8,0) {};
	\node [wb] at (7,0) {};
	\node [wb] at (6,0) {};
	\node [wb] at (5,0) {};
	\node [wb] at (4,0) {};
	\node [bb] at (3,0) {};
	\node [bb] at (2,0) {};
	\node [wb] at (1,0) {};
	\node [bb] at (0,0) {};
	\node [bb] at (-1,0) {};
	\node [bb] at (-2,0) {};
	\node [bb] at (-3,0) {};
	\node [bb] at (-4,0) {};
	\node [bb] at (-5,0) {};
	\node [bb] at (-6,0) {};
	\node [bb] at (-7,0) {};
	\node [bb] at (-8,0) {};
	\node [bb] at (-9,0) {};
				\draw (-10,0) node[]{$\ldots$};

	\node [wb] at (11,1) {};
	\node [wb] at (10,1) {};
	\node [wb] at (9,1) {};
	\node [wb] at (8,1) {};
	\node [wb] at (7,1) {};
	\node [wb] at (6,1) {};
	\node [wb] at (5,1) {};
	\node [wb] at (4,1) {};
	\node [wb] at (3,1) {};
	\node [wb] at (2,1) {};
	\node [bb] at (1,1) {};
	\node [wb] at (0,1) {};
	\node [wb] at (-1,1) {};
	\node [bb] at (-2,1) {};
	\node [bb] at (-3,1) {};
	\node [bb] at (-4,1) {};
	\node [bb] at (-5,1) {};
	\node [bb] at (-6,1) {};
	\node [bb] at (-7,1) {};
	\node [bb] at (-8,1) {};
	\node [bb] at (-9,1) {};
				\draw (-10,1) node[]{$\ldots$};
	
	\node [wb] at (11,2) {};
	\node [wb] at (10,2) {};
	\node [wb] at (9,2) {};
	\node [wb] at (8,2) {};
	\node [wb] at (7,2) {};
	\node [wb] at (6,2) {};
	\node [wb] at (5,2) {};
	\node [wb] at (4,2) {};
	\node [bb] at (3,2) {};
	\node [bb] at (2,2) {};
	\node [wb] at (1,2) {};
	\node [wb] at (0,2) {};
	\node [wb] at (-1,2) {};
	\node [bb] at (-2,2) {};
	\node [bb] at (-3,2) {};
	\node [bb] at (-4,2) {};
	\node [bb] at (-5,2) {};
	\node [bb] at (-6,2) {};
	\node [bb] at (-7,2) {};
	\node [bb] at (-8,2) {};
	\node [bb] at (-9,2) {};
				\draw (-10,2) node[]{$\ldots$};
				\draw[dashed](-8.5,-0.5)--node[]{}(-8.5,2.5);\draw (-10,2) node[]{$\ldots$};
				\draw[dashed](-5.5,-0.5)--node[]{}(-5.5,2.5);\draw (-10,2) node[]{$\ldots$};
				\draw[dashed](-2.5,-0.5)--node[]{}(-2.5,2.5);\draw (-10,2) node[]{$\ldots$};
				\draw[dashed](0.5,-0.5)--node[]{}(0.5,2.5);\draw (-10,2) node[]{$\ldots$};
				\draw[dashed](3.5,-0.5)--node[]{}(3.5,2.5);\draw (-10,2) node[]{$\ldots$};
				\draw[dashed](6.5,-0.5)--node[]{}(6.5,2.5);\draw (-10,2) node[]{$\ldots$};
				\draw[dashed](9.5,-0.5)--node[]{}(9.5,2.5);\draw (-10,2) node[]{$\ldots$};				
	\end{tikzpicture}
\end{center}
We thus have $\ulambda_e=((1,1),(2),(3,3))$ and ${\bf s}_e=(2,-1,0)$. 
 We compute $\sigma_1 \star \ulambda^l$. We have:
$$w_{i}(\ulambda^l,{\bf s}^l) =AARAA$$
and thus 
$$w_{i}(\sigma \star \ulambda^l,{\bf s}^l)=RRRAR$$
We thus have $\sigma_1 \star \ulambda^l=((3,3,1),(3,2,1),(1))$. On the other hand, we have that the $e$-abacus of $\Psi^{ {{\bf s}}_e \to  {\bf s}_e.\sigma_{1}} ({\ulambda}_e)$ is 
\begin{center}
\begin{tikzpicture}[scale=0.5, bb/.style={draw,circle,fill,minimum size=2.5mm,inner sep=0pt,outer sep=0pt}, wb/.style={draw,circle,fill=white,minimum size=2.5mm,inner sep=0pt,outer sep=0pt}]
	
	\node [] at (11,-1) {10};
	\node [] at (10,-1) {9};
	\node [] at (9,-1) {8};
	\node [] at (8,-1) {7};
	\node [] at (7,-1) {6};
	\node [] at (6,-1) {5};
	\node [] at (5,-1) {4};
	\node [] at (4,-1) {3};
	\node [] at (3,-1) {2};
	\node [] at (2,-1) {1};
	\node [] at (1,-1) {0};
	\node [] at (0,-1) {-1};
	\node [] at (-1,-1) {-2};
	\node [] at (-2,-1) {-3};
	\node [] at (-3,-1) {-4};
	\node [] at (-4,-1) {-5};
	\node [] at (-5,-1) {-6};
	\node [] at (-6,-1) {-7};
	\node [] at (-7,-1) {-8};
	\node [] at (-8,-1) {-9};
	\node [] at (-9,-1) {-10};
			\draw (-10,-1) node[]{$\ldots$};

	\node [wb] at (11,0) {};
	\node [wb] at (10,0) {};
	\node [wb] at (9,0) {};
	\node [wb] at (8,0) {};
	\node [wb] at (7,0) {};
	\node [wb] at (6,0) {};
	\node [wb] at (5,0) {};
	\node [wb] at (4,0) {};
	\node [wb] at (3,0) {};
	\node [bb] at (2,0) {};
	\node [wb] at (1,0) {};
	\node [wb] at (0,0) {};
	\node [wb] at (-1,0) {};
	\node [bb] at (-2,0) {};
	\node [bb] at (-3,0) {};
	\node [bb] at (-4,0) {};
	\node [bb] at (-5,0) {};
	\node [bb] at (-6,0) {};
	\node [bb] at (-7,0) {};
	\node [bb] at (-8,0) {};
	\node [bb] at (-9,0) {};
				\draw (-10,0) node[]{$\ldots$};

	\node [wb] at (11,1) {};
	\node [wb] at (10,1) {};
	\node [wb] at (9,1) {};
	\node [wb] at (8,1) {};
	\node [wb] at (7,1) {};
	\node [wb] at (6,1) {};
	\node [wb] at (5,1) {};
	\node [wb] at (4,1) {};
	\node [bb] at (3,1) {};
	\node [wb] at (2,1) {};
	\node [bb] at (1,1) {};
	\node [bb] at (0,1) {};
	\node [bb] at (-1,1) {};
	\node [bb] at (-2,1) {};
	\node [bb] at (-3,1) {};
	\node [bb] at (-4,1) {};
	\node [bb] at (-5,1) {};
	\node [bb] at (-6,1) {};
	\node [bb] at (-7,1) {};
	\node [bb] at (-8,1) {};
	\node [bb] at (-9,1) {};
				\draw (-10,1) node[]{$\ldots$};
	
	\node [wb] at (11,2) {};
	\node [wb] at (10,2) {};
	\node [wb] at (9,2) {};
	\node [wb] at (8,2) {};
	\node [wb] at (7,2) {};
	\node [wb] at (6,2) {};
	\node [wb] at (5,2) {};
	\node [wb] at (4,2) {};
	\node [bb] at (3,2) {};
	\node [bb] at (2,2) {};
	\node [wb] at (1,2) {};
	\node [wb] at (0,2) {};
	\node [wb] at (-1,2) {};
	\node [bb] at (-2,2) {};
	\node [bb] at (-3,2) {};
	\node [bb] at (-4,2) {};
	\node [bb] at (-5,2) {};
	\node [bb] at (-6,2) {};
	\node [bb] at (-7,2) {};
	\node [bb] at (-8,2) {};
	\node [bb] at (-9,2) {};
				\draw (-10,2) node[]{$\ldots$};
				\draw[dashed](-8.5,-0.5)--node[]{}(-8.5,2.5);\draw (-10,2) node[]{$\ldots$};
				\draw[dashed](-5.5,-0.5)--node[]{}(-5.5,2.5);\draw (-10,2) node[]{$\ldots$};
				\draw[dashed](-2.5,-0.5)--node[]{}(-2.5,2.5);\draw (-10,2) node[]{$\ldots$};
				\draw[dashed](0.5,-0.5)--node[]{}(0.5,2.5);\draw (-10,2) node[]{$\ldots$};
				\draw[dashed](3.5,-0.5)--node[]{}(3.5,2.5);\draw (-10,2) node[]{$\ldots$};
				\draw[dashed](6.5,-0.5)--node[]{}(6.5,2.5);\draw (-10,2) node[]{$\ldots$};
				\draw[dashed](9.5,-0.5)--node[]{}(9.5,2.5);\draw (-10,2) node[]{$\ldots$};				
	\end{tikzpicture}
\end{center}
and the associated $l$-abacus is:
\begin{center}
\begin{tikzpicture}[scale=0.5, bb/.style={draw,circle,fill,minimum size=2.5mm,inner sep=0pt,outer sep=0pt}, wb/.style={draw,circle,fill=white,minimum size=2.5mm,inner sep=0pt,outer sep=0pt}]
	
	\node [] at (11,-1) {10};
	\node [] at (10,-1) {9};
	\node [] at (9,-1) {8};
	\node [] at (8,-1) {7};
	\node [] at (7,-1) {6};
	\node [] at (6,-1) {5};
	\node [] at (5,-1) {4};
	\node [] at (4,-1) {3};
	\node [] at (3,-1) {2};
	\node [] at (2,-1) {1};
	\node [] at (1,-1) {0};
	\node [] at (0,-1) {-1};
	\node [] at (-1,-1) {-2};
	\node [] at (-2,-1) {-3};
	\node [] at (-3,-1) {-4};
	\node [] at (-4,-1) {-5};
	\node [] at (-5,-1) {-6};
	\node [] at (-6,-1) {-7};
	\node [] at (-7,-1) {-8};
	\node [] at (-8,-1) {-9};
	\node [] at (-9,-1) {-10};
			\draw (-10,-1) node[]{$\ldots$};

	\node [wb] at (11,0) {};
	\node [wb] at (10,0) {};
	\node [wb] at (9,0) {};
	\node [wb] at (8,0) {};
	\node [wb] at (7,0) {};
	\node [wb] at (6,0) {};
	\node [wb] at (5,0) {};
	\node [wb] at (4,0) {};
	\node [bb] at (3,0) {};
	\node [bb] at (2,0) {};
	\node [wb] at (1,0) {};
	\node [wb] at (0,0) {};
	\node [bb] at (-1,0) {};
	\node [wb] at (-2,0) {};
	\node [bb] at (-3,0) {};
	\node [bb] at (-4,0) {};
	\node [bb] at (-5,0) {};
	\node [bb] at (-6,0) {};
	\node [bb] at (-7,0) {};
	\node [bb] at (-8,0) {};
	\node [bb] at (-9,0) {};
				\draw (-10,0) node[]{$\ldots$};

	\node [wb] at (11,1) {};
	\node [wb] at (10,1) {};
	\node [wb] at (9,1) {};
	\node [wb] at (8,1) {};
	\node [wb] at (7,1) {};
	\node [wb] at (6,1) {};
	\node [wb] at (5,1) {};
	\node [wb] at (4,1) {};
	\node [bb] at (3,1) {};
	\node [wb] at (2,1) {};
	\node [bb] at (1,1) {};
	\node [wb] at (0,1) {};
	\node [bb] at (-1,1) {};
	\node [wb] at (-2,1) {};
	\node [bb] at (-3,1) {};
	\node [bb] at (-4,1) {};
	\node [bb] at (-5,1) {};
	\node [bb] at (-6,1) {};
	\node [bb] at (-7,1) {};
	\node [bb] at (-8,1) {};
	\node [bb] at (-9,1) {};
				\draw (-10,1) node[]{$\ldots$};
	
	\node [wb] at (11,2) {};
	\node [wb] at (10,2) {};
	\node [wb] at (9,2) {};
	\node [wb] at (8,2) {};
	\node [wb] at (7,2) {};
	\node [wb] at (6,2) {};
	\node [wb] at (5,2) {};
	\node [wb] at (4,2) {};
	\node [wb] at (3,2) {};
	\node [bb] at (2,2) {};
	\node [wb] at (1,2) {};
	\node [bb] at (0,2) {};
	\node [bb] at (-1,2) {};
	\node [bb] at (-2,2) {};
	\node [bb] at (-3,2) {};
	\node [bb] at (-4,2) {};
	\node [bb] at (-5,2) {};
	\node [bb] at (-6,2) {};
	\node [bb] at (-7,2) {};
	\node [bb] at (-8,2) {};
	\node [bb] at (-9,2) {};
				\draw (-10,2) node[]{$\ldots$};
				\draw[dashed](-8.5,-0.5)--node[]{}(-8.5,2.5);\draw (-10,2) node[]{$\ldots$};
				\draw[dashed](-5.5,-0.5)--node[]{}(-5.5,2.5);\draw (-10,2) node[]{$\ldots$};
				\draw[dashed](-2.5,-0.5)--node[]{}(-2.5,2.5);\draw (-10,2) node[]{$\ldots$};
				\draw[dashed](0.5,-0.5)--node[]{}(0.5,2.5);\draw (-10,2) node[]{$\ldots$};
				\draw[dashed](3.5,-0.5)--node[]{}(3.5,2.5);\draw (-10,2) node[]{$\ldots$};
				\draw[dashed](6.5,-0.5)--node[]{}(6.5,2.5);\draw (-10,2) node[]{$\ldots$};
				\draw[dashed](9.5,-0.5)--node[]{}(9.5,2.5);\draw (-10,2) node[]{$\ldots$};				
	\end{tikzpicture}
\end{center}
which is the abacus that we wanted.
\end{exa}
Let us give a consequence of the above Theorem.  The above result shows that one can interpret the Chuang-Rouquier equivalence between two blocks of the same weight as a crystal isomorphism after the level-rank duality. 
 More precisely, after this duality,  the equivalence between two blocks labelled by the $e$-core multicharge ${\bf s}_e$  
  and $\sigma. {\bf s}_e$ can be interpreted via the crystal isomorphism between Fock spaces associated to the two multicharges ${\bf s}_e$     and $\sigma. {\bf s}_e$.  Quite remarkably, this result is also available in the case of the  Hecke algebra of type $A$ (and the symmetric group). This phenomenon can also be noted in the context of cyclotomic Cherednik algebras. In this case, 
   the crystal isomorphism may themselves be interpreted as wall crossing maps for Cherednik algebras \cite{L,JL3} and the action on blocks of the Cherednik algebras may be seen as  perverse equivalences between the two category $\mathcal{O}$ (see \cite{L} again) defined via level-rank duality. 
  
On the other hand, one can interpret the crystal isomorphism between the two Fock spaces associated with two multicharges ${\bf s}_e$     and $\sigma_i . {\bf s}_e$ 
 ($i=1,\ldots,l-1$) as the action on the associated $l$-partitions after an inverse level-rank duality. Note that this in particular show that this isomorphism do not depend on $e$ (which was already remarked in \cite{JL}).

\section{Orbits of blocks}

At the moment, all the main results for blocks of the symmetric groups seem to have a natural generalization  in the case of Ariki-Koike algebras. However, this is not really the case.  When $l=1$, the action is transitive.   When $l>1$, this is no more the case.
This explain the result found in \cite{LQ} that two blocks of the same weight are non necessarily derived equivalent.

\subsection{Orbits of blocks}  The first problem  is to understand when two blocks $B({\bf s}_e,w)$ and $B({\bf s}_e ',w')$ are in the same orbit.
 We first need to have $w=w'$ because the action preserves the weight. Denote by 
 $(\ulambda^l,{\bf s}^l)$ and $({\ulambda'}^l,{{\bf s}'}^l)$  the  $e$-cores associated to ${\bf s}_e$ and ${\bf s}_e '$, because the action of the affine symmetric group  $\widetilde{\mathfrak{S}}_e$ preserves the $l$-multicharge,  
we must have 
   ${\bf s}^l={{\bf s}'}^l$ . 
   
   Reciprocally,  assume that $(\ulambda^l,{\bf s}^l)$ and $({\ulambda'}^l,{{\bf s}}^l)$  are two $e$-cores.
    Denote by ${\bf s}_e$ and ${\bf s}_e ' $  the associated $e$-core multicharges. There exists $\sigma \in \widehat{\mathfrak{S}}_e$ such that 
      $\sigma. {\bf s}_e =(s_0,\ldots,s_{e-1})$ satisfies $l>s_0 \geq s_1 \geq \ldots \geq s_{e-1} \geq 0$.  This implies that 
       the $e$-core associated with the  $e$-core multicharge $\sigma. {\bf s}_e$  is $(\uemptyset,{\bf s}^l)$ by Remark \ref{mvide}.
      We can do the same for $({\ulambda'}^l,{\bf s}^l)$, 
     Thus both $(\ulambda^l,{\bf s}^l)$ and $({\ulambda'}^l,{{\bf s}}^l)$ have  $(\uemptyset,{\bf s}_l)$ in their orbits.    We thus proved:

 \begin{Prop}
 $B({\bf s}_e,w)$ and  $B({\bf s}_e',w)$  are in the same orbit if and only if the associated $l$-multicharges 
  of  ${\bf s}_e$ and ${\bf s}_e'$ are the same.
 
 \end{Prop}

So now, it is a natural question to ask which types of $l$-multicharges appear in the above proposition. 

\subsection{Classification of Orbits}
 We want to understand  which types of blocks appear for the Ariki-Koike algebra  $\mathcal{H}_n^{{\bf s}^l}(\eta)$ for all $n\geq 0$.  
  Thus by our previous discussion, 
Wwe want  to find all the multicharge  ${{\bf s} '}^l$ such that $(\ulambda^l,{{\bf s} '}^l)$ appears as a $e$-core  
   for the Ariki-Koike algebra $\mathcal{H}_n^{{\bf s}^l}(\eta)$, for all $n\geq 0$.  
  We claim that this is exactly $\overline{\mathcal{A}}^l_e[m]$.  First,  note that we already know that we indeed have ${{\bf s}'}^l\in \overline{\mathcal{A}}_e^l[m]$.
 Now  assume that we get ${\bf s}^l=(s_0,\ldots,s_{l-1})$ in  $\overline{\mathcal{A}}_e^l[m]$. 
  
We start with the empty $l$-partition together with the multicharge 
  ${\bf s}^l$. We need to show that there exists $\ulambda^l$ such that the $l$-multicharge associated to  the $e$-core multicharge 
   of $(\ulambda^l,{\bf s}^l)$ is precisely ${{\bf s}'}^l$. To do this,  we argue algorithmically. First, we start  with the $l$-abacus of $(\uemptyset,{\bf s}^l)$ and we move black beads in this abacus to obtain the desired abacus using the following algorithm. We start to study the runner $l-1$:  If $s_{l-1}=s_{l-1}'$, then we  consider runner $l-2$ and so on. If $s_{l-1}<s_{l-1}'$,
  Let  $j$ be maximal  such that $s_j>s_j '$ and $s_{j-1}\neq s_j$ (recall that $\sum_{0\leq i\leq l-1} s_i=\sum_{0\leq i\leq l-1} s_i'=m$). 
   We move the rightmost  black bead in runner $j$ to a position $s_{l-1}+1.$  Then, 
    we continue this process changing ${\bf s}^l$ with the $(s_0, \ldots,s_j-1,\ldots,s_{l-1}+1)$.
          
      Assume that $s_{l-1}>s_{l-1}'$ then there exists $k$ such that $s_{l-1}=s_{l-2}=\ldots=s_k>s_{k-1}$.  Note that $s_k>s_k '$. Moreover,  there exists $j$  such that $s_j<s_j'$. We choose $s_j$ maximal for this property.  Then we move the rightmost black bead of runner $k$ to the position $s_{j}+1+e$.  We then continue this process by considering the $l$-partition $(s_0,\ldots,s_j+1,\ldots,s_k-1,\ldots,s_{l-1})$. 
 At the end, by construction, the $l$-partition has the desired property. 
      
      Here is an example: let $l=5$ and $e=7$, we take ${\bf s}^l=(-3,-1,0,0,4)$ and ${{\bf s}^l }'=(-2,-2,0,1,3)$. We start with the $l$-abacus of $(\emptyset,{\bf s}^l)$:
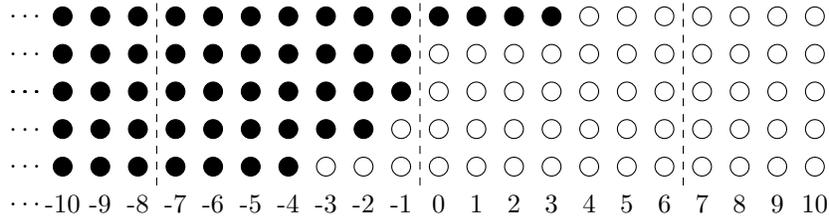
       \begin{center}
\begin{tikzpicture}[scale=0.5, bb/.style={draw,circle,fill,minimum size=2.5mm,inner sep=0pt,outer sep=0pt}, wb/.style={draw,circle,fill=white,minimum size=2.5mm,inner sep=0pt,outer sep=0pt}]
	
	\node [] at (11,-1) {10};
	\node [] at (10,-1) {9};
	\node [] at (9,-1) {8};
	\node [] at (8,-1) {7};
	\node [] at (7,-1) {6};
	\node [] at (6,-1) {5};
	\node [] at (5,-1) {4};
	\node [] at (4,-1) {3};
	\node [] at (3,-1) {2};
	\node [] at (2,-1) {1};
	\node [] at (1,-1) {0};
	\node [] at (0,-1) {-1};
	\node [] at (-1,-1) {-2};
	\node [] at (-2,-1) {-3};
	\node [] at (-3,-1) {-4};
	\node [] at (-4,-1) {-5};
	\node [] at (-5,-1) {-6};
	\node [] at (-6,-1) {-7};
	\node [] at (-7,-1) {-8};
	\node [] at (-8,-1) {-9};
	\node [] at (-9,-1) {-10};
			\draw (-10,-1) node[]{$\ldots$};

	\node [wb] at (11,0) {};
	\node [wb] at (10,0) {};
	\node [wb] at (9,0) {};
	\node [wb] at (8,0) {};
	\node [wb] at (7,0) {};
	\node [wb] at (6,0) {};
	\node [wb] at (5,0) {};
	\node [wb] at (4,0) {};
	\node [wb] at (3,0) {};
	\node [wb] at (2,0) {};
	\node [wb] at (1,0) {};
	\node [wb] at (0,0) {};
	\node [wb] at (-1,0) {};
	\node [wb] at (-2,0) {};
	\node [bb] at (-3,0) {};
	\node [bb] at (-4,0) {};
	\node [bb] at (-5,0) {};
	\node [bb] at (-6,0) {};
	\node [bb] at (-7,0) {};
	\node [bb] at (-8,0) {};
	\node [bb] at (-9,0) {};
				\draw (-10,0) node[]{$\ldots$};

	\node [wb] at (11,1) {};
	\node [wb] at (10,1) {};
	\node [wb] at (9,1) {};
	\node [wb] at (8,1) {};
	\node [wb] at (7,1) {};
	\node [wb] at (6,1) {};
	\node [wb] at (5,1) {};
	\node [wb] at (4,1) {};
	\node [wb] at (3,1) {};
	\node [wb] at (2,1) {};
	\node [wb] at (1,1) {};
	\node [wb] at (0,1) {};
	\node [bb] at (-1,1) {};
	\node [bb] at (-2,1) {};
	\node [bb] at (-3,1) {};
	\node [bb] at (-4,1) {};
	\node [bb] at (-5,1) {};
	\node [bb] at (-6,1) {};
	\node [bb] at (-7,1) {};
	\node [bb] at (-8,1) {};
	\node [bb] at (-9,1) {};
				\draw (-10,1) node[]{$\ldots$};
	
	\node [wb] at (11,2) {};
	\node [wb] at (10,2) {};
	\node [wb] at (9,2) {};
	\node [wb] at (8,2) {};
	\node [wb] at (7,2) {};
	\node [wb] at (6,2) {};
	\node [wb] at (5,2) {};
	\node [wb] at (4,2) {};
	\node [wb] at (3,2) {};
	\node [wb] at (2,2) {};
	\node [wb] at (1,2) {};
	\node [bb] at (0,2) {};
	\node [bb] at (-1,2) {};
	\node [bb] at (-2,2) {};
	\node [bb] at (-3,2) {};
	\node [bb] at (-4,2) {};
	\node [bb] at (-5,2) {};
	\node [bb] at (-6,2) {};
	\node [bb] at (-7,2) {};
	\node [bb] at (-8,2) {};
	\node [bb] at (-9,2) {};
				\draw (-10,2) node[]{$\ldots$};
				
					\node [wb] at (11,3) {};
	\node [wb] at (10,3) {};
	\node [wb] at (9,3) {};
	\node [wb] at (8,3) {};
	\node [wb] at (7,3) {};
	\node [wb] at (6,3) {};
	\node [wb] at (5,3) {};
	\node [wb] at (4,3) {};
	\node [wb] at (3,3) {};
	\node [wb] at (2,3) {};
	\node [wb] at (1,3) {};
	\node [bb] at (0,3) {};
	\node [bb] at (-1,3) {};
	\node [bb] at (-2,3) {};
	\node [bb] at (-3,3) {};
	\node [bb] at (-4,3) {};
	\node [bb] at (-5,3) {};
	\node [bb] at (-6,3) {};
	\node [bb] at (-7,3) {};
	\node [bb] at (-8,3) {};
	\node [bb] at (-9,3) {};
				\draw (-10,3) node[]{$\ldots$};
				
					\node [wb] at (11,4) {};
	\node [wb] at (10,4) {};
	\node [wb] at (9,4) {};
	\node [wb] at (8,4) {};
	\node [wb] at (7,4) {};
	\node [wb] at (6,4) {};
	\node [wb] at (5,4) {};
	\node [bb] at (4,4) {};
	\node [bb] at (3,4) {};
	\node [bb] at (2,4) {};
	\node [bb] at (1,4) {};
	\node [bb] at (0,4) {};
	\node [bb] at (-1,4) {};
	\node [bb] at (-2,4) {};
	\node [bb] at (-3,4) {};
	\node [bb] at (-4,4) {};
	\node [bb] at (-5,4) {};
	\node [bb] at (-6,4) {};
	\node [bb] at (-7,4) {};
	\node [bb] at (-8,4) {};
	\node [bb] at (-9,4) {};
				\draw (-10,4) node[]{$\ldots$};
				
				\draw[dashed](-6.5,-0.5)--node[]{}(-6.5,4.5);\draw (-10,2) node[]{$\ldots$};
				\draw[dashed](0.5,-0.5)--node[]{}(0.5,4.5);\draw (-10,2) node[]{$\ldots$};
				\draw[dashed](7.5,-0.5)--node[]{}(7.5,4.5);\draw (-10,2) node[]{$\ldots$};				
	\end{tikzpicture}
\end{center}
The algorithm  then gives the following abacus:

            
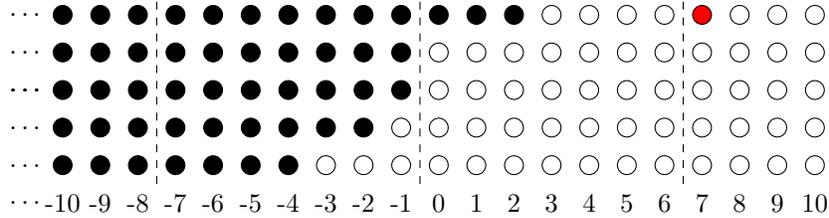
\begin{center}
\begin{tikzpicture}[scale=0.5, bb/.style={draw,circle,fill,minimum size=2.5mm,inner sep=0pt,outer sep=0pt}, rb/.style={draw,circle,fill=red,minimum size=2.5mm,inner sep=0pt,outer sep=0pt}, wb/.style={draw,circle,fill=white,minimum size=2.5mm,inner sep=0pt,outer sep=0pt}]
	
	\node [] at (11,-1) {10};
	\node [] at (10,-1) {9};
	\node [] at (9,-1) {8};
	\node [] at (8,-1) {7};
	\node [] at (7,-1) {6};
	\node [] at (6,-1) {5};
	\node [] at (5,-1) {4};
	\node [] at (4,-1) {3};
	\node [] at (3,-1) {2};
	\node [] at (2,-1) {1};
	\node [] at (1,-1) {0};
	\node [] at (0,-1) {-1};
	\node [] at (-1,-1) {-2};
	\node [] at (-2,-1) {-3};
	\node [] at (-3,-1) {-4};
	\node [] at (-4,-1) {-5};
	\node [] at (-5,-1) {-6};
	\node [] at (-6,-1) {-7};
	\node [] at (-7,-1) {-8};
	\node [] at (-8,-1) {-9};
	\node [] at (-9,-1) {-10};
			\draw (-10,-1) node[]{$\ldots$};

	\node [wb] at (11,0) {};
	\node [wb] at (10,0) {};
	\node [wb] at (9,0) {};
	\node [wb] at (8,0) {};
	\node [wb] at (7,0) {};
	\node [wb] at (6,0) {};
	\node [wb] at (5,0) {};
	\node [wb] at (4,0) {};
	\node [wb] at (3,0) {};
	\node [wb] at (2,0) {};
	\node [wb] at (1,0) {};
	\node [wb] at (0,0) {};
	\node [wb] at (-1,0) {};
	\node [wb] at (-2,0) {};
	\node [bb] at (-3,0) {};
	\node [bb] at (-4,0) {};
	\node [bb] at (-5,0) {};
	\node [bb] at (-6,0) {};
	\node [bb] at (-7,0) {};
	\node [bb] at (-8,0) {};
	\node [bb] at (-9,0) {};
				\draw (-10,0) node[]{$\ldots$};

	\node [wb] at (11,1) {};
	\node [wb] at (10,1) {};
	\node [wb] at (9,1) {};
	\node [wb] at (8,1) {};
	\node [wb] at (7,1) {};
	\node [wb] at (6,1) {};
	\node [wb] at (5,1) {};
	\node [wb] at (4,1) {};
	\node [wb] at (3,1) {};
	\node [wb] at (2,1) {};
	\node [wb] at (1,1) {};
	\node [wb] at (0,1) {};
	\node [bb] at (-1,1) {};
	\node [bb] at (-2,1) {};
	\node [bb] at (-3,1) {};
	\node [bb] at (-4,1) {};
	\node [bb] at (-5,1) {};
	\node [bb] at (-6,1) {};
	\node [bb] at (-7,1) {};
	\node [bb] at (-8,1) {};
	\node [bb] at (-9,1) {};
				\draw (-10,1) node[]{$\ldots$};
	
	\node [wb] at (11,2) {};
	\node [wb] at (10,2) {};
	\node [wb] at (9,2) {};
	\node [wb] at (8,2) {};
	\node [wb] at (7,2) {};
	\node [wb] at (6,2) {};
	\node [wb] at (5,2) {};
	\node [wb] at (4,2) {};
	\node [wb] at (3,2) {};
	\node [wb] at (2,2) {};
	\node [wb] at (1,2) {};
	\node [bb] at (0,2) {};
	\node [bb] at (-1,2) {};
	\node [bb] at (-2,2) {};
	\node [bb] at (-3,2) {};
	\node [bb] at (-4,2) {};
	\node [bb] at (-5,2) {};
	\node [bb] at (-6,2) {};
	\node [bb] at (-7,2) {};
	\node [bb] at (-8,2) {};
	\node [bb] at (-9,2) {};
				\draw (-10,2) node[]{$\ldots$};
				
					\node [wb] at (11,3) {};
	\node [wb] at (10,3) {};
	\node [wb] at (9,3) {};
	\node [wb] at (8,3) {};
	\node [wb] at (7,3) {};
	\node [wb] at (6,3) {};
	\node [wb] at (5,3) {};
	\node [wb] at (4,3) {};
	\node [wb] at (3,3) {};
	\node [wb] at (2,3) {};
	\node [wb] at (1,3) {};
	\node [bb] at (0,3) {};
	\node [bb] at (-1,3) {};
	\node [bb] at (-2,3) {};
	\node [bb] at (-3,3) {};
	\node [bb] at (-4,3) {};
	\node [bb] at (-5,3) {};
	\node [bb] at (-6,3) {};
	\node [bb] at (-7,3) {};
	\node [bb] at (-8,3) {};
	\node [bb] at (-9,3) {};
				\draw (-10,3) node[]{$\ldots$};
				
					\node [wb] at (11,4) {};
	\node [wb] at (10,4) {};
	\node [wb] at (9,4) {};
	\node [rb] at (8,4) {};
	\node [wb] at (7,4) {};
	\node [wb] at (6,4) {};
	\node [wb] at (5,4) {};
	\node [wb] at (4,4) {};
	\node [bb] at (3,4) {};
	\node [bb] at (2,4) {};
	\node [bb] at (1,4) {};
	\node [bb] at (0,4) {};
	\node [bb] at (-1,4) {};
	\node [bb] at (-2,4) {};
	\node [bb] at (-3,4) {};
	\node [bb] at (-4,4) {};
	\node [bb] at (-5,4) {};
	\node [bb] at (-6,4) {};
	\node [bb] at (-7,4) {};
	\node [bb] at (-8,4) {};
	\node [bb] at (-9,4) {};
				\draw (-10,4) node[]{$\ldots$};
				
				\draw[dashed](-6.5,-0.5)--node[]{}(-6.5,4.5);\draw (-10,2) node[]{$\ldots$};
				\draw[dashed](0.5,-0.5)--node[]{}(0.5,4.5);\draw (-10,2) node[]{$\ldots$};
				\draw[dashed](7.5,-0.5)--node[]{}(7.5,4.5);\draw (-10,2) node[]{$\ldots$};				
	\end{tikzpicture}
\end{center}
 The new multicharge is  ${\bf s}^l=(-3,-1,0,1,3)$

            
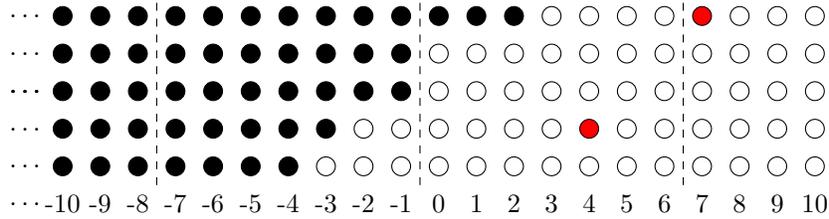
\begin{center}
\begin{tikzpicture}[scale=0.5, bb/.style={draw,circle,fill,minimum size=2.5mm,inner sep=0pt,outer sep=0pt}, rb/.style={draw,circle,fill=red,minimum size=2.5mm,inner sep=0pt,outer sep=0pt}, wb/.style={draw,circle,fill=white,minimum size=2.5mm,inner sep=0pt,outer sep=0pt}]
	
	\node [] at (11,-1) {10};
	\node [] at (10,-1) {9};
	\node [] at (9,-1) {8};
	\node [] at (8,-1) {7};
	\node [] at (7,-1) {6};
	\node [] at (6,-1) {5};
	\node [] at (5,-1) {4};
	\node [] at (4,-1) {3};
	\node [] at (3,-1) {2};
	\node [] at (2,-1) {1};
	\node [] at (1,-1) {0};
	\node [] at (0,-1) {-1};
	\node [] at (-1,-1) {-2};
	\node [] at (-2,-1) {-3};
	\node [] at (-3,-1) {-4};
	\node [] at (-4,-1) {-5};
	\node [] at (-5,-1) {-6};
	\node [] at (-6,-1) {-7};
	\node [] at (-7,-1) {-8};
	\node [] at (-8,-1) {-9};
	\node [] at (-9,-1) {-10};
			\draw (-10,-1) node[]{$\ldots$};

	\node [wb] at (11,0) {};
	\node [wb] at (10,0) {};
	\node [wb] at (9,0) {};
	\node [wb] at (8,0) {};
	\node [wb] at (7,0) {};
	\node [wb] at (6,0) {};
	\node [wb] at (5,0) {};
	\node [wb] at (4,0) {};
	\node [wb] at (3,0) {};
	\node [wb] at (2,0) {};
	\node [wb] at (1,0) {};
	\node [wb] at (0,0) {};
	\node [wb] at (-1,0) {};
	\node [wb] at (-2,0) {};
	\node [bb] at (-3,0) {};
	\node [bb] at (-4,0) {};
	\node [bb] at (-5,0) {};
	\node [bb] at (-6,0) {};
	\node [bb] at (-7,0) {};
	\node [bb] at (-8,0) {};
	\node [bb] at (-9,0) {};
				\draw (-10,0) node[]{$\ldots$};

	\node [wb] at (11,1) {};
	\node [wb] at (10,1) {};
	\node [wb] at (9,1) {};
	\node [wb] at (8,1) {};
	\node [wb] at (7,1) {};
	\node [wb] at (6,1) {};
	\node [rb] at (5,1) {};
	\node [wb] at (4,1) {};
	\node [wb] at (3,1) {};
	\node [wb] at (2,1) {};
	\node [wb] at (1,1) {};
	\node [wb] at (0,1) {};
	\node [wb] at (-1,1) {};
	\node [bb] at (-2,1) {};
	\node [bb] at (-3,1) {};
	\node [bb] at (-4,1) {};
	\node [bb] at (-5,1) {};
	\node [bb] at (-6,1) {};
	\node [bb] at (-7,1) {};
	\node [bb] at (-8,1) {};
	\node [bb] at (-9,1) {};
				\draw (-10,1) node[]{$\ldots$};
	
	\node [wb] at (11,2) {};
	\node [wb] at (10,2) {};
	\node [wb] at (9,2) {};
	\node [wb] at (8,2) {};
	\node [wb] at (7,2) {};
	\node [wb] at (6,2) {};
	\node [wb] at (5,2) {};
	\node [wb] at (4,2) {};
	\node [wb] at (3,2) {};
	\node [wb] at (2,2) {};
	\node [wb] at (1,2) {};
	\node [bb] at (0,2) {};
	\node [bb] at (-1,2) {};
	\node [bb] at (-2,2) {};
	\node [bb] at (-3,2) {};
	\node [bb] at (-4,2) {};
	\node [bb] at (-5,2) {};
	\node [bb] at (-6,2) {};
	\node [bb] at (-7,2) {};
	\node [bb] at (-8,2) {};
	\node [bb] at (-9,2) {};
				\draw (-10,2) node[]{$\ldots$};
				
					\node [wb] at (11,3) {};
	\node [wb] at (10,3) {};
	\node [wb] at (9,3) {};
	\node [wb] at (8,3) {};
	\node [wb] at (7,3) {};
	\node [wb] at (6,3) {};
	\node [wb] at (5,3) {};
	\node [wb] at (4,3) {};
	\node [wb] at (3,3) {};
	\node [wb] at (2,3) {};
	\node [wb] at (1,3) {};
	\node [bb] at (0,3) {};
	\node [bb] at (-1,3) {};
	\node [bb] at (-2,3) {};
	\node [bb] at (-3,3) {};
	\node [bb] at (-4,3) {};
	\node [bb] at (-5,3) {};
	\node [bb] at (-6,3) {};
	\node [bb] at (-7,3) {};
	\node [bb] at (-8,3) {};
	\node [bb] at (-9,3) {};
				\draw (-10,3) node[]{$\ldots$};
				
					\node [wb] at (11,4) {};
	\node [wb] at (10,4) {};
	\node [wb] at (9,4) {};
	\node [rb] at (8,4) {};
	\node [wb] at (7,4) {};
	\node [wb] at (6,4) {};
	\node [wb] at (5,4) {};
	\node [wb] at (4,4) {};
	\node [bb] at (3,4) {};
	\node [bb] at (2,4) {};
	\node [bb] at (1,4) {};
	\node [bb] at (0,4) {};
	\node [bb] at (-1,4) {};
	\node [bb] at (-2,4) {};
	\node [bb] at (-3,4) {};
	\node [bb] at (-4,4) {};
	\node [bb] at (-5,4) {};
	\node [bb] at (-6,4) {};
	\node [bb] at (-7,4) {};
	\node [bb] at (-8,4) {};
	\node [bb] at (-9,4) {};
				\draw (-10,4) node[]{$\ldots$};
				
				\draw[dashed](-6.5,-0.5)--node[]{}(-6.5,4.5);\draw (-10,2) node[]{$\ldots$};
				\draw[dashed](0.5,-0.5)--node[]{}(0.5,4.5);\draw (-10,2) node[]{$\ldots$};
				\draw[dashed](7.5,-0.5)--node[]{}(7.5,4.5);\draw (-10,2) node[]{$\ldots$};				
	\end{tikzpicture}
\end{center}
   The new multicharge is  ${\bf s}^l=(-2,-2,0,1,3)$, as desired.

\end{document}